\documentclass{amsart}
\usepackage{amsmath, amscd, amssymb}
\usepackage{caption}
\usepackage{subcaption}
\usepackage{diagrams}
\usepackage{cite}
\usepackage{color}
\usepackage{graphicx, psfrag, tikz, tikz-cd}
\usepackage[margin=1.1in]{geometry}

\usepackage{mathtools}

\usepackage{paralist}

\newcommand{\bC}{\mathbb{C}}
\newcommand{\bI}{\mathbb{I}}
\newcommand{\bP}{\mathbb{P}}
\newcommand{\bR}{\mathbb{R}}
\newcommand{\bZ}{\mathbb{Z}}

\newcommand{\bT}{\mathbb{T}}
\newcommand{\cM}{\mathcal{M}}
\newcommand{\cO}{\mathcal{O}}
\newcommand{\cF}{\mathcal F}

\newcommand{\Hom}{\mathrm{Hom}}
\newcommand{\coker}{\mathrm{coker}}
\newcommand{\Tot}{\mathrm{Tot}}
\newcommand{\Lie}{\mathrm{Lie}}

\newcommand{\fn}{\mathfrak{n}}
\newcommand{\ft}{\mathfrak{t}}

\newcommand{\tv}{\widetilde{v}}

\newcommand{\tit}{\widetilde{t}}

\newcommand{\tB}{\widetilde B}

\newcommand{\hA}{\widehat A}

\newcommand{\la}{\lambda}
\newcommand{\dd}{\partial}

\newcommand{\vin}{\rotatebox[origin=c]{90}{$\in$}}

\newtheorem{theorem}{Theorem}[section]
\newtheorem{lemma}[theorem]{Lemma} 

\newtheorem{corollary}[theorem]{Corollary}
\theoremstyle{definition} 
\newtheorem{definition}[theorem]{Definition}

\newtheorem{remark}[theorem]{Remark}
\newtheorem{example}[theorem]{Example}

\usepackage{hyperref}

\usepackage{xcolor}
\hypersetup{
    colorlinks,
    linkcolor={red!50!black},
    citecolor={blue!50!black},
    urlcolor={blue!80!black}
}

\begin{document}

\title{Action-angle and complex coordinates on toric manifolds}

\author{Haniya Azam, Catherine Cannizzo, Heather Lee}

 \email{Haniya Azam: haniya.azam@lums.edu.pk} 
 \email{Catherine Cannizzo: ccannizzo@scgp.stonybrook.edu}
 \email{Heather Lee: heathml@uw.edu}

\maketitle

\begin{abstract}
In this article, we provide an exposition about symplectic toric manifolds, which are symplectic manifolds $(M^{2n}, \omega)$ equipped with an effective Hamiltonian $\bT^n\cong (S^1)^n$-action.  We summarize the construction of $M$ as a symplectic quotient of $\bC^d$, the $\bT^n$-actions on $M$ and their moment maps, and Guillemin's K\"ahler potential on $M$.  While the theories presented in this paper are for compact toric manifolds, they do carry over for some noncompact examples as well, such as the canonical line bundle $K_M$, which is one of our main running examples, along with the complex projective space $\bP^n$ and its canonical bundle $K_{\bP^n}$.  One main topic explored in this article is how to write the moment map in terms of the complex homogeneous coordinates $z\in \bC^d$, or equivalently, the relationship between the action-angle coordinates and the complex toric coordinates.  We end with a brief review of homological mirror symmetry for toric geometries, where the main connection with the rest of the paper is that $K_M$ provides a prototypical class of examples of a Calabi-Yau toric manifold $Y$ which serves as the total space of a symplectic fibration $W: Y \to \bC$ with a singular fiber above $0$, known as a Landau-Ginzburg model in mirror symmetry.  Here we write $W$ in terms of the action-angle coordinates, which will prove to be useful in understanding the geometry of the fibration in our forthcoming work \cite{ACLL}.

\end{abstract}

\tableofcontents

\section{Introduction}

We start off with an overview of the two sets of coordinates, action-angle and complex, on compact toric manifolds.

Compact symplectic toric manifolds $(M^{2n}, \omega)$ modulo $\bT^n\cong (S^1)^n$ equivariant symplectomorphisms are in one-to-one correspondence with a class of compact convex polytopes known as Delzant polytopes, modulo translation \cite{Delzant}.  One direction of this correspondence is  Atiyah \cite{atiyah} and Guillemin-Sternberg's \cite{GS} convexity theorem that given a symplectic toric manifold $(M^{2n}, \omega, \mu)$, where $\mu: M^{2n} \to \mathbb R^n$ is the moment map of the Hamiltonian toric action, the image  $\mu(M^{2n})$ is a convex polytope.  Conversely, given a Delzant polytope ${\mathrm{\Delta}}^n$, Delzant's construction produces a compact symplectic toric manifold $M^{2n}_{{\mathrm{\Delta}}}$ such that the image of its moment map $\mu:M^{2n}_{\mathrm{\Delta}}\to \mathbb R^n$ is ${\mathrm{\Delta}}^n$.  The toric manifold $M^{2n}_{{\mathrm{\Delta}}}$ is obtained from a symplectic reduction of $\bC^d$ (where $d$ is the number of facets of ${\mathrm{\Delta}}^n$) with respect to the action of a $(d-n)$-dimensional subtorus $N \subset \bT^d$.  Denote by $\mu_N$ the moment map of this action by $N$.  The reduced space, $M^{2n}_{{\mathrm{\Delta}}} = \bC^d/\!\!/N = \mu_N^{-1}(a)/N$ (for regular values $a$), carries a Hamiltonian $\bT^n\cong \bT^d/N$-action and a canonical $\bT^n$-invariant symplectic form $\omega$.   Each $\bT^n$-orbit is a fiber of the moment map $\mu$, i.e. the preimage of a point in ${\mathrm{\Delta}}^n$ under $\mu$.  This $\bT^n$-action is free on the preimage  $\mu^{-1}(\mathring{{\mathrm{\Delta}}})$ of the interior $\mathring{{\mathrm{\Delta}}}$ of ${\mathrm{\Delta}}^n$, and it is degenerate on the complement, which is the preimage $\mu^{-1}(\partial {\mathrm{\Delta}}^n)$ of the boundary.  Therefore, $\mu^{-1}(\mathring{{\mathrm{\Delta}}})$  is diffeomorphic to $\mathring{{\mathrm{\Delta}}}\times \bT^n$.

 Since $\bC^d$ carries a $\bT^d$-invariant K\"ahler structure, it induces an $\omega$-compatible $\bT^n$-invariant K\"ahler structure $(g, J)$ on the reduced space $M^{2n}_{\mathrm{\Delta}}$.  The complex manifold $(M^{2n}_{\mathrm{\Delta}}, J)$ is then a toric variety with the complexified torus $\bT^n_\bC$-action.  This complex torus $\bT^n_\bC\cong (\bC^*)^n$ acts freely and transitively
on $\mu^{-1}(\mathring{{\mathrm{\Delta}}})$, hence $\mu^{-1}(\mathring{{\mathrm{\Delta}}})$ is diffeomorphic to $\bT^n_\bC$.   We now have two natural sets of coordinates on $\mu^{-1}(\mathring{{\mathrm{\Delta}}})\subset M^{2n}_{\mathrm{\Delta}}$ coming from the action-angle coordinates on $\mathring{{\mathrm{\Delta}}}\times \bT^n$ and the complex toric coordinates on $\bT^n_\bC$ via the identification
\begin{equation}
\mathring{{\mathrm{\Delta}}} \times \bT^n \cong \mu^{-1}(\mathring{{\mathrm{\Delta}}}) \cong \bT_\bC^{n}.
\end{equation}
These two sets of coordinates are related by Legendre transform; however, explicit formulas for switching between them might be complicated or sometimes impossible to obtain,  as we will illustrate with examples (most elaborately in Section \ref{subsec: KP1}). Understanding the relation between these two sets of coordinates was the initial motivation for this paper (see Section \ref{subsec:monodromy}) and is one of the main topics explored.

For the K\"ahler form $\omega$ on $M$, Guillemin \cite{G94} extended this story further by showing that the $\omega$ possesses a $\bT^n$-invariant K\"ahler potential  $F: \bT^n_\bC\to \bR$ such that $\omega=2 i \partial \overline\partial F$ and it only depends on the real part of the Lie algebra $\Lie(\bT^n_\bC)\cong \bC^n$ coordinates, so it is a function $F:\bR^n\to \bR$.  Guillemin also provided a dual potential $G:\mathring{{\mathrm{\Delta}}}\to \bR$, which is a function of the moment map coordinates and it is a Legendre transform of $F$.  The formulas for the potential functions are given by (up to adding constants) the combinatorial structure of ${\mathrm{\Delta}}^n$.  

Even though the theories above are all done for compact manifolds, the idea behind describing symplectic toric manifolds via symplectic reduction carries over to noncompact manifolds as well.  However, for noncompact symplectic toric manifolds, the moment map image might not be convex and the fibers of the moment map might not be connected; see \cite{KL15} for a classification result for noncompact symplectic manifolds.  In this paper, for noncompact examples, we will only focus on those where the theories for compact toric manifolds carry over, i.e. those with convex moment map images and connected moment map fibers.

In fact, when it comes to noncompact examples, we will almost exclusively focus on the total space of the canonical line bundle $\pi: K_M \to M$ of a compact toric manifold.  Because the tangent bundle $T_{K_M}\cong \pi^*T_M\oplus K_M$ and $c_1(K_M)=-c_1(T_M)$, we get that $c_1(T_{K_M})=0$, so $K_M$ is a Calabi-Yau (CY) toric manifold. 
In connection to mirror symmetry, the singular symplectic fibration $W: K_M\to \bC$, where $W$ is the product of the homogeneous coordinates, gives an example of a Landau-Ginzburg model $(K_M,W)$.  Furthermore, when $M$ is in addition Fano, the critical locus of the fibration $W:K_M\to \bC$ provides another source of Calabi-Yau manifolds that are of interest to many. This is behind our motivation for understanding the symplectic structure on $K_M$, and more generally, $(K_M, W)$ above serves as a prototypical class of examples of a toric Landau-Ginzburg model $(Y, W)$, where $Y$ is a CY toric manifold.  

Here is an outline of the paper. In Section \ref{sec: symp_quot}, we summarize the construction of compact toric manifolds $M^{2n}$ as symplectic quotients.  In Section \ref{sec: moment map}, we discuss the $\bT^n$-action on $M^{2n}$ and its moment map.  In Section \ref{sec: kahler},  we present Guillemin's \cite{G94} K\"ahler potentials.  We use many examples to illustrate the theories, with the complex projective space $\bP^n$ being the main running example throughout Sections 2--4.   Another running example throughout this paper is the aforementioned noncompact example of $K_M$.  (In Sections \ref{sec: symp_quot} and \ref{sec: moment map}, the $K_M$ example is contained in Sections \ref{subsec: KM}, \ref{sec: KM alpha}, \ref{sec: KM justification}, \ref{subsec: KP1}.) Finally in Section \ref{sec: HMS}, we explain some applications to mirror symmetry. In particular, we mention how to compute the monodromy of a fiber around the singularity of a symplectic fibration given by a superpotential on $K_M$, where the horizontal distribution is the $\omega$-orthogonal complement to the tangent space of each fiber; for this we use the symplectic form $\omega$ obtained via symplectic reduction. Section \ref{sec:notation} lists notation used in the text.

\subsection*{Acknowledgement} We thank Chiu-Chu Melissa Liu for her significant involvement in an earlier draft of this paper and for her continued support.  We would like to thank Ana Rita Pires for her encouragement and comments on  an earlier draft. We thank the referees for their helpful and detailed comments. We are also grateful to the 2019 Research Collaboration Conference for Women in Symplectic and Contact Geometry and Topology (WiSCon), during which this project was initiated, and the Institute for Computational and Experimental Research in Mathematics (ICERM) at Brown University for hosting this conference.   

\section{Toric manifolds as symplectic quotients}\label{sec: symp_quot}
In this section, we describe the compact symplectic toric manifold $M$ of complex dimension $n$, but we will defer the discussion about the toric $\bT^n$-action on $M$ to Section 3 and the K\"{a}hler structure to Section 4.    The description of $M$ doesn't start until Section \ref{subsec:compact_red}.  Below we give a brief summary of the content of each subsection.

Section \ref{sec: cx quotient} is entirely complex geometric, in which we set up some notation and describe a proper and holomorphic action of a complex $(d-n)$-dimensional torus $N_\bC= (\bC^*)^{d-n}$ on $\bC^d$. We then introduce an open set $U$ where $N_\bC$ acts freely, and construct the quotient $U/N_\bC$ with the induced complex manifold structure.  We take this complex geometric starting point because many examples we would like to consider, such as complex projective space, are most naturally described in this way.  

Section \ref{subsec:compact_red} gives the description of $M$ as a symplectic quotient of $\bC^d$ (with the standard symplectic structure) with respect to a Hamiltonian $N$-action, where $N=U(1)^{d-n}$ is a real $(d-n)$-dimensional torus.  More explicitly, $M$ can be described as a quotient $\mu^{-1}_N(a)/N$ where $a$ is a regular value of the moment map $\mu_N$ of the $N$-action.  This symplectic quotient $\mu^{-1}_N(a)/N$ can be identified with a complex geometric quotient $U_a/N_\bC$ for a certain Zariski open set $U_a\subset \bC^d$ on which $N_\bC$ acts freely. Now $U=U_a$ depends on which chamber $a$ is in, where going from one chamber to another amounts to crossing a wall of values where $\mu_N$ is not regular.

We use the complex project space $\bP^n$ as a running example to illustrate the set-up and theories in Section \ref{sec: cx quotient} and Section \ref{subsec:compact_red} (see Example \ref{ex: cx Pn}, Example \ref{ex: muN Pn}, and Example \ref{ex: Pn}).  In Section \ref{subsec: KM}, we provide an example of a noncompact toric manifold $K_M$, which is the canonical line bundle of a compact toric manifold $M$.

\subsection{Complex geometric quotients} \label{sec: cx quotient} Consider the action of the complex torus $\bT^d_\bC=(\bC^*)^d$ on $\bC^d$ given by
\begin{equation}\label{eq: Td action}
 \begin{array}{ccl}
 	\bT^d_\bC \times \bC^d & \to & \bC^d\\
       (\tilde{t},z) & \mapsto & \tit\cdot z  = (\tit_1 z_1,\ldots,\tit_d z_d),
  \end{array}
\end{equation}
where $\tit=(\tit_1,\ldots,\tit_d)\in \bT^d_\bC$ (to distinguish it from coordinates $t_i$ on the complex algebraic torus $\bT^{n = \dim_\bC M}_\bC$ in Chapter \ref{sec: moment map}) and $z=(z_1,\ldots, z_d) \in \bC^d$.

Let $N_\bC := (\bC^*)^{d-n}$ and fix any injective group homomorphism 
\begin{equation}\label{eq: rho_action}
\rho_\bC:  N_\bC=(\bC^*)^{d-n}  \longrightarrow \bT^d_\bC =(\bC^*)^d.
\end{equation}
Then the image $\rho_\bC(N_\bC)$ is a $(d-n)$-dimensional subtorus in $\bT_\bC^d$.
We then have a short exact sequence of abelian groups
\begin{equation}\label{eqn:NTT}
1\to N_\bC \stackrel{\rho_\bC}{\longrightarrow}\bT^d_\bC \stackrel{\beta_\bC}{\longrightarrow}\bT^n_\bC  \to 1,
\end{equation}
where $\bT^n_\bC :=\coker \rho_\bC= \bT^d_\bC/N_\bC \cong (\bC^*)^n$ and $\beta_\bC$ is any group homomorphism that makes the above sequence exact. Let's fix notations for the maps $\rho_{\bC}$ and $\beta_{\bC}$ below.
\begin{enumerate}
\item[(i)] The map $\rho_\bC$ is of the form 
\begin{equation} \label{eq: general rho}
\rho_\bC (\lambda_1,\ldots,\lambda_{d-n}) = \left(\prod_{\ell=1}^{d-n} \lambda_\ell^{Q^\ell_1},\ldots, \prod_{\ell=1}^{d-n} \lambda_\ell^{Q^\ell_d}\right)
\end{equation}
where $Q^\ell_k \in \bZ$, $1\leq \ell\leq d-n$, $1\leq k\leq d$.  
\item[(ii)] The map $\beta_\bC$  is of the form
\begin{equation}\label{eq: general beta}
\beta_\bC(\tit_1,\ldots, \tit_d) = \left(\prod_{k=1}^d \tit_k ^{v^k_1},\ldots, \prod_{k=1}^d \tit_k^{v^k_n}\right)
\end{equation} 
where $v^k_m \in \bZ$, $1\leq k \leq d$, $1\leq m\leq n$. 

\item[(iii)] The exactness of \eqref{eqn:NTT} implies the exponent on each $\la_\ell$ in the $m^{\textrm{th}}$ coordinate of $\beta_\bC \circ \rho_\bC$ is 0, namely
\begin{equation}\label{eqn:vQ} 
\sum_{k=1}^d  v^k_m Q^\ell_k =0, \quad 1\leq m\leq n, \quad 1\leq \ell\leq d-n. 
\end{equation}
\end{enumerate}
The choice of $\beta_\bC$ delimited by Equation \eqref{eqn:vQ} is not unique, and one can see this more easily by considering the linearized maps in matrix form.  Let  $\fn_\bC\cong \bC^{d-n}$, $\ft^d_\bC\cong \bC^d$, and $\ft_\bC^n\cong\bC^n$ be the Lie algebras of $N_\bC$, $\bT_\bC^d$, and $\bT_\bC^n$, respectively.
Taking the differentials of the two group homomorphisms in  \eqref{eqn:NTT} at the identity, we obtain the following short exact sequence of complex vector spaces: 
\begin{equation}\label{eqn:ntt}
0\to \fn_\bC  \stackrel{ (d \rho_\bC)_1 }{\longrightarrow}\ft^d_\bC \stackrel{ (d\beta_\bC)_1}{\longrightarrow}\ft^n_\bC \to 0.
\end{equation}
   The linear maps $(d\rho_{\bC})_1$ and $(d\beta_{\bC})_1$ are given by matrices $Q$ and $B$ respectively, where
\begin{equation}\label{eq: Q and B}
Q_{d\times (d-n)} = \left[ \begin{array}{lll}  Q^1_1 & \cdots & Q^{d-n}_1 \\
\vdots &  & \vdots\\
Q^1_d & \cdots & Q^{d-n}_d
 \end{array}\right],\quad
 B_{n\times d}= \left[\begin{array}{ccc} v^1_1 & \cdots & v^d_1 \\
\vdots &  & \vdots\\
v^1_n & \cdots & v^d_n
 \end{array}\right].
 \end{equation}
 Equation \eqref{eqn:vQ} is then equivalent to $BQ=0$. Given $Q$, the relation $BQ=0$ does not uniquely determine $B$ since it only requires that the row vectors of $B$ generate  $Q^\perp$. Given $Q$, to write $B: \bZ^d \to \bZ^d/\left<\mathrm{col}(Q) \right> \cong \bZ^n$ in matrix form we need to make an identification of $\bZ^d/\left<\mathrm{col}(Q) \right>$ with $\bZ^n$, where $\left<\mathrm{col}(Q) \right>$ denotes the column space of $Q$. Also because $B$ is surjective, $\left<\mathrm{col}(B)\right>=\bZ^n$ and the column vectors of $B$ are primitive.  Other choices for the matrix $B$ will hence differ by left multiplication of elements in $\mathrm{GL}(n,\bZ)$ with determinant $\pm 1$.  (On the other hand, if we fix $B$, then the choices for $Q$ would differ by right multiplication of elements in $\mathrm{GL}(d-n,\bZ)$ so as to keep $\ker B$ fixed).
 \begin{example}[$\bP(1,2,3)$] As an example, take $Q = [1,2,3]^T$, $d=3$, and $n=2$. First we make an identification of the quotient with $\bZ^{n=2}$. Extend the column vectors of $Q$ to a $\bZ$-basis for $\bZ^3$, given by 
 $$\{ f_1,f_2,f_3\}:=\{(1,2,3)^T,(0,1,0)^T ,(0,0,1)^T\},$$
 so that 
 $$\frac{\bZ^3}{\left<\mathrm{col}(Q) \right>}= \frac{\bZ f_1 \oplus \bZ f_2 \oplus \bZ f_3}{\bZ f_1} = \bZ f_2 \oplus \bZ f_3.$$ 
 Since $\ker B = \left<\mathrm{col}(Q) \right>$, let $B$ be the linear map $f_1 \mapsto (0,0)$, $f_2 \mapsto (1,0)$, and $f_3 \mapsto (0,1)$. Then using that 
 $$(1,0,0) = f_1-2f_2-3f_3 \implies B: (1,0,0) \mapsto (0,0)-2(1,0)-3(0,1)=(-2,-3),$$ 
 with respect to the standard bases on $\bZ^{d=3}$ on $\bZ^{n=2}$ we see that $B: \bZ^d \to \bZ^n$ is
 $$B = \left[\begin{matrix} -2 & 1 & 0\\ -3 & 0 & 1    \end{matrix} \right].$$
(This example leads to the weighted projective space $\bP(1,2,3)$, cf. the complex projective space $\bP^2=\bP(1,1,1)$ in Example \ref{ex: cx Pn}.)
 \qed
 \end{example}
  
Let $N_\bC= (\bC^*)^{d-n}$ act on $\bC^d$ by its image under $\rho_\bC: N_\bC \to \bT^d_\bC$, where the action of the subtorus $\rho_\bC(N_\bC)\subset \bT^d_\bC$ is inherited from the action of $\bT^d_\bC$.  In other words, this action is
 \begin{equation} \label{eq: Nc action}
 \begin{array}{ccl}
 N_\bC\times \bC^d & \to & \bC^d\\
(\lambda, z)& \mapsto & \lambda \cdot z := \rho_\bC(\lambda)\cdot z
\end{array}
\end{equation}
where $\lambda\in N_\bC$, $z\in \bC^d$, and $\rho_\bC(\lambda)\cdot z$ on the right hand side is given by  component-wise multiplication as in Equation \eqref{eq: Td action}. 

This $N_\bC$-action on $\bC^d$ is holomorphic and proper, but it is not free everywhere, so simply taking the quotient $\bC^d/N_\bC$ will not give us a manifold structure.  If $U$ is a subset of $\bC^d$ on which $N_\bC$ acts freely, then the quotient $U/N_\bC$ has a unique complex structure such that the quotient map $\pi: U\to U/N_\bC$
is holomorphic.  Many complex manifolds can be described in this way as a quotient, such as Example \ref{ex: cx Pn} for the complex projective space below.

\begin{example} [$\bP^n$]\label{ex: cx Pn} Consider $d=n+1$ and $N_\bC= (\bC^*)^{d-n=1}=\bC^*$.   The image of the embedding
\begin{equation}\label{eq: rho Pn}
\rho_\bC: N_\bC=\bC^* \to  \bT^d_\bC=(\bC^*)^{n+1}, \quad \rho_\bC(\lambda_1)=(\lambda_1, \ldots, \lambda_1)
\end{equation}
is a subtorus $\rho_\bC(N_\bC)\cong \bC^*$ in $\bT^d_\bC$, which gives the following $N_\bC=\bC^*$-action on $\bC^{d=n+1}$  
\[
\lambda_1\cdot (z_1,\ldots, z_{n+1})=(\lambda_1 z_1,\ldots, \lambda_1 z_{n+1}).
\]
Consider $U=\bC^{n+1}-\{0\}$; then the $N_\bC$ action on $U$ is free.  We find that the quotient $U/N_\bC$ 
\begin{equation}\label{eq: Pn cx quotient}
\bP^n=(\bC^{n+1}-\{0\})/\bC^*,
\end{equation}
is the complex projective space consisting of  complex lines through the origin in $\mathbb C^{n+1}$.

Note that so long as $\rho_\bC$ is the one given above in Equation \eqref{eq: rho Pn}, the resulting quotient is $\bP^n$ as in Equation \eqref{eq: Pn cx quotient}, no matter what $\beta_\bC$ is.  The choice of $\beta_\bC$ is not unique for a given $\rho_\bC$.  For the $\rho_\bC$ given by Equation \eqref{eq: rho Pn}, the linear map $(d\rho_\bC)_1$ has the matrix form \begin{equation}\label{eq: Pn Q}
Q=[\begin{array}{ccc} 1 & \cdots & 1\end{array}]^T.
\end{equation} One choice of $B$ such that $BQ=0$ is 
\begin{equation}\label{eq: B Pn}
B=
\left[
\begin{array}{ccc|c@{}c} 
 &  & & -1 \\
&\mathbb I_{n} & & \vdots \\
&  &  & -1 \\
\end{array}
\right],
\end{equation}
where $\mathbb I_n$ is the $n\times n$ identity matrix.  This matrix $B$ corresponds to 
\begin{equation}\label{eq: beta Pn}
\beta_\bC(\tit_1,\ldots,\tit_{n+1}) =  (\tit_1 \tit_{n+1}^{-1}, \ldots \tit_n \tit_{n+1}^{-1}).
\end{equation}
\qed 
\end{example}

\subsection{Symplectic quotients}\label{subsec:compact_red} 

 Let $N=U(1)^{d-n}$ be the maximal compact subgroup of $N_\bC$. The $N_\bC$-action on $\bC^d$ defined in Equation \eqref{eq: Nc action} restricts to a Hamiltonian $N$-action on the symplectic manifold
\begin{equation} \label{eq: Omega Cd}
\left(\bC^d, \omega_0=\frac{i}{2} \sum_{k=1}^d dz_k \wedge d\bar{z}_k\right).
\end{equation}
In this subsection, we first consider the symplectic quotient $\bC^d/\!\!/N$ of $(\bC^d,\omega_0)$ with respect to the action by $N$.  The symplectic quotient provided in Theorem \ref{thm: MW quotient} is a symplectic toric manifold.  After that, we will state in Theorem \ref{thm: GIT} that this symplectic quotient can be identified with a complex geometric quotient.  

To describe the symplectic quotient $\bC^d/\!\!/ N$, we first need to obtain the moment map of the $N$-action. Consider the maximal compact subgroup $\bT^d=U(1)^d$ of $\bT^d_\bC$. Then
the $\bT^d_\bC$-action on $\bC^d$ defined in Equation \eqref{eq: Td action} restricts to a Hamiltonian $\bT^d$-action on the symplectic manifold $(\bC^d, \omega_0)$ given in Equation \eqref{eq: Omega Cd}
with a moment map (which is unique up to addition of a constant vector in $\bR^d$ that plays a role in Equation \eqref{eq: L}) 
\begin{equation}\label{eq: Cd moment map}
\mu_{\bT^d} :\bC^d\to \bR^d, \quad \mu_{\bT^d}(z_1,\ldots,z_d) =\frac{1}{2}(|z_1|^2,\ldots, |z_d|^2).
\end{equation}
The image of $N$ under $\rho_\bC$ is a subtorus of $\bT^d$, and one can equivalently think of the $N$-action on $\bC^d$ as an action of its image $\rho_\bC(N)$, which  inherits the  action of $\bT^d$.  Hence, 
the moment map (which is unique up to addition of a constant vector in $\bR^{d-n}$) of the Hamiltonian $N$-action  is
\begin{equation}\label{eq: muN}
\begin{array}{c}
 \mu_N :=(d\rho_\bC)_1^* \circ \mu_{\bT^d}:  \bC^d\to \bR^{d-n},\\ \quad \mu_N(z_1,\ldots,z_d)  =\frac{1}{2} \left( \sum\limits_{k=1}^d Q^1_k |z_k|^2,\ldots,  \sum\limits_{k=1}^d Q^{d-n}_k |z_k|^2\right).
 \end{array}
\end{equation}
\begin{example}[$\bP^n$] \label{ex: muN Pn} In our running example $\bP^n$, the holomorphic $N_\bC=\bC^*$-action on $\bC^{n+1}$ restricts to a Hamiltonian $N=U(1)$-action on the  symplectic manifold $(\bC^{n+1},\omega_0)$. Its moment map can be calculated from the linear map $(d\rho_\bC)_1$, which in matrix form is given by Equation \eqref{eq: Pn Q}, and $\mu_{\bT^d}$ given in Equation \eqref{eq: Cd moment map}.   Up to a constant, it is 
\begin{equation}\label{eq:facets_mom_map}
\begin{array}{ll}
& \mu_N:  \bC^{n+1}\to \bR,\\
& \mu_N(z_1,\ldots,z_{n+1})  =(d\rho_\bC)_1^*\circ \mu_{\bT^d}(z_1,\ldots, z_{n+1})=\frac{1}{2}\sum\limits_{k=1}^{n+1} |z_k|^2.
\end{array}
\end{equation}
\qed
\end{example}

For  $a \in \bR^{d-n}$, the following two statements are equivalent: 
\begin{itemize}
\item $a$ is a regular value of $\mu_N$, meaning that $d\mu_N$ is surjective at each point in $\mu_N^{-1}(a)$, and hence $\mu^{-1}_N(a)$ is a smooth manifold of the expected real dimension $2d-(d-n)=d+n$. 
\item The $N$-action on $\mu^{-1}_N(a)$ is locally free. 
\end{itemize}
This is because the surjectivity of $d\mu_N$ is equivalent to the linear independence of the Hamiltonian  vector fields generated by the $d-n$ components of $\mu_N$.

For the rest of this paper, we will only be considering the cases where $a$ is a regular value of $\mu_N$ and that $N$ acts freely on $\mu_N^{-1}(a)$, not just locally free.  When the $N$ doesn't act freely, the quotient might be an orbifold.  We are now ready to describe the symplectic quotient $\bC^d/\!\!/N$ of $\bC^d$ by the action of $N$. 
\begin{theorem} [Marsden-Weinstein \cite{MW74} and Meyer \cite{Meyer} symplectic reduction]\label{thm: MW quotient}  Consider the Hamiltonian action on $(\bC^d, \omega_0)$ by the compact group $N$  with a moment map $\mu_N: \bC^d\to \bR^{d-n}$, as described above.  If $a$ is a regular value of $\mu_N$ and that $N$ acts freely on $\mu^{-1}_{N}(a)$,  then the symplectic quotient 
\[
M:=\bC^d/\!\!/ N = \mu_N^{-1}(a)/N
\]
admits a unique structure of a smooth manifold of real dimension $2n$ such that the projection $\pi_a:\mu_N^{-1}(a)\to \mu_N^{-1}(a)/N$ is a submersion, and it carries a canonical symplectic form $\omega_a$ such that 
\[
\pi_a^*\omega_a = \iota_a^*\omega_0,
\]
where  $\iota_a: \mu^{-1}_N(a)\hookrightarrow \bC^d$ is the inclusion.  
\end{theorem}
\noindent In Section \ref{section: alpha}, we will discuss that $M$ carries an effective Hamiltonian $\bT^n$-action, which makes it a toric symplectic manifold.  Since $(\bC^d, \omega_0)$ is K\"{a}hler, compatible with the standard complex structure, the symplectic reduction $\mu^{-1}_N(a)/N$ also has a natural K\"ahler structure.  We will discuss the K\"ahler structure in more detail in Section \ref{sec: kahler}.

This symplectic quotient $M=\mu^{-1}_N(a)/N$ can also be described as a complex geometric quotient by the following theorem.
\begin{theorem}[Kempf-Ness \cite{KempfNess}, Kirwan \cite{KirwanBook}, Audin {\cite{AudinBook}}, Guillemin {\cite{GuilleminBook}}] \label{thm: GIT} There is a Zariski open set $U_a\subset \bC^d$  such that $\mu_N^{-1}(a)\subset U_a$, $N_\bC$ acts freely on $U_a$, and 
\[M=\mu^{-1}_N (a)/N=U_a/N_\bC.\]
The quotient $M$ is a smooth manifold carrying the canonical symplectic structure $\omega_a$ as in Theorem \ref{thm: MW quotient} and it carries the unique complex structure such that $\widetilde \pi_a: U_a\to U_a/N_\bC$ is holomorphic.  
\end{theorem}

\noindent  Here we decided to simply use the notation $M$ instead of $M_a$ since whenever we use $M$ later on, it's always for a fixed $a$, so there's no confusion.  The Kempf-Ness theorem \cite{KempfNess} is a deep and more general theorem describing the equivalence between the geometric invariant theory (GIT) quotient of a smooth complex projective variety $X$ by the linear action of a complex reductive group $G_c$ and the symplectic quotient of $X$ by the maximal compact subgroup $G$ of $G_c$.  We will not describe the geometric invariant theory in this article.  Audin {\cite[Proposition 3.1.1]{AudinBook}} gave a proof of the above result in the specialized setting for toric manifolds that is similar to Kirwan's {\cite[Theorem 7.4]{KirwanBook}} more general proof for quotients of K\"ahler manifolds.  Guillemin {\cite[Section A1.2]{GuilleminBook}} gave a more elementary proof in the case of toric manifolds.  

We do not explain the proof for Theorem \ref{thm: GIT} in this article; however, in Section \ref{subsubsec: Ua} we will provide an explicit description of $U_a$ that can be found in \cite{AudinBook}, \cite{GuilleminBook}, corresponding to $\mu^{-1}_N(a)/N$, in this setting of toric manifolds.  This theorem is also illustrated in our Example \ref{ex: O(-1)+O(-1)}, Example \ref{ex: Pn}, and Example \ref{ex: KPn_setup}.  Note that the proof of this theorem depends on the convexity property of the moment map of the $\bT^n$-action on the symplectic toric manifold $\mu^{-1}_N(a)/N$, which is guaranteed for compact toric manifolds by Theorem \ref{thm: delzant} but not guaranteed in general for noncompact toric manifolds.  However, for all of our noncompact examples, the image of the moment map is convex.

  \begin{example}[$\Tot(\cO(-1)\oplus \cO(-1)\to \bP^1)$]\label{ex: O(-1)+O(-1)}  Consider $N_\bC=\bC^*$ acting on $\bC^4$ by 
\[\lambda_1 \cdot (z_1, z_2, z_3, z_4)=(\lambda_1 z_1, \lambda_1 z_2, \lambda_1^{-1}z_3, \lambda_1^{-1}z_4).
\]
It restricts to a Hamiltonian $N=U(1)$-action on $(\bC^4, \omega_0)$ with the moment map $\mu_N:\bC^4\to \bR$ given by 
$
\mu_N(z_1,z_2,z_3,z_4)=\frac{1}{2}(|z_1|^2+|z_2|^2-|z_3|^2-|z_4|^2).
$
Then $a\in \bR$ is a regular value if and only if $a\neq 0$, and the level set is 
\[\mu^{-1}_N(a)=\{|z_1|^2+|z_2|^2-|z_3|^2-|z_4|^2=2a\}\subset \bC^4.\]
We can see that for $z\in \mu_N^{-1}(a)$, if  $a>0$, then $z_1$ and $z_2$ cannot both be 0, and if  $a<0$, $z_3$ and $z_4$ cannot both be 0.  If we take $U_a$ to be 
\[
U_a=\begin{cases} 
(\bC^2 - \{0\})\times \bC^2, & a>0;\\
\bC^2\times (\bC^2-\{0\}), & a<0,
\end{cases}
\]
then $\mu^{-1}_N(a) \subset U_a$ and
\[
\mu^{-1}_{N}(a)/U(1)=U_a/\bC^*,
\]
giving the total space of $\cO(-1)\oplus \cO(-1)\to \bP^1$.  For $a>0$, the zero section is given by 
\[
\{(z_1, z_2, 0 ,0 ) \mid |z_1|^2+|z_2|^2=2a\}/U(1)=S^3(\sqrt{2a})/U(1)=S^2(\sqrt{a/2})\cong \bP^1.
\]
Similarly, for $a<0$, the zero section is given by
\[
\{(0, 0, z_3 , z_4 ) \mid -|z_3|^2-|z_4|^2=2a\}/U(1)=S^3(\sqrt{-2a})/U(1)=S^2(\sqrt{-a/2})\cong \bP^1.
\]
\qed
\end{example}

Given the identification of orbit spaces $U_a/N_\bC=\mu^{-1}_N(a)/N$, for any $z\in U_a$, it is in one of the $N_\bC$-orbits, which corresponds to a unique $N$-orbit in $\mu^{-1}_N(a)\subset U_a$, so there exists $\widetilde\lambda_a(z)\in N_\bC$ such that $\rho_\bC(\widetilde \lambda_a(z))\cdot z$ is in that $N$-orbit. Furthermore, because the $N_\bC$-action is free, there is a unique $\lambda_a(z)\in N_\bC/N=(\bR_{>0})^{d-n}$ such that $\rho_\bC(\lambda_a(z))\cdot z\in \mu^{-1}_N(a)$.  Therefore, we have the following deformation retraction.
\begin{definition}[Deformation retraction $R_a$] \label{def:def_retract} For any $z\in U_a$, take the unique $\lambda_a(z) \in N_\bC/N =  (\bR_{>0})^{d-n}$  such that $\rho_\bC(\lambda_a(z)) \cdot z \in \mu_N^{-1}(a).$ This defines a deformation retraction
\begin{equation}\label{eq:def_retract}
R_a: {U_a} \rightarrow \mu_N^{-1}(a) \subseteq \bC^d, \quad z\mapsto \lambda_a(z)\cdot z=\rho_\bC(\lambda_a(z))\cdot z.
\end{equation}
We will see in the example provided in Section \ref{subsec: KP1} that computing this deformation retraction is the main challenge one encounters when trying to express  the moment map of the $\bT^n$-action on $M$ in terms of the homogeneous coordinates $z\in \bC^d$.

\end{definition}

\begin{example}[$\bP^n$]\label{ex: Pn}  One can see from Equation \eqref{eq:facets_mom_map} that $a\in \bR$ is a regular value of $\mu_N$ if and only if $a\neq 0$, and the level sets are 
\[\mu^{-1}_N(a)=\begin{cases}
\{z\in \bC^{n+1} \mid \sum_{k=1}^{n+1}|z_k|^2=2a\}=S^{2n-1}(\sqrt{2a}), & a>0;\\
\emptyset, & a<0.
\end{cases}
\]
So if $z\in \mu^{-1}_N(a)\subset \bC^{n+1}$ and $a>0$, then $z\neq 0$.  We see that if we take $U_a=\bC^{n+1}-\{0\}$, then 
$$
\mu_N^{-1}(a)/U(1) = U_a/\bC^*=(\bC^{n+1}-\{0\})/\bC^* =\bP^n,
$$
where $\mu^{-1}_N(a)$ inherits a symplectic structure while $(\bC^{n+1}-\{0\})/\bC^*$ inherits a complex structure. 

For any $z\in \bC^{n+1}-\{0\}$, by Equation  \eqref{eq:facets_mom_map},  the unique $\lambda_a(z) \in \bC^*/U(1) \simeq \bR_{>0}$ such that $\lambda_a(z)\cdot z\in \mu_N^{-1}(a)$ is   $\lambda_a(z_1,\ldots,z_{n+1})=\sqrt{ \frac{2a}{\sum_{j=1}^{n+1}|z_j|^2}} 
$.  Hence, with  $\rho_\bC$ given in Equation  \eqref{eq: rho Pn}, the deformation retraction is given by
\begin{equation}\label{eq: Pn Ra}
R_a: \bC^{n+1}-\{0\} \rightarrow \mu_N^{-1}(a), \qquad (z_1,\ldots,z_{n+1}) \mapsto   \sqrt{ \frac{2a}{\sum_{j=1}^{n+1} |z_j|^2}} (z_1,\ldots,z_{n+1}).
\end{equation}
 \qed

\end{example}

\subsection{Canonical line bundle $K_M$ of a toric manifold $M$} \label{subsec: KM} This is a noncompact example, though as mentioned in the introduction section, the theory for compact toric manifolds carries over to this example.  
Let $M$ be a compact toric manifold obtained as a symplectic quotient of $(\bC^d, \omega_0)$ with respect to the action of $N$ as in the previous Section \ref{subsec:compact_red}.  In this subsection, we give a description of the total space of the canonical line bundle $K_M$ over $M$ as a symplectic reduction of
\begin{equation}\label{def:C_d+1}
    \left(\bC^{d+1}, \  \omega_0^+ =\frac{i}{2} \Big(\sum_{k=1}^d dz_k \wedge d\bar{z}_k + dp\wedge d\bar{p}\Big)\right),
\end{equation}
also with respect to an action of $N$.  So $K_M$ is a toric manifold of complex dimension $n+1$, which is one complex dimension higher than $M$.   We use $(z_1,\ldots z_d, p)$ to denote the coordinates on $\bC^{d+1}=\bC^d\times \bC$, where $z\in \bC^d$ are the homogeneous coordinates on $M$ as in the previous section, and $p$ is the additional coordinate in the fiber direction.  Notation-wise in this subsection, we add a superscript $+$ to maps from the previous section to denote the extensions of those maps. 

First we describe the complex geometric quotient. Let $\bT^{d+1}_\bC = (\bC^*)^{d+1}$ act on $\bC^{d+1}$ by 
\[
\begin{array}{ccl}
\bT^{d+1}_\bC \times (\bC^d \times \bC )&\to& \bC^d\times \bC\\
 (\tit^+, z,p)& \mapsto& \tit^+ \cdot (z,p)  := (\tit_1 z_1,\ldots,\tit_d z_d, \tit_{d+1}p).
 \end{array}
\]
Again the $\bT^{d+1}_\bC$-action on $\bC^{d+1}$ restricts to a Hamiltonian $U(1)^{d+1}$-action on the symplectic manifold $(\bC^{d+1}, \omega_0^+)$ with a moment map (which is unique up to addition of a constant vector in $\bR^{d+1}$) 
$$
\mu_{\bT^{d+1}} :\bC^{d+1}\to \bR^{d+1}, \quad \mu_{\bT^{d+1}}(z_1,\ldots,z_d, p) =\frac{1}{2}(|z_1|^2,\ldots, |z_d|^2, |p|^2).
$$

The $\rho_\bC$ chosen above for $M$ in Equation (\ref{eq: rho_action}) determines that for $K_M$, and the extended map $\rho^+_\bC:  N_\bC=(\bC^*)^{d-n}  \longrightarrow \bT^{d+1}_\bC$ needs to be
\begin{equation}
\rho^+_\bC (\lambda_1,\ldots,\lambda_{d-n}) = \left(\prod_{\ell=1}^{d-n} \lambda_\ell^{Q^\ell_1},\ldots, \prod_{\ell=1}^{d-n} \lambda_\ell^{Q^\ell_d},
\prod_{\ell=1}^{d-n} \lambda_\ell^{-\sum_{k=1}^d Q^\ell_k} \right).
\end{equation}
In Section \ref{sec: KM justification}, we will provide the justification that this extension $\rho_\bC^+$ determines a $N$-action on $\bC^{d+1}$ such that the symplectic reduction $\bC^{d+1}/\!\!/N$ is $K_M$, but for now we take this fact for granted.

For this given $\rho^+_\bC$, the map $\beta^+_\bC: \bT^{d+1}_\bC \to \bT^{n+1}_\bC$  is of the form
$$
\beta^+_\bC(\tit_1,\ldots, \tit_{d+1}) = \left(\prod_{k=1}^d \tit_k ^{v^k_1}   ,\ldots, \prod_{k=1}^d \tit_k^{v^k_n} , \prod_{k=1}^{d+1} \tit_k \right)
$$ 

and the corresponding matrices $Q^+$ and $B^+$ are given by
\begin{equation}\label{eq: Q+ and B+}
Q^+ = \left[ \begin{array}{ccc}  Q^1_1 & \cdots & Q^{d-n}_1  \\
\vdots &   & \vdots\\
Q^1_d & \cdots & Q^{d-n}_d \\
-\sum_{k=1}^d Q^1_k &\cdots & -\sum_{k=1}^d Q^{d-n}_k 
 \end{array}\right], \quad
 B^+= \left[\begin{array}{cccc} v^1_1 & \cdots & v^d_1 & 0 \\
\vdots &  &    & \vdots\\
v^1_n & \cdots & v^d_n & 0 \\
1 & \cdots & 1 & 1
 \end{array}\right].
 \end{equation}
Again by exactness of the analogue of Equation (\ref{eqn:NTT}), $0\to \fn_\bC  \stackrel{ (d \rho^+_\bC)_1 }{\longrightarrow}\ft^{d+1}_\bC \stackrel{ (d\beta^+_\bC)_1}{\longrightarrow}\ft^{n+1}_\bC \to 0$, we know that $B^+Q^+=0$.
\begin{example}[$K_{\bP^2}$]\label{ex: KP2 QB}
The $Q$ and $B$ for $\bP^2$ from Example \ref{ex: cx Pn} extend to $Q^+$ and $B^+$ for $K_{\bP^2}$ as
$$
Q^+= \left[ \begin{array}{r} 1 \\  1\\  1 \\ -3 \end{array}\right], \quad B^+ = \left[ \begin{array}{rrrr} 1 & 0 & -1 & 0 \\ 0 & 1 & -1 & 0\\ 1 & 1 & 1 & 1 \end{array}\right],
$$ 
where the corresponding maps are 
$$\rho^+_\bC (\lambda_1) =(\lambda_1,\lambda_1,\lambda_1,\lambda_1^{-3}), \quad \beta^+_\bC(\tit_1,\tit_2,\tit_3,\tit_{4})=(\tit_1\tit_3^{-1},\tit_2\tit_3^{-1},\tit_1\tit_2\tit_3\tit_4).$$ 
\qed
\end{example}

Now we describe the symplectic quotient. Let $N_\bC\cong (\bC^*)^{d-n}$ act on $\bC^{d+1}$ by $\lambda\cdot (z,p) = \rho^+_\bC(\lambda)\cdot (z,p)$, where $\lambda\in N_\bC$, $z\in \bC^d$, and $p\in \bC$. The restriction 
of this action to the maximal compact subgroup $N=U(1)^{d-n}$ is a Hamiltonian $N$-action on the standard $\bC^{d+1}$ of Equation (\ref{def:C_d+1}) with a moment map (unique up to addition of a constant vector in $\bR^{d-n}$) given by 
\begin{align}\label{eq: muN+}
\begin{split}
&\mu^+_N : \bC^{d+1} \to \bR^{d-n}\\
&\mu_N^+(z_1,\ldots,z_d,p) =\frac{1}{2} \left( \sum_{k=1}^d Q^1_k (|z_k|^2-|p|^2) ,\ldots,  \sum_{k=1}^d Q^{d-n}_k (|z_k|^2 -|p|^2) \right).
\end{split}
\end{align}

Let $a =(a_1,\ldots, a_{d-n}) \in \bR^{d-n}$ be a regular value of $\mu_N^+$. Then $(\mu^+_N)^{-1}(a)$ is a real submanifold of $\bC^{d+1}=\bR^{2d+2}$ of dimension $d+n+2$ preserved by the $N$-action.  The total space $K_M$ of the canonical line bundle over $M$ is the symplectic quotient
\begin{equation}\label{eq: Km}
K_M=(\mu^+_N)^{-1}(a)/N=(U_a\times \bC)/N_\bC.
\end{equation}
Note that $U_a^+=U_a\times\bC$ because $M$ sits in $K_M$ as the zero section as discussed in Remark \ref{subsec:connxn_M_KM} below. Like Equation \eqref{eq: define alpha}, there is again a deformation retraction to go from $U_a \times \bC$ to $(\mu_N^{+})^{-1}(a)$; for any $(z,p) \in {U}_a\times\bC $, there is a unique $\lambda_a(z,p) \in (\bR_{>0})^{d-n} \subset (\bC^*)^{d-n}=N_{\bC}$ such that
$$
\lambda_a(z,p) \cdot (z,p) \in (\mu^+_N)^{-1}(a),
$$
from which we obtain the deformation retraction
$$
R^+_a: {U}_a \times \bC \rightarrow (\mu^+_N)^{-1}(a), \quad (z,p) \mapsto \lambda_a(z,p)\cdot (z,p)  .
$$ Then $K_M$ carries a canonical symplectic form $\omega_a^+$ satisfying $(\pi_a^+)^*\omega_a^+ = (\iota_a^+)^*\omega_0$, where 
 \[
 \iota^+_a:  (\mu^+_N)^{-1}(a) \hookrightarrow \bC^{d+1}
 \] is the inclusion  and 
\begin{equation}\label{eq:projs_KM}
\begin{cases}
\pi^+_a: (\mu^+_N)^{-1}(a) \rightarrow K_M =  (\mu^+_N)^{-1}(a)/N, \\
\widetilde \pi^+_a: {U_a} \times \bC \rightarrow K_M = ({U_a}\times \bC)/N_\bC,
\end{cases}
\end{equation}
are natural projections to the quotient $K_M$. We now illustrate with some examples.

\begin{example}[$K_{\bP^n}$]\label{ex: KPn_setup} For $M=\bP^n$, from Example \ref{ex: KP2 QB}, we see that to get $K_{\bP^n}$, the $N_\bC=\bC^*$ action is 
\begin{equation}\label{eq: Nc action KPn}
\lambda_1 \cdot (z_1,\ldots, z_{n+1},p)  = (\lambda_1 z_1,\ldots, \lambda_1 z_{n+1}, \lambda_1^{-n-1}p)
\end{equation}
which restricts to a Hamiltonian $U(1)$-action with moment map \begin{equation}\label{eq:tmu+}
\mu_N^+(z_1,\ldots,z_{n+1},p)  =\frac{1}{2} \left(\sum_{k=1}^{n+1} |z_k|^2-(n+1)|p|^2\right),   
\end{equation}
that is unique up to a constant.  Then $a\in \bR$ is a regular value of $\mu^+_N$ if and only if $a\neq 0$.  From the description of the regular level sets
\begin{equation}\label{eq: level set KPa}
(\mu_N^+)^{-1}(a)=\left\{\sum_{k=1}^{n+1} |z_k|^2-(n+1)|p|^2=2a\right\},
\end{equation}
we see that if $(z,p) \in (\mu_N^+)^{-1}(a)\subset \bC^{n+2}$ and $a>0$, then $z\neq 0$, and if $a<0$, then $p\neq 0$.  Taking 
\begin{equation}\label{eq: Ua for Kpn}
U_a^+= \begin{cases}
(\bC^{n+1}-\{0\})\times \bC, & a>0;\\
 \bC^{n+1} \times (\bC-\{0\}), & a< 0,
 \end{cases}
\end{equation}
we see that
\begin{equation}\label{eq: formula for Kpn}
(\mu^+_N)^{-1}(a)/U(1) = \begin{cases} 
\left((\bC^{n+1}-\{0\})\times \bC\right)/\bC^* =K_{\bP^n}, & a>0;\\
\left( \bC^{n+1} \times (\bC-\{0\})\right)/\bC^* = \bC^{n+1}/\bZ_{n+1}, & a< 0.
\end{cases}
\end{equation} 
The justification for $\left((\bC^{n+1}-\{0\})\times \bC\right)/\bC^* =K_{\bP^n}$ is provided in Section \ref{sec: KM justification}.  Indeed, for $K_{\bP^n}$, $U_a^+=U_a\times \bC$.

The deformation retraction $R_a^+:U_a^+\to (\mu_N^+)^{-1}(a)$ is given by 
\begin{equation}\label{eq: retraction KPn}
    R^+_a(z_1,\ldots,z_{n+1},p) = \la_a(z,p)\cdot(z,p) = (\la_a(z,p)z, \la_a(z,p)^{-n-1}p)
\end{equation}
Computing $\la_a(z,p)$ will prove to be more challenging, which we explain in Section \ref{subsec: KP1}.  \qed
\end{example}

\begin{remark}[Connection between $M$ and $K_M$]\label{subsec:connxn_M_KM}

In short, the connection is that the $p=0$ level set is $M$. In more detail, let $a \in \bR^{d-n}$ be a regular value, and define
$$
h: (\mu_N^+)^{-1}(a)\to \bC , \quad h(z_1,\ldots,z_d,p) = p.
$$

Then from Equation \eqref{eq: muN+} for $\mu_N^+$ and Equation \eqref{eq: muN} for $\mu_N$, we get
\begin{eqnarray*}
h^{-1}(p)& =&\left\{(z_1,\ldots,z_d, p)\in\bC^d \times \{p\}\; \Big| \; \sum_{k=1}^d Q^\ell_k|z_k|^2 = 2a_\ell + \sum_{k=1}^d Q_k^\ell |p|^2  \right\}\\
&\cong& \mu_N^{-1}\left(a_1 +\frac{1}{2}\sum_{k=1}^d Q_k^1|p|^2,\ \ldots, \ a_{d-n} + \frac{1}{2} \sum_{k=1}^d Q_k^{d-n} |p|^2\right).
\end{eqnarray*}
When $p=0$, we see that $h^{-1}(p)=\mu_N^{-1}(a)$. Between $M$ and $K_M$ there exist inclusion maps
$$
\mu_{\textcolor{black}{N}}^{-1}(a) \times \{0\} \subset (\mu^+_N)^{-1}(a), \quad {U_a}\times \{0\}\subset  {U_a}\times \bC.
$$
which descend to the following inclusion as the zero section: 
$$
i_0: M = \mu_N^{-1}(a)/N  = {U_a}/N_{\bC}   \hookrightarrow 
K_M = (\mu^+_N)^{-1}(a)/N =  ({U_a}\times\bC)/N_{\bC},
$$
and $i_0^*\omega^+_a =\omega_a$.
\qed

\begin{example}[$K_{\bP^n}$] For the $n$-dimensional projective space,
$$
\begin{array}{ll}
h^{-1}(p)  &  =\left\{ (z,p) \in \bC^{n+1}\times \{p\} \mid |z_1|^2+\cdots+ |z_{n+1}|^2 = 2a +(n+1)|p|^2\right\}\\
& \cong \mu_N^{-1}\left(a + \frac{(n+1)|p|^2}{2}\right)= S^{2n-1}\left(\sqrt{2a+(n+1)|p|^2}\right)
\end{array}
$$
where $\mu_N$ is the moment map given by Equation \eqref{eq:facets_mom_map}. In particular, we can see that the level sets of $h$ are compact spheres.  
 \qed

\end{example}

\end{remark}

\section{Toric actions and moment maps} \label{sec: moment map}
We now discuss the toric geometry of $M$. That is, the geometry arising from the effective $\bT^d_\bC/N_\bC\cong \bT^n_\bC$-action on $M$.  

\subsection{Toric $\bT^n$-action on $M$ and its moment map}\label{section: alpha}  For clarity, we have broken this section into 4 smaller subsections.  Example \ref{ex: s_quot_mom_map} on  $\bP^n$ is provided towards the end to illustrate the theory discussed in this section. Since $\bT^n_\bC=\bT^d_\bC/N_\bC$ is a quotient, we start by describing orbits of the $\bT^d_\bC$-action on $\bC^d$.  

\subsubsection{Orbits of the $\bT^d_\bC$-action on $\bC^d$.} There is a one-to-one correspondence between orbits of $\bT^d_\bC$ in $\bC^d$ and multi-indices of the form
\begin{equation}\label{eq: multi-index}
J=\emptyset \text{ or } J=(j_1,\ldots, j_r), \quad 1\leq j_1< \cdots <  j_r\leq d. 
\end{equation}
 This is because each orbit of $\bT^d_\bC$ is of the form
\begin{equation} \label{eq: CJ}
\bC^d_J:=\{(z_1,\ldots, z_d)\mid z_j =0 \text{ iff. } j\in J\}.
\end{equation}
which has complex dimension $d-r$. For $z\in \bC^d_J$, the stabilizer of $z$ is the subtorus  
\begin{equation}\label{eq: stabilizer}
(\bT^d_\bC)_J:=\{(\tilde t_1,\ldots, \tilde t_d)\in \bT^d_\bC \mid \tilde t_j =1 \text{ iff. } j\in \{1,\ldots, d\} \backslash J\}. 
\end{equation}
And note  $\bC^d_\emptyset =(\bC^*)^d=\bT^d_\bC$ is the only orbit whose elements have trivial stabilizer. 

Later in this section we will see that for the symplectic toric manifold $M=U_a/N_\bC$, $U_a$ is a union of $\bC^d_J$ for a certain collection of $J$.

\subsubsection{Toric $\bT^n$-action on $M$.} In Section \ref{subsec:compact_red}, we obtained the symplectic toric manifold $M=\mu^{-1}_N(a)/N=U_a/N_\bC$.  The holomorphic $\bT^d_\bC$-action is effective on $U_a$, so there is an open embedding of $\bT^d_\bC\cong (\bC^*)^d$ in $U_a$, and it descends to an open embedding of $\bT^n_\bC=\bT^d_\bC/N_\bC\cong (\bC^*)^n$ in $M=U_a/N_\bC$.   As shown in the diagram
\begin{equation}\label{eq: alpha diagram}
\begin{tikzcd}[row sep=huge]
N_\bC \arrow[hookrightarrow, r, "\rho_\bC"]  & 
\bT^d_\bC\cong (\bC^*)^d \arrow[r, hookrightarrow] \arrow[d, shift right, "\beta_\bC" ']  &
 U_a \arrow[d, "\widetilde \pi_a"]  \\
& \bT^n_\bC=\bT_\bC^d/N_\bC\cong (\bC^*)^n   \arrow[u, hookrightarrow, shift right,  "\alpha_\bC"']  \arrow[r, hookrightarrow]&
M=U_a/N_\bC 
\end{tikzcd},
\end{equation}
this embedding of $\bT^n_\bC\hookrightarrow M=U_a/N_\bC$ can be factored through $\bT^d_\bC$ if we choose a group homomorphism 
\begin{equation}\label{eqn:alpha}
\alpha_\bC: \bT_\bC^n \to \bT_\bC^d
\end{equation}
that is the right inverse of the surjective group homomorphism $\beta_{\bC}:\bT_\bC^d\to \bT_\bC^n$, characterized by 
$\beta_{\bC} \circ \alpha_\bC(t) = t$ for all $t\in \bT_\bC^n$.   In particular, $\alpha_\bC$ is injective, as the right split of the exact sequence $0\to N_\bC\to \bT^d_\bC\to \bT^n_\bC\to 0$, and $\bT^d_\bC=\rho_\bC(N_\bC)\oplus \alpha_\bC(\bT^n_\bC)$.

More explicitly, $\alpha_\bC$ is of the form
\begin{equation} \label{eq: define alpha}
\alpha_{\bC}(t_1,\ldots,t_n) =\left(\prod_{m=1}^n t_m^{s_1^m},\ldots, \prod_{m=1}^n t_m^{s_d^m}\right)\in \bT^d_\bC 
\end{equation}
where $s_k^m \in \bZ$, $1\leq k\leq d$, $1\leq m\leq n$, and
\begin{equation}\label{eqn:vs}
\sum_{k=1}^d  v^k_m s_k^{m'}   ={\delta}_m^{m'},\quad 1\leq m,m'\leq n.
\end{equation} 
Taking the differential of $\alpha_\bC$ at identity, we obtain an injective linear map 
$$
\ft_\bC^n =\bC^n \stackrel{(d\alpha_\bC)_1}{\longrightarrow} \ft_\bC^d =\bC^d
$$
given by the matrix
\begin{equation}\label{eq: A}
A =\left[ \begin{array}{ccc} s^1_1 &\cdots & s^n_1\\ \vdots & & \vdots \\ s^1_d &\cdots & s^n_d \end{array}\right].
\end{equation}
Equation \eqref{eqn:vs} is equivalent to  $BA=\mathbb I_{n}$, where $\mathbb I_n$ is the $n\times n$ identity matrix. Applying $\Hom(  -,\bC^*)$ to \eqref{eqn:alpha}, we obtain a surjective
map
\begin{equation}\label{eqn:alphaz}
\alpha^*_\bZ: \Hom(\bT_{\bC}^d,\bC^*) \cong \bZ^d \to \Hom(\bT_{\bC}^n,\bC^*) \cong \bZ^n.
\end{equation}
Tensoring \eqref{eqn:alphaz} by $\bR$, we obtain a surjective linear map
$$
\alpha^*: \bR^d \to \bR^n. 
$$
Note that $\alpha^*$ thus defined is the same as the map $(d\alpha_\bC)_1^*$.

Given $\alpha_\bC$, there is a $\bT^n_\bC\cong \alpha_\bC(\bT_\bC^n)$-action on $U_a$ given by
\[
 \begin{array}{ccl}
 \bT^n_\bC\times U_a & \to & U_a\\
(t, z)& \mapsto & \alpha_\bC(t) \cdot z,
\end{array}
\]
where the action of $\alpha_\bC(t)\in \bT^d_\bC$ on $U_a\subset \bC^d$ is the usual component-wise multiplication. Because $\alpha_\bC(\bT^n_\bC)=\bT^d_\bC/\rho_\bC(N_\bC)$, this $\bT^n_\bC$-action on $U_a$ descends to a  $\bT^n_\bC=\bT^d_\bC/N_\bC$-action on $M=U_a/N_\bC$.  That is, for each $z\in U_a$, $t\in \bT^n_\bC$  sends $\widetilde\pi_a(z)\in M$ to $\widetilde\pi_a\left(\alpha_\bC(t)\cdot z\right)$.  This $\bT^n_\bC$-action on $M$ restricts to an effective Hamiltonian $\bT^n\cong U(1)^n$-action   on the symplectic manifold $(M, \omega_a)$, and we discuss its moment map below.

\subsubsection{Moment map of the $\bT^n$-action, moment polytope, and the action-angle coordinates.} By the above construction of the $\bT^n_\bC$-action on $M$ (as descending from a composition of $\alpha_\bC$ and the usual $\bT^d_\bC$-action on $\bC^d$), the moment map $\mu_a: M\to \bR^n$ for the Hamiltonian $\bT^n$-action on $M$, which is unique  to the addition of a constant vector in $\bR^n$, is given by
 \begin{equation}\label{eq: mu a}
\mu_a \circ \pi_a(z) =  \alpha^* \circ \mu_{\bT^d}(z),\quad z\in \mu_N^{-1}(a),
\end{equation}
hence 
\begin{equation}\label{eq: mu a pt2}
\mu_a \circ \widetilde{\pi}_a(z) = \alpha^*\circ \mu_{\bT^d}\circ R_a(z),\quad z\in {U_a}.
\end{equation}
where $R_a$ was the deformation retract defined in Definition \ref{def:def_retract}.  The theorem below provides some important properties about the moment map $\mu_a$.

\begin{theorem} [Atiyah \cite{atiyah} and Guillemin-Sternberg \cite{GS} convexity theorem for compact symplectic toric manifolds] \label{thm: convexity} If $(M^{2n},\omega_a)$ is a compact symplectic toric manifold carrying an effective  Hamiltonian $\bT^n$-action as above, with the moment map $\mu_a: M\to \Lie(\bT^n)^*=\bR^n$, then the level sets of $\mu_a$ are connected, and the image  ${\mathrm{\Delta}} :=\mu_a(M)$ is a convex polytope (know as the moment polytope of $M$) that is equal to the convex hull of the image under $\mu_a$ of the fixed points of the $\bT^n$-action.   
\end{theorem} 
\noindent Delzant further completes the description of the moment polytope ${\mathrm{\Delta}}$ in the following theorem.
\begin{theorem} [Delzant \cite{Delzant} classification theorem for compact symplectic toric manifolds] \label{thm: delzant}
Compact toric symplectic manifolds $(M^{2n}, \omega_a)$ up to $\bT^n$-equivariant symplectomorphism are in one-to-one correspondence with Delzant polytopes ${\mathrm{\Delta}}=\mu_a(M^{2n})\subset \bR^{n}$ up to translation.  Delzant polytopes are convex polytopes satisfying 
\begin{itemize} 
\item simple, i.e. each vertex has $n$ edges, 
\item rational,  i.e. the vectors normal to each facet are generated by a vector in $\bZ^n$, and
\item smooth, i.e. the $n$ integral normal vectors, one for each of the $n$ facets adjacent to a vertex, form a basis of $\bZ^n$.
\end{itemize}

\end{theorem}
\noindent In other words, the one-to-one correspondence in the above theorem says that if $(M_1, \omega_1, \bT^n, \mu_1)$  and $(M_2, \omega_2, \bT^n, \mu_2)$ are equivariantly symplectomorphic toric manifolds, then $\mu_1(M_1)$ and $\mu_2(M_2)$ differ by a translation.  However, for the same manifold,  if we change the basis of $\bT^n$, then that is equivalent to a $\mathrm{GL}(n,\bZ)$-action on $\Lie(\bT^n)^* \cong \bR^n$, which changes the shape of the polytope ${\mathrm{\Delta}}\subset \bR^n$; this is discussed more in Section \ref{sec: change Tn basis} and Example \ref{ex: P2}.

We now explain Delzant's observation that the primitive column vectors of $B$ (i.e. the matrix form of $(d\beta_\bC)_1$ given in Equation \eqref{eq: Q and B}) are inward pointing normal vectors to the facets of ${\mathrm{\Delta}}$. Let $\beta: \bR^d\to \bR^n$ be the restriction to $\bR$ of $(d\beta_\bC)_1: \bC^d\to \bC^n$.  
Applying $\Hom(-, \bR)$ to $\beta$ gives us the dual linear map $\beta^*: \bR^n\to \bR^d$.  From Equation (\ref{eq: mu a}) and the fact that $\alpha_\bC$ is a right inverse of $\beta_\bC$, \textcolor{black}{we can obtain defining equations of the moment polytope, as follows.}
 
As in Equation \eqref{eq: Cd moment map}, each $(\mu_{\bT^d})_j =\frac{1}{2}|z_j|^2+\kappa_j$ for some constant $\kappa_j\in \bR$. The relation between $\kappa$ and $a$ is $(d \rho_\bC)_1^*(\kappa)=-a$, we take the 0 level set of $(d\rho_\bC)_1^* \mu_{\bT^d} = (d\rho_\bC)_1^*\sum_{j=1}^d (\frac{1}{2}|z_j|^2e^j + \kappa_je^j) =(d\rho_\bC)_1^*(\sum_{j=1}^d\frac{1}{2}|z_j|^2e^j) - a = 0 = \mu_N -a$ where $e^j$ form the standard basis on $\bR^d$. Let $L:\bR^n\to \bR^d$ be 
\begin{equation}\label{eq: L}
 L(\xi)_j:= \beta^*(\xi)_j -\kappa_j  = \langle v^j, \xi\rangle - \kappa_j = \sum_{k=1}^n v^j_k \xi_k -\kappa_j,
 \end{equation}
 where $v^j$, $j=1,\ldots d$, are the column vectors of $B$ and $(\xi_1,\ldots, \xi_n) \in {\mathrm{\Delta}}=\mu_a(M)\subset \bR^n$ are the moment map coordinates $(\mu_{a,1},\ldots \mu_{a,n})$. 

 Recall we have the short exact sequence
 \begin{equation}\label{eq: linearized dual ses}
     0 \to \bR^n \xrightarrow{\beta^*} \bR^d \xrightarrow{(d\rho_\bC)_1^*} \mathfrak{n}^* \to 0.
 \end{equation}
So the kernel of $(d\rho_\bC)_1^*$ is precisely the image of $\beta^*$. 


By the definition of symplectic reduction, the kernel can be identified isomorphically with the values $\Re(z)=\Re(z_1,\ldots,z_d)\in \bR^d_{\geq 0}$ such that $\frac{1}{2}\Re(z_j)^2$ satisfy the affine linear equation $$0=(d\rho_\bC)_1^*\left(\frac{1}{2}\Re(z_1)^2 + \kappa_1, \ldots, \frac{1}{2}\Re(z_n)^2 + \kappa_n\right) =(d\rho_\bC)_1^* \circ \mu_{\bT^d}(\Re(z_1),\ldots, \Re(z_n)),
$$
hence $\Re(z) \in \mu_N^{-1}(a)$. In other words, there is a one-to-one correspondence with points $\xi \in \bR^n$:
\begin{equation}
    \beta^*(\xi) = \frac{1}{2}\Re(z)^2 + \kappa.
\end{equation}
Thus, for $\xi\in {\mathrm{\Delta}}$  we have that $L(\xi)_j =\frac{1}{2}|\Re(z_j)|^2\geq 0$, which leads to the following corollary.  (A similar way to think about this is if we restrict the domain of $\beta_\bC$ to the $(\bC^*)^n$ part of an affine coordinate chart such as those given in Equation \eqref{eq: affine chart}, then $\alpha_\bC$ is the inverse of $\beta_\bC$ on that chart.  Then  $\beta^*(\xi)=\beta^*\alpha^*\mu_{\bT^d}(z)=\mu_{\bT^d}(z)$.)    

\begin{corollary}\label{coro: polytope}
The moment polytope ${\mathrm{\Delta}}$ is given by 
\begin{equation}\label{eq: polytope facets}
{\mathrm{\Delta}} = \left\{\xi\in \bR^n\mid L(\xi)_j = \langle v^j ,\xi\rangle -\kappa_j \geq 0 \right\},
\end{equation}
where $v^j$, $j=1,\ldots d$, are the column vectors of $B$.  So each of the $d$ facets of ${\mathrm{\Delta}}$ is given by  
\begin{equation}\label{eq: facet}
\mathcal F_j = \{L(\xi)_j= \langle v^j ,\xi\rangle -\kappa_j =  0\}\cap {\mathrm{\Delta}}, 
\end{equation} 
with $v^j$ being a primitive inward pointing normal vector.  
\end{corollary}

\begin{remark}
This is a result proven in \cite[Equation (6)]{CDG}. The proof here is based off of \cite[Theorem 3.3]{G94} where $\beta = (d \beta_\bC)_1|_{\bR^d}$, $Z^\epsilon_r = \bR^d_+\ni (e^{r_1},\ldots,e^{r_d})$ where $e^{r_j} = \Re(z_j)$, $\iota = (d \rho_\bC)_1$ and their $\alpha_k = \iota^* e_k$, $s_j=(e^{r_j})^2/2$, $\lambda_j=\kappa_j$ though $\la = a$ and one considers the 0 level set, $x\in \bR^n$ there is $\xi$ here, and $h = \mu_{\bT^d}$.
\end{remark}

As a consequence of Theorem \ref{thm: convexity}, we have a fibration $\mu_a: M^{2n}\to {\mathrm{\Delta}}$ whose fibers are connected.  For $\xi$ in the boundary $\partial {\mathrm{\Delta}}$, if $\xi$ belongs to a face of codimension $r$, then  $\mu_a^{-1}(\xi)$ consists of fixed points of a $r$-dimensional subtorus of $\bT^n$.   So the fibers of $\mu_a$ above $\partial {\mathrm{\Delta}}$ are degenerate smaller tori.  Any point $\xi$ in the interior of the polytope $\mathring{{\mathrm{\Delta}}}$ is a regular value of $\mu_a$, meaning that $d\mu_a$ is surjective and so the $n$ Hamiltonian vector fields generated by the components of $\mu_a$ are linearly independent, so the fiber above $\xi$ is $\bT^n$.  Because $\bT^n$ is abelian, these vector fields commute, hence the components of $\mu_a$ Poisson commute, and so the $\bT^n$-orbit is isotropic of dimension $n$.  That means each fiber $\mu_a^{-1}(\xi)\cong \bT^n$ is a Lagrangian submanifold of $(M,\omega)$.  Hence we have the following symplectomorphism identifying $\mu^{-1}_a(\mathring{{\mathrm{\Delta}}})$ with the total space of a Lagrangian torus fibration
\begin{equation}\label{eq: action angle}
\Big(\mu^{-1}_a(\mathring{{\mathrm{\Delta}}}),\omega_a\Big)  \cong \left(\mathring{{\mathrm{\Delta}}}\times \bT^n,\  \sum\limits_{j=1}^n d\xi_j\wedge d\theta_j\right)
\end{equation}
where $(\xi, \theta)$ is known as the angle-action coordinates, with $\xi$ a coordinate on $\mathring {\mathrm{\Delta}}$ and $\theta$ a coordinate on the Lie algebra of $\bT^n$, so $e^{i\theta}\in \bT^n$. This will be discussed in more detail in Section \ref{sec: kahler} on the K\"ahler potential.

\subsubsection{Description of $U_a$ and complex toric coordinates.} \label{subsubsec: Ua}   In Section \ref{sec: cx quotient}, below Theorem \ref{thm: GIT}, we mentioned Audin \cite{AudinBook} and Guillemin \cite{GuilleminBook} gave a description of $U_a$, which we summarize here.  

If $f$ is a face of ${\mathrm{\Delta}}$ of codimension $r$, then $f$ is at the intersection of $r$ facets $\cF_{j_1}\cap \cdots \cap \cF_{j_r}$.  So $f$ corresponds to an ordered multi-index $J_f=(j_1,\ldots, j_r)$ of the form given in Equation \eqref{eq: multi-index}.  Then $U_a$ that satisfies the properties for Theorem \eqref{thm: GIT} is described by 
\begin{equation}
U_a=\bigcup_f \bC_{J_f}^d,
\end{equation}
for $f$ a face of  our moment polytope ${\mathrm{\Delta}}=\mu_a(M)$ and $\bC_{J_f}^d$ as defined in Equation \eqref{eq: CJ}.   For each face $f$ and $z\in \bC_{J_f}^d$, the stabilizer of $z$ is $(\bT^d_\bC)_{J_f}$ given in Equation \eqref{eq: stabilizer}.

The open face that is the interior $\mathring {\mathrm{\Delta}}$ of the polytope is of codimension 0, and corresponding to that we have the $\bT^d_\bC$-orbit $\bC^d_{\emptyset}=(\bC^*)^d$ on which the $\bT^d_\bC$-action is free, and thus descends to a single $\bT^n_\bC$-orbit on which $\bT^n_\bC$-action is free.  Hence we have 
\begin{equation}\label{eq: cx toric}
 \mu_a^{-1}(\mathring {\mathrm{\Delta}})\cong \bT^n_\bC.  
\end{equation}
\begin{definition}[Notation for complex toric coordinates]\label{def: toric coordinates} Let $(u_1,\ldots,u_n)$ be the complex coordinates on the Lie algebra $\ft^n_\bC \cong\bC^n$ of $\bT_\bC^n\cong(\bC^*)^n$. Also, denote $u_j:=x_j+i\theta_j$, so $\theta_j$ is the coordinate on the Lie algebra of $\bT^n$. The exponential map  $\exp: \ft^n_\bC  \to \bT^n_\bC$ 
$$
u= (u_1,\ldots,u_n)\mapsto  t=(t_1,\ldots, t_n)=(e^{u_1},\ldots, e^{u_n}).
$$ 
gives the complex toric coordinates $t$.  
\qed 
\end{definition} 
\noindent With this notation, the identification in Equation \eqref{eq: cx toric} is given by 
\begin{equation}\label{eq: moment map identification}
\mu_a(t_1, \ldots, t_n)=\mu_a( e^{x_1+i \theta_1},\ldots, e^{x_n+i \theta_n}) = (\xi_1,\ldots, \xi_n)\in \mathring{{\mathrm{\Delta}}}. 
\end{equation}
This gives the exchange between the complex toric coordinates and the moment map coordinates established by the combination of 
Equation \eqref{eq: cx toric} and Equation \eqref{eq: action angle}, 
\begin{equation}
\bT^n_\bC \cong \mu_a^{-1}(\mathring{{\mathrm{\Delta}}})\cong \mathring{{\mathrm{\Delta}}}\times \bT^n.
\end{equation}

\begin{example}[$\bP^n$]
\label{ex: s_quot_mom_map}
We continue Examples \ref{ex: cx Pn}, \ref{ex: muN Pn}, and \ref{ex: Pn} on $\bP^n$.  For the choice of $B$ in Equation \eqref{eq: B Pn}, which corresponds to the $\beta_\bC$ in Equation \eqref{eq: beta Pn}, one possible right inverse $A$ and its corresponding map $\alpha_\bC: \bT^n_\bC\to \bT^d_\bC$ are
\begin{equation}\label{eq: alpha Pn}
A = \left[ \begin{array}{ccc} & \bI_n & \\ \hline  0  & \cdots & 0 \end{array}\right],\quad \alpha_\bC(t_1,\ldots, t_n) =(t_1,\ldots t_n, 1).
\end{equation}
(We'll discuss more about the different choices of $A$ in Example \ref{ex: P2}.) Then $t\in \bT^n_\bC$ acts on $U_a=\bC^{n+1}-\{0\}$ via $\alpha_\bC(t)\in \bT^d_\bC$, which acts on $z\in U_a\subset \bC^{n+1}$ by component-wise multiplication 
\[
(t_1,\ldots, t_n)\cdot (z_1,\ldots, z_{n+1}) = (t_1,\ldots t_n, 1)\cdot (z_1,\ldots, z_{n+1})= (t_1 z_1,\ldots, t_n z_n, z_{n+1}).
\]
 This restricts to a Hamiltonian $\bT^n=U(1)^n$-action on $\bC^{n+1}$ with the standard symplectic form, and the moment map is 
 \[
 \alpha^*\circ \mu_{\bT^d}: \bC^{n+1}\to \bR^n, \quad (\alpha^*\circ \mu_{\bT^d})(z_1,\ldots,z_{n+1}) = \frac{1}{2}(|z_1|^2,\ldots, |z_n|^2),
 \]
  up to addition of a constant vector in $\bR^n$.  The above $\bT^n_\bC$-action on $U_a=\bC^{n+1}-\{0\}$ descends to a $\bT^n_\bC$-action on $\bP^n=U_a/N_\bC=\mu^{-1}_N(a)/N$, where $N_\bC=\bC^*$ and $N=U(1)$, as
\[
(t_1,\ldots,t_n)\cdot [z_1: \cdots  :z_{n+1}]= [t_1 z_1: \cdots : t_n z_n: z_{n+1}].
\]
This $\bT^n_\bC$-action restricts to a Hamiltonian $\bT^n$-action on $(\bP^n, \omega_a)$ for any $a>0$ and with the moment map $\mu_a:\bP^n \to \bR^n$, which (up to a constant in $\bR^n$) written in terms of the homogeneous coordinates is 
\begin{equation}\label{eq:Pn_mom_map}
\begin{array}{ll}
\mu_a([z_1: \cdots :z_{n+1}])& = (\alpha^*\circ \mu_{\bT^d} \circ R_a)(z_1,\ldots,z_{n+1}) \\
&=  \frac{a  (|z_1|^2,\ldots, |z_n|^2)}{|z_1|^2+\cdots + |z_{n+1}|^2}=(\xi_1,\ldots, \xi_n),
\end{array}
\end{equation}
where recall that the formula for $R_a$ is given by Equation \eqref{eq: Pn Ra}.  The image of this moment map is then the moment polytope 
\[
\mathring{{\mathrm{\Delta}}} =\{ (\xi_1,\ldots,\xi_n)\in \bR^n: \xi_1>0, \ldots, \xi_n>0, a-\xi_1-\cdots - \xi_n >0\},
\]
and one can see that the inward pointing primitive normal vectors to the facets are indeed the column vectors of $B$.  For $n=2$, this moment polytope is shown in Figure (\ref{subfig: standard P2}).
\qed
\end{example}
 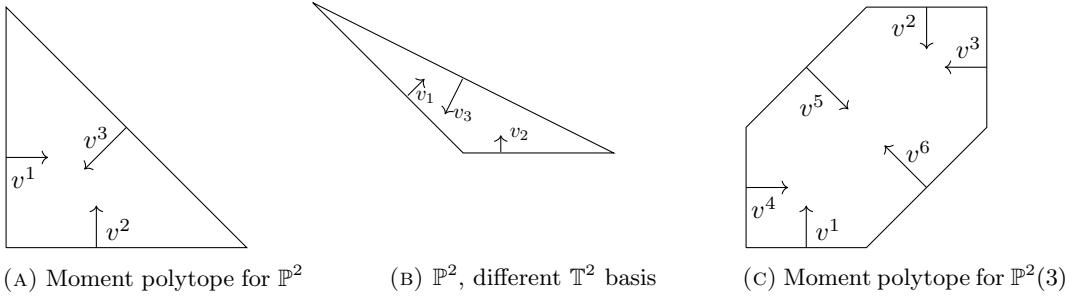
\begin{figure}[htbp]
\centering
\begin{subfigure}[b]{0.25\textwidth}
\begin{tikzpicture}[scale=0.8]
\draw (0,0) -- (4,0) -- (0, 4) -- (0,0);
\draw[->] (1.5, 0)-- (1.5, 0.7);
\node [right] at  (1.5, 0.3) {$v^2$};
\draw[->] (0, 1.5)-- (0.7, 1.5);
\node [below] at  (0.3, 1.5) {$v^1$};
\draw[->] (2,2)-- (1.3, 1.3);
\node [above] at  (1.5, 1.5) {$v^3$};
\end{tikzpicture}
\caption{Moment polytope for $\bP^2$}
\label{subfig: standard P2} 
\end{subfigure}
\begin{subfigure}[b]{0.35\textwidth}
\begin{tikzpicture}[scale=0.5, every node/.style={inner sep=0,outer sep=0, scale = 0.8}]
\draw (0,0) -- (4,-4) -- (8, -4) -- (0,0);
\node (v1) at (2.5,-2.5) {};
\node (v2) at (3,-2) {};
\node (v5) at (5,-4) {};
\node (v6) at (5,-3.5) {};
\node (v3) at (4,-2) {};
\node (v4) at (3.5,-3) {};
\draw[->]  (v1) edge (v2);
\draw[->]  (v3) edge (v4);
\draw[->]  (v5) edge (v6);
\node at (3,-2.5) {$v_1$};
\node at (4,-3) {$v_3$};
\node at (5.5,-3.5) {$v_2$};
\end{tikzpicture}
\vspace{0.5in}
\caption{$\bP^2$, different $\bT^2$ basis}
\label{subfig: slanted P2}
\end{subfigure}
\begin{subfigure}[b]{0.27\textwidth}
\begin{tikzpicture}[scale=0.8]
\draw  (6,0)--(8,0)--(10, 2)--(10, 4)--(8, 4)--(6,2)--(6,0);
\draw [->] (7,0)--(7, 0.7);
\node [right] at (7, 0.3) {$v^1$};
\draw[->] (6,1)--(6.7, 1);
\node[below] at (6.3, 1) {$v^4$};
\draw[->] (9,1)--(8.3, 1.7);
\node [right] at (8.5, 1.6) {$v^6$};
\draw[->] (10, 3)--(9.3, 3);
\node [above] at (9.7, 3) {$v^3$};
\draw[->] (9, 4)--(9, 3.3);
\node [left] at (9, 3.7) {$v^2$};
\draw[->] (7, 3)--(7.7, 2.3);
\node[below] at (7.1, 2.7) {$v^5$};
\end{tikzpicture}
\caption{Moment polytope for $\bP^2(3)$}
\label{subfig: hexagon}
\end{subfigure}
\caption{In the three figures above, the labeling, $v^j$, of the inward pointing normal vectors to the facets matches with the ordering of the columns of $B$ in Examples \ref{ex: s_quot_mom_map} ($n=2$) and \ref{ex: P2}, $B'$ in Example $\ref{ex: P2}$, and $B$ in Example \ref{ex: CP2(3)}, respectively.}
\label{fig: CP2 and CP2(3)}
\end{figure}

\begin{example} [The blowup $\bP^2(3)$ of $\bP^2$ at the three torus fixed points]  \label{ex: CP2(3)}  The projective plane blown up at one point is isomorphic to $\bP^1 \times \bP^1$, see \cite[p 479-480]{gh} or note that blow-up of a complex surface at a vertex of the polytope replaces the point with $\bP^1$ hence the triangle for $\bP^2$ becomes a quadrilateral. Then we can perform two toric blow-ups at the fixed points corresponding to two opposite vertices of the rectangle which is the moment map of $\bP^1 \times \bP^1$, to obtain a hexagon. (See \cite[p 40-41]{fulton} for the toric blow up and \cite[Section 7]{symp_intro} for the symplectic blow up.) So in this case we know $n=2$ and $d=6$. Here is one choice of $B$, giving the moment polytope in Figure (\ref{subfig: hexagon}).  
$$
B = \left[ \begin{array}{rrrrrr} 0& 0 & -1 & 1 & 1 & -1 \\ 1 & -1 & 0 & 0 & -1 & 1 \end{array}\right],\quad
Q = \left[ \begin{array}{cccc}
 1 & 0 & 1 & -1 \\  
 1 & 0 & 0 & 0\\
 0& 1 & 0 & 0\\
 0 & 1 & -1 & 1\\
 0 & 0 & 1 & 0\\
0 & 0 & 0 & 1 
\end{array}\right],
$$
$$
\beta_\bC(\tit_1,\ldots,\tit_{6}) =  (\tit_3^{-1}\tit_4\tit_5\tit_6^{-1},  \tit_{1} \tit_2^{-1}\tit_5^{-1}\tit_6),
$$
$$
\rho_\bC(\lambda_1,\ldots, \lambda_4)=(\lambda_1\lambda_3\lambda_4^{-1},\  \lambda_1,\ \lambda_2,  \ \lambda_2\lambda_3^{-1}\lambda_4,\   \lambda_3, \  \lambda_4).
$$
We can make the following choice for $\alpha_{\bC}$:
$$
A = \left[ \begin{array}{cc}
 0 & 1  \\  
 0& 0 \\
 0& 0 \\
 1 & 0\\
 0& 0\\
 0 & 0 
\end{array}\right], \quad \alpha_{\bC}(t_1,t_2)=(t_2, 1, 1, t_1, 1,1).
$$
Again, the choice of $\alpha_\bC$ is not unique. 
Here $\mu_N: \bC^6\to \bR^4$ is given by  
$$
\begin{array}{l}
\mu_N(z_1,\ldots, z_6)\\
\hspace{0.3in}=\frac{1}{2}\Big(|z_1|^2+|z_2|^2, |z_3|^2+|z_4|^2, |z_1|^2-|z_4|^2+|z_5|^2, -|z_1|^2+|z_4|^2+|z_6|^2\Big),
\end{array}
$$
and we will not try finding $R_a$ for this example here. It involves solving a third degree polynomial. We will get a feel for the complexity of solving for $R_a$ in Section \ref{subsec: KP1} below for $K_{\bP^1}$; the formula gets very complicated. 
\qed
\end{example}

\subsection{Canonical bundle $K_M$ continued}  \label{sec: KM alpha}

In Section \ref{subsec: KM}, from $M=U_a/N_\bC=\mu_N^{-1}(a)/N$, we constructed $K_M$ as quotients $K_M=U_a^+/N_\bC=(\mu_N^+)^{-1}(a)/N$, where $U_a^+=U_a\times \bC$.  In this section, we writen down the general formulation for the moment map of the $\bT^{n+1}$-action on $K_M$. Extending the theory for the construction of the moment map from $M$ to $K_M$ is straightforward, the main difference is keeping track of the extra $p$ coordinate corresponding to the fiber direction.  However, writing down the moment map explicitly is a hard problem in general, as we'll illustrate in Section \ref{subsec: KP1}.

\begin{definition}\label{def:alpha-p} We define $\alpha^+_\bC: \bT_\bC^{n+1} \to \bT_\bC^{d+1}
$ by
\begin{equation}
\alpha^+_{\bC}(t_1,\ldots,t_n, t_{n+1}) =\left(\prod_{m=1}^n t_m^{s_1^m},\ldots, \prod_{m=1}^n t_m^{s_d^m},   \prod_{m=1}^n t_m^{-\sum_{k=1}^d s_k^m} \cdot t_{n+1}\right).
\end{equation}
\end{definition}
\noindent Then $\beta^+_\bC \circ \alpha^+_\bC (t)=t$ for all $t\in \bT^{n+1}_\bC$, i.e. 
$\alpha^+_\bC$ is a right inverse of the surjective group homomorphism $\beta^+_\bC: \bT^{d+1}_\bC \to \bT^{n+1}_\bC$. 
Taking the differential of $\alpha^+_\bC$ at the identity, we obtain an injective linear map 
\begin{equation}\label{eqn:alpha-p}
\ft_\bC^{n+1} =\bC^{n+1} \stackrel{(d\alpha^+_\bC)_1}{\longrightarrow} \ft_\bC^{d+1} =\bC^{d+1}
\end{equation}
given by the matrix
$$
A^+ =\left[ \begin{array}{cccc} s^1_1 &\cdots & s^n_1 & 0\\ \vdots & &  & \vdots \\ s^1_d &\cdots & s^n_d & 0  \\ 
\displaystyle{-\sum_{k=1}^d s_k^1} & \cdots & \displaystyle{ -\sum_{k=1}^d s_k^n} & 1 \end{array}\right].
$$
Then $B^+A^+=\bI_{n+1}$, where $\bI_{n+1}$ is the $(n+1)\times (n+1)$ identity matrix. Below we write the matrix for a given choice of $\alpha^+_{\bC}$ for a couple of examples:

\begin{example}[$K_{\bP^n}$]\label{ex:KPn_choice1} 

$$
B^+ = \left[ \begin{array}{c|r|r}  & -1 & 0 \\ \mathbb{I}_{n} & \vdots & \vdots \\ & -1 & 0\\ \hline 1 \ldots 1 & 1 & 1\end{array}\right],\quad
A^+ = \left[ \begin{array}{rrr|r} & & & 0 \\ & \mathbb{I}_{n} & & \vdots \\ & & & 0 \\ \hline 0 & \ldots & 0 & 0\\ \hline -1 &  \ldots &  -1  & 1\end{array}\right].
$$
This choice, with all 1's in the last row of $B^+$, is often used in the polytopes defining SYZ mirrors in mirror symmetry (discussed in Section \ref{sec: HMS}). We can make other choices for $B^+$, and we discuss one such choice in Example \ref{ex: KPn}.
\qed
 \end{example}

\begin{example}[$K_{\bP^2(3)}$]  \label{key example}  Extending Example \ref{ex: CP2(3)} on $\bP^2(3)$, for $K_{\bP^2(3)}$ we can take
$$
B^+ = \left[ \begin{array}{rrrrrrr} 0& 0 & -1 & 1 & 1 & -1 & 0 \\ 1 & -1 & 0 & 0 & -1 & 1 & 0 \\ 1 & 1 & 1 & 1 & 1 & 1 & 1 \end{array}\right],
Q^+ = \left[ \begin{array}{rrrr}
 1 & 0 & 1 & -1 \\  
 1 & 0 & 0 & 0\\
 0& 1 & 0 & 0\\
 0 & 1 & -1 & 1\\
 0 & 0 & 1 & 0\\
0 & 0 & 0 & 1 \\
-2& -2 & -1& -1
\end{array}\right], A^+ = \left[ \begin{array}{rrr}
 0 & 1  & 0  \\  
 0& 0 & 0 \\
 0& 0 & 0\\
 1 & 0 &  0\\
 0& 0 & 0\\
 0 & 0 & 0\\
 -1 & -1 & 1
\end{array}\right],
$$ 
and the maps are expressed as follows 
$$\beta^+_\bC(\tit_1,\tit_2,\tit_3,\tit_{4},\tit_5,\tit_6,\tit_7)=(\tit_3^{-1}\tit_4\tit_5\tit_6^{-1},\tit_1\tit_2^{-1}\tit_5^{-1}\tit_6,\tit_1\tit_2\tit_3\tit_4\tit_5\tit_6\tit_7),$$
$$\rho^+_\bC (\lambda_1,\lambda_2,\lambda_3,\lambda_4) =(\lambda_1\lambda_3\lambda_4^{-1}, \lambda_1,\lambda_2,\lambda_2\lambda_3^{-1}\lambda_4,\lambda_3,\lambda_4,\lambda_1^{-2}\lambda_2^{-2}\lambda_3^{-1}\lambda_4^{-1}), \text{ and}$$
\[
\alpha_\bC^+ = (t_2, 1, 1, t_1, 1, 1, t_1^{-1}t_2^{-1}t_3).
\]
\qed
\end{example} 

Furthermore, we have an analogue in the $K_M$ case of Equations \eqref{eq: mu a} and \eqref{eq: mu a pt2} by the same arguments. They are
\begin{equation}
\mu^+_a \circ \pi^+_a(z) =  (\alpha^+)^* \circ \mu_{\bT^{d+1}}(z),\quad z\in (\mu_N^+)^{-1}(a),
\end{equation}
and
\begin{equation}\label{eq: mu a for KM}
\mu^+_a \circ \widetilde\pi_a^+(z) = (\alpha^+)^*\circ \mu_{\bT^{d+1}}\circ R^+_a(z),\quad z\in {U_a}\times \bC.
\end{equation}
The image of $\mu^+_a$ is again a convex polyhedron ${\mathrm{\Delta}}^+$.


\subsection{Holomorphic coordinate charts for $M$} \label{sec: holom chart} In this section, we describe holomorphic coordinate charts for the complex manifold $M=U_a/N_\bC$.  This will help us better understand both $M$ and the $\bT^n_\bC$-action.  Examples  \ref{ex: Pn chart} and \ref{ex: P2} are provided at the end for illustrating the content of this section. 
\subsubsection{Open covering of $M$.}  Any vertex $v$ of the moment polytope ${\mathrm{\Delta}}$ is at the intersection of $n$ facets $\cF_{j_1}\cap \cdots \cap \cF_{j_n}$, where $J_v=(j_1,\ldots, j_n)$ is an ordered multi-index of the form in Equation \eqref{eq: multi-index} and each facet  is given by Equation \eqref{eq: facet}.   For any $z\in \bC_{J_v}^d$, there is an open neighborhood of $z$ 
\begin{equation}
\widetilde V_{J_v}:=\{ (z_1, \ldots, z_d) \in \bC^d \mid z_j\neq 0 \text{ if } j\notin J_v\}\subset U_a,
\end{equation}
and this descends to an open set
\begin{equation}
V_{J_v} :=\{[z_1:\ldots:z_d]\in  M=U_a/N_\bC  \mid z_j \neq 0 \text{ if } j\notin J_v \} \subset M=U_a/N_\bC.
\end{equation}
 Open sets of this kind, one for each vertex $v$ of ${\mathrm{\Delta}}$, give an open covering of 
 \begin{equation} 
 M=\bigcup_{v} V_{J_v}.
 \end{equation}
 \subsubsection{Holomorphic charts and inhomogeneous coordinates.} For each $V_{J_v}$, we have a biholomorphic map $\varphi^{J_v}: V_{J_v}\to \bC^n$.  To describe $\varphi^{J_v}$, let us assume, without loss of generality by renumbering the facets, that $J_{v}=(1,\ldots, n)$.   
 
 Then 
 \[
V_{J_v} =\left\{[z_1:\ldots:z_d]\in  M=U_a/N_\bC  \mid z_j \neq 0 \text{ if } j\in \{n+1,\ldots, d\} \right\}.
\]
Now, we would like to scale $z_{n+1}, \ldots, z_d$ to $1$ using the the $N_\bC$-action and the notations for it in Equation \eqref{eq: general rho}.  Pick $\lambda_1,\ldots \lambda_{d-n}\in N_\bC$ such that 
\begin{equation}\label{eq: scaling}
\prod_{\ell=1}^{d-n} \lambda_\ell^{Q^\ell_{n+1}}=z_{n+1}^{-1}, \ \  \ldots,\ \  \prod_{\ell=1}^{d-n} \lambda_\ell^{Q^\ell_{d}}=z_d^{-1}.
\end{equation}
This is doable for the following reason.  We know that the point $(0, \ldots, 0, 1, \ldots , 1)$, with the first $n$ entries being 0 and the last $d-n$ entries being 1, is a point in $U_a$.  The $N_\bC$-action on this point is given by 
\[
(\lambda_1,\ldots,\lambda_{d-n})\cdot (0,\ldots , 0, 1,\ldots, 1)= \left(0,\ldots, 0, \prod_{\ell=1}^{d-n} \lambda_\ell^{Q^\ell_{n+1}},\ldots, \prod_{\ell=1}^{d-n} \lambda_\ell^{Q^\ell_{d}}\right).
\]
Because the $N_\bC$-action is free, from the above equation we see that the following map is an isomorphism
\[
\begin{array}{rcl}
N_\bC=(\bC^*)^{d-n} & \rightarrow & (\bC^*)^{d-n}\\
(\lambda_1,\ldots, \lambda_{d-n}) &\mapsto & \left(\prod\limits_{\ell=1}^{d-n} \lambda_\ell^{Q^\ell_{n+1}},\ldots, \prod\limits_{\ell=1}^{d-n} \lambda_\ell^{Q^\ell_{d}}\right).
\end{array}
\]
This shows that there is a choice of $\lambda_1,\ldots ,\lambda_{d-n}$ such that Equation \eqref{eq: scaling} holds. Then we have
 \begin{equation}\label{eq: V in y coordinates}
V_{J_v} =\left\{\left[ y^{J_v}_1= \prod_{\ell=1}^{d-n} \lambda_\ell^{Q^\ell_{1}}z_1 :\ldots:  y^{J_v}_n=\prod_{\ell=1}^{d-n} \lambda_\ell^{Q^\ell_{n}}z_n: 1:\cdots :1\right] \right\}\subset M=U_a/N_\bC.
\end{equation}
So we get the map 
\begin{equation}\label{eq: holomorphic chart map}
\varphi^{J_v}: V_{J_v} \to  \bC^n, \quad 
 [z_1:\cdots :z_d] \mapsto y^{J_v}=(y_1^{J_v},\ldots, y_n^{J_v}).
 \end{equation}
 The origin $y^{J_v}=0$ of this $\bC^n$-chart corresponds to the preimage of the vertex $v$,
 \[p^{J_v}=\mu_a^{-1}(v)=[0:\cdots : 0: 1: \cdots :1]\in V_{J_v},\]
  which is a fixed point of the $\bT^n_\bC$-action.
  
  Equations \eqref{eq: V in y coordinates} and \eqref{eq: holomorphic chart map} identify $\varphi^{J_v}(V_{J_v})=\bC[y_1^{J_v},\ldots, y_n^{J_v}]$  with the affine subspace
  \begin{equation} \label{eq: affine chart}
    U_{J_v}:=\{(z_1,\ldots, z_d)\mid  z_j=1 \text { if } j\notin J_v\}\subset \widetilde V_{J_v}\subset \bC^d,
  \end{equation}  
  with $y_{k}^{J_v}=z_{j_k}$ for $k=1,\ldots, n$.
  Again, when using $J_v=(1,\ldots, n)$, we have $U_{J_v}=\{(z_1,\ldots, z_n, 1,\ldots, 1)\} \subset \bC^d$, and its identification with $\bC[y_1^{J_v},\ldots, y_n^{J_v}]$ is given by $z_k=y^{J_v}_k$, $k=1,\ldots, n$. Because of this identification, below in section we will often conveniently write $U_{J_v}=\bC[y_1^{J_v},\ldots, y_n^{J_v}]$.
  
  To note the terminology, the coordinates $z=(z_1,\ldots, z_d)\in U_a\subset \bC^d$ are called the homogeneous coordinates on $M$. For each vertex $v$, the coordinates $y^{J_v}=(y_1^{J_v},\ldots, y_n^{J_v})$ on the $\bC^n$-chart associated to $v$, more precisely on the image of the chart map $\varphi^{J_v}(V_{J_v})$, are called the inhomogeneous coordinates.  Because of the identification of $(y_k^{J_v})$ with the coordinates $(z_{j_k})\in U_{J_v}$ on the affine subspace, $y^{J_v}$ are also called the affine coordinates in the literature.

\subsubsection{Change of basis of $\bT^n_\bC$ and the embedding of $\bT^n_\bC$ in each  $\bC^n$-chart.} \label{sec: change Tn basis} Our matrix $B$ is a linear transformation $\mathfrak t^d_\bC\to \mathfrak t^n_\bC$, where $\mathfrak t^n_\bC=\Lie(\bT^n_\bC)$, so changing the basis of $\mathfrak t^n_\bC$ is equivalent to performing a row operation on $B$ to get a new matrix $\widehat B$.

Let us continue to assume that $v$ is the vertex such that $J_v=(1,\ldots, n)$, so the first $n$ column vectors $\{v^1,\ldots, v^n\}$ of $B$ are linearly independent. The remaining column vectors $v^{n+1},\ldots, v^d$ can then be expressed as linear combinations of $\{v^1,\ldots, v^n\}$ as 
\begin{equation}\label{eq: change basis cjk}
v^{n+k}=\sum_{j=1}^nc_{jk}v^j,
\end{equation}
for $k=1,\ldots, d-n$.  Now let us take
\[
\{v^1,\ldots, v^n\}
\]
to be the new basis for $\mathfrak t^n_\bC$. We see that $B$ in the old basis is now
\begin{equation}
\widehat B= PB= \left[\begin{array}{c|ccc} \mathbb{I}_{n} & C
\end{array}\right]
\end{equation}
in the new basis, where $P$ is the $n\times n$ change of basis matrix with $P^{-1}$ being the matrix whose columns are $v^1,\ldots, v^n$, in that order,  and $ C=(c_{jk})$ is a $n\times (d-n)$ matrix consisting of the constants $(c_{jk})$ from  Equation \eqref{eq: change basis cjk} above.  Note for the moment polytope, since the first $n$ columns of $B$ are the vectors normal to the facets meeting at the vertex $v$, this change of basis of $\mathfrak t^n_\bC$ transforms those normal vectors from from $\{v^1,\ldots, v^n\}$ to the column vectors of $\bI_n$.  In other words, it straightens out that corner.  

One choice of right inverse to $\widehat B$ is 
\[
\widehat A^{J_v}= \left[\begin{array}{c} \bI_n \\ \hline 0 \end{array}\right].
\]
This is not the only choice of right inverse to $\widehat B$, but this is the choice that equips the $\bC^n$-chart at this vertex $v$ with the standard $\bT^n_\bC$-action.  Indeed, following the setup in  Equations \eqref{eq: define alpha} and \eqref{eq: A}, this $\widehat A^{J_v}$ corresponds to $\widehat \alpha_\bC^{J_v}$ where
\[
\widehat \alpha_\bC^{J_v}(t_1,\ldots, t_n)=(t_1,\ldots, t_n,1,\ldots, 1) \in U_{J_v}\subset U_a.
\]
The $\widehat \alpha_\bC^{J_v}(\bT^n_\bC)$-action on $U_a$ descends to the following  standard $\bT^n_\bC$-action on $\varphi^{J_v}(V_{J_v}) = \bC^n$ 
\[
(t_1,\ldots, t_n)\cdot \left(y_1^{J_v},\ldots, y_n^{J_v}\right) =\left(t_1y_1^{J_v},\ldots, t_ny_n^{J_v}\right).
\]
 The point $y^{J_v}=0$ is the toric fixed point, which corresponds to the moment map preimage $p^{J_v}=\mu_a^{-1}(v)$ of the vertex $v\in {\mathrm{\Delta}}$.  This action is free on $(\bC^*)^n$, i.e. where $y^{J_v}_j\neq 0$ for all $j=1,\ldots, n$, so this action embeds $\bT^n_\bC\cong (\bC^*)^n$ via $\widehat\alpha_\bC^{J_v}$ in $\varphi^{J_v}(V_{J_v})=\bC^n$.  Combining this embedding with the identification of $U_{J_v}$ with $\bC[y_1^{J_v},\ldots, y_n^{J_v}]$ we discussed around Equation \eqref{eq: affine chart}, we have that $\widehat \alpha^{J_v}_\bC$ gives the following isomorphism 
\[
\begin{array}{rccl}
\widehat \alpha_\bC^{J_v}:&  \bT^n_\bC &\xrightarrow{\ \ \cong\ \ }& U_{J_v}\backslash \{z_k=0\}_{k=1}^n= \bC[y_1^{J_v},\ldots, y_n^{J_v}]\backslash\{y^{J_v}_k=0\}_{k=1}^n\\
 & t_k& \mapsto &  t_k=y_k^{J_v}, \quad \text{for } k=1,\ldots n.
\end{array}
\]

The diagram below summarizes the change of basis we made on $\mathfrak t^n_\bC=\Lie(\bT^n_\bC)$,
\begin{equation}\label{eq: change Tn basis}
\begin{tikzcd}[column sep=huge, ampersand replacement=\&]
\mathfrak t^d_\bC \arrow[r, "B"]  \ar[equal]{d} \&  \mathfrak t^n_\bC\cong \bigoplus\limits_{j=1}^n \bC e^j  \arrow[r, "A^{J_v}=\widehat A^{J_v}P"] \arrow[d, "P"]  \&    \mathfrak t^d_\bC \ar[equal]{d}  \\
\mathfrak t^d_\bC \arrow[r, "\widehat B=PB"]\&\mathfrak t^n_\bC \cong \bigoplus\limits_{j=1}^n \bC v^j    \arrow[r, "\widehat A^{J_v}"]\& 
\mathfrak t^d_\bC 
\end{tikzcd},
\end{equation}
where $e^j$ is the $j^{\textrm{th}}$ standard basis vector.  For the chosen $\hA^{J_v}$ above, the corresponding $A^{J_v}$ in our original basis of $\mathfrak t^n_\bC$ satisfies $BA^{J_v}=\widehat B\widehat A^{J_v}=\bI_n$, so 
\begin{equation}\label{eq: vertex alpha choice}
A^{J_v}=\widehat A^{J_v}P= \left[\begin{array}{c} P \\ \hline 0 \end{array}\right].
\end{equation}
Corresponding to this $A^{J_v}$, we get an $\alpha_\bC^{J_v}$.  This is the choice of $\alpha^{J_v}_\bC$ such that the $\alpha^{J_v}_\bC(\bT^n_\bC)$-action on $U_a$ descends to the standard $\bT^n_\bC$-action on the $\varphi^{J_v}(V_{J_v})=\bC^n$-chart. The point $0\in \bC^n$ in this chart is the fixed point of this action and it corresponds to the moment map preimage $p^{J_v}=\mu_a^{-1}(v)$ of the vertex.   This $\bT^n_\bC$-action is free on the $(\bC^*)^n$ part of the chart, which gives the isomorphism 
\begin{equation}\label{eq: alpha affine chart}
\begin{array}{rccl}
\alpha_\bC^{J_v}:& \bT^n_\bC & \xrightarrow{\ \cong\ } &U_{J_v}\backslash \{z_k=0\}_{k=1}^n=\bC[y_1^{J_v},\ldots, y_n^{J_v}]\backslash\{y^{J_v}_k=0\}_{k=1}^n\\
 & t_k& \mapsto &  \alpha^{J_v}_\bC(t)_k=y_k^{J_v}, \quad \text{for } k=1,\ldots n.
\end{array}
\end{equation}
This relationship is very useful because it gives an identification of the inhomogeneous coordinates $y^{J_v}$ on each of the $\bC^n$-chart with the the complex toric coordinates $t\in \bT^n_\bC$.
 
 \subsubsection{Transition functions.} Suppose $v$ is a vertex at the intersection of $\bigcap_{\ell=1}^n \cF_{j_\ell}$ and $\tilde v$ is another vertex at the intersection of $\bigcap_{k=1}^n \cF_{j_k}$, so $J_v=(j_\ell)_{\ell=1}^n$ and $J_{\tv}=(j_k)_{k=1}^n$. Then the overlap $\varphi^{J_{\tv}}(V_{J_v}\cap V_{J_{\tv}})$ can be identified with 
\[
\varphi^{J_{\tv}}(V_{J_v}\cap V_{J_{\tv}})
= U_{J_{\tv}}\backslash \{z_j=0\}_{j \notin J_{v}}=\bC[y_1^{J_{\tv}},\ldots, y_n^{J_{\tv}}]\backslash \{y_\ell^{J_{\tv}}=0\}_{j_\ell \notin J_{v}}.
\]
Similarly, $\varphi^{J_{v}}(V_{J_v}\cap V_{J_{\tv}})$ can be identified with 
\[
\varphi^{J_{v}}(V_{J_v}\cap V_{J_{\tv}})
= U_{J_{v}}\backslash \{z_j=0\}_{ j\notin J_{\tv}}=\bC[y_1^{J_{v}},\ldots, y_n^{J_{v}}]\backslash \{y_k^{J_{v}}=0\}_{j_k \notin J_{\tv}}.
\]
Note that $\varphi^{J_{\tv}}(V_{J_v}\cap V_{J_{\tv}})=(\bC^*)^n$ if $J_v\cap J_{\tv}=\emptyset$; otherwise, it contains $(\bC^*)^n$ as a proper subset.

By Equation \eqref{eq: alpha affine chart}, we see that $\alpha_\bC^{J_v}$ and $\alpha_\bC^{J_{\tv}}$ give the following isomorphisms,
 \[
  \bT^n_\bC \overset{\alpha_\bC^{J_{\tv}}}{\cong} U_{J_{\tv}}\backslash \{z_j=0\}_{j\in J_{\tv}}=\bC[y_1^{J_{\tv}},\ldots, y_n^{J_{\tv}}]\backslash \{y_j^{J_{\tv}}=0\}_{j=1}^n\subset \varphi^{J_{\tv}}(V_{J_v}\cap V_{J_{\tv}}),
\]
\[
  \bT^n_\bC \overset{\alpha_\bC^{J_{v}}}{\cong} U_{J_{\tv}}\backslash \{z_j=0\}_{j\in J_{v}}=\bC[y_1^{J_{v}},\ldots, y_n^{J_{v}}]\backslash \{y_j^{J_{v}}=0\}_{j=1}^n\subset \varphi^{J_{v}}(V_{J_v}\cap V_{J_{\tv}}).
\]
Note again that the ``$\subset$'' above is ``$=$'' if $J_v\cap J_{\tv}=\emptyset$; otherwise, it is ``$\subsetneq$''.
 
So the transition function 
\[
\varphi^{J_{v}}\circ (\varphi^{J_{\tv}})^{-1}: \varphi^{J_{\tv}}(V_{J_v}\cap V_{J_{\tv}})  \longrightarrow \varphi^{J_{v}}(V_{J_v}\cap V_{J_{\tv}})
\]
 restricted to the $(\bC^*)^n\cong \bT^n_\bC$ part is exactly  $\alpha_\bC^{J_{v}}\circ (\alpha^{J_{\tv}}_\bC)^{-1}$.

Now let $J_v=(1,\ldots, n)$ and below we discuss more explicitly the relationship between the inhomogeneous coordinates $y^{J_v}$ and $y^{J_{\tv}}$ in the overlap of the two charts. Denote by $\tv^k:=v^{j_k}$, $k=1,\ldots, n$, the column vector of $B$ that is the normal vector to the facet $\cF_{j_k}$.  Then
\begin{equation}\label{eq: vertex transition}
\tv^j =\sum_{k=1}^n d_{jk}v^k, 
\end{equation}
where $D=(d_{jk}) \in \mathrm{GL}(n,\bZ)$.  In fact, because $\det(DD^{-1})=\det(D)\det(D^{-1})=1$, and that $\det(D)$ and $\det(D^{-1})$ are both integers, we must have $\det(D)=\pm 1$.  We can further assume that $\det(D)=1$, i.e. $D\in \mathrm{SL}(n,\bZ)$, since otherwise if $\det(D)=-1$, we can interchange $\tv^1$ and $\tv^2$ to make $\det(D)=1$.

Then the transition map for the inhomogeneous coordinates is given by
\begin{equation}
y_k^{J_{v}} =\prod_{j=1}^n ( y_j^{J_{\tv}})^{d_{jk}}\quad\text{ for } k=1,\ldots ,n.
\end{equation}

\subsubsection{Summary.}
To summarize: in \cite{G94} and \cite{AudinBook}, the starting point is a Delzant moment polytope, from which the symplectic and complex geometry of $M$ can be read off. Let the inward normals to its facets $\cF_1,\ldots, \cF_d$ be $v^1,\ldots, v^d$. Because the space of normal vectors to each facet of ${\mathrm{\Delta}}$ is generated by a vector in $\bZ^n$, we can always choose each $v^j$ to be a primitive integral vector that is the inward pointing normal vector to the facet. There is a standard open covering of $M$, corresponding to vertices of ${\mathrm{\Delta}}$, and transition functions, which are obtained by the following pieces of information:

\begin{itemize}[\textbullet]
\item Any subset $J \subseteq \{1,\ldots,d\}$ indexes orbits $\bT^d_{\bC} \cdot(z_1,\ldots,z_d)\subset \bC^d$ such that $z_{j} =0$ when $j \in J$.
\item Each vertex $v$ of the moment polytope is the intersection $\cF_{j_1} \cap \cdots \cap \cF_{j_n}$ of $n$ facets of the polytope ${\mathrm{\Delta}}$ so corresponds to an index set $J_v :=(j_1,\ldots, j_n)\subset \{1,\ldots, d\}$.   
\item Let $p^{J_v} := \mu_a^{-1}(\cF_{j_1} \cap \ldots \cap \cF_{j_n})$ be a toric fixed point of the residual $\bT^d_\bC/N_\bC \cong \bT^n_\bC$-action on $M$, obtained by quotienting $U_a \subset \bT^d_\bC$ by the action of $N_\bC$. 
\item $(\bC^*)^n$ is the dense open $\bT^n_\bC$-orbit corresponding to the $J=\emptyset$ orbit, modulo $N_\bC$, in the quotient $U_a/N_\bC=M$. 
\item For the index set $J_v$, the primitive inward pointing normal vectors $v^{j_k}$, $k=1,\ldots, n$, give the $j_k^{\textrm{th}}$ column of $B$ for each $k$. Different choices of $B$ correspond to $\mathrm{GL}(n, \bZ)$ transformations of ${\mathrm{\Delta}}$.
\item $\alpha_\bC^{J_v}: \bT_\bC^n \to \bT^d_\bC$ passes to an embedding $\widetilde{\pi}_a \circ \alpha_\bC^{J_v}: \bT_\bC^n \xrightarrow{\alpha_\bC^{J_v}} U_a \xrightarrow{\widetilde{\pi}_a} M$ into an open $\bC^n$ set centered at $p^{J_v}$ in $M$, which we may denote $U_{J_v} \subset M$:
\begin{itemize}[$\circ$]
\item $V_{J_v}$ consists of points $[z_1:\ldots:z_d] \in U_a/N_\bC=M$ such that the components of $(z_{j_1},\ldots,z_{j_n})\in \bC^n$ may go to 0 and the other homogeneous coordinates are always nonzero, 
$$V_{J_v} :=\{[z_1:\ldots:z_d]\in U_a/N_\bC=M \mid z_j\neq 0 \text{ if } j \notin J_v \}, $$ 
\item we can express transition functions in terms of the $y_k^{J_v}$ on the overlap of two $\bC^n$ charts; for example in $\bP^2$ in Figure (\ref{subfig: standard P2}) corresponding to the vertex at the right angle we have $\alpha_\bC(t_1,t_2) = (t_1,t_2,1)$ extends to $\bC^2$ and passing to the quotient we obtain $\bC^2 \hookrightarrow \{[t_1:t_2:1] \in \bP^2 \mid (t_1,t_2) \in \bC^2 \} \subset \bP^2$, where the transition function is defined on $\bC \times \bC^*$ (for more details see Example \ref{ex: P2}),
\item we choose the unique representative for each $[z_1:\ldots: z_d]\in V_{J_v} $ which has been scaled to have 1's in all the $d-n$ strictly nonzero coordinates, 
\item the remaining scaled coordinates are the affine (inhomogeneous) coordinates $y^{J_v}_1,\ldots,y^{J_v}_n$ on $U_{J_v}$, 
\end{itemize}
\item the $V_{J_v}$ with charts $U_{J_v}$ give a standard open cover of $M$ indexed by the toric fixed points $v$, in bijection with vertices of the moment polytope at the intersection of the $J_v^{\textrm{th}}$ facets, and
\item transition functions on the dense open orbit of $M$ are of the form $ \alpha_\bC^{J_{v}}\circ  (\alpha_\bC^{J_{\tv}})^{-1}$, where $\bT^n_\bC$ embeds into each $U_{J_v}$ under the corresponding $\alpha_\bC$ and we may use coordinates $(y_1^{J_v},\ldots,y_n^{J_v})$ and $(y_1^{J_{\widetilde v}},\ldots,y_n^{J_{\widetilde v}})$ to express the transition maps, again see Example \ref{ex: P2} for the case of $\bP^2$.
\end{itemize}

\subsubsection{Examples.}
\begin{example}[$\bP^n$]\label{ex: Pn chart}
We continue Example \ref{ex: s_quot_mom_map}. Consider the vertex $v$ of ${\mathrm{\Delta}}$ given by $(\xi_1,\ldots, \xi_n)=(0,\ldots ,0)$, i.e. $v = \cF_1\cap \cdots \cap \cF_n$ is the intersection of the first $n$ facets whose normal vectors are the first $n$ column vectors of $B$ in Equation \eqref{eq: B Pn}.  The $\bC^n$-chart corresponding to this vertex is then 
\[
\begin{array}{ll}
V_{J_v}&=\{[z_1:\cdots : z_{n+1}] \in \bP^n \mid z_{n+1}\neq 0\} \\
&=\left\{ \left[y_1=\frac{z_1}{z_{n+1}}: \cdots: y_n=\frac{z_n}{z_{n+1}}: 1\right] \in \bP^n\right\},
\end{array}
\]
and $\varphi^{J_{v}}: V_{J_v}\to  \bC^n$ given by $[z_1:\cdots:z_{n+1}]\mapsto (y_1,\ldots y_n)$ is a biholomorphic map, with $y=0$ being a fixed point of the $\bT^n_\bC$-action.  The image $\varphi^{J_v}(V_{J_v})=\bC[y_1,\ldots, y_n]$ can be identified with the affine subspace 
\[U_{J_v}=\{(z_1,\ldots z_n, 1)\}\subset \bC^{n+1}
\]
with $y_k=z_k$ for $k=1,\ldots,n$.

Following the diagram in Equation \eqref{eq: alpha diagram}, we see that the map $\alpha_\bC$ chosen in Equation \eqref{eq: alpha Pn} gives an embedding of $\bT^n_\bC$ in $\bP^n$ by 
\[
\bT^n_\bC \overset{\alpha_\bC}{\hookrightarrow} U_{J_v} \rightarrow \bP^n, \quad
(t_1,\ldots, t_n) \mapsto (t_1,\ldots, t_n, 1) \mapsto [t_1:\cdots : t_n :1].
\]
We see that $\alpha_\bC$ is an isomorphism onto $U_{J_v}\backslash \{z_k=0\}_{k=1}^n$. This choice of $\alpha_\bC$, along with the identification of $U_{J_v}$ with $\bC[y_1,\ldots, y_n]$, gives the relation $t_k=y_k$ between the complex toric coordinates $t$ and the inhomogeneous coordinates $y$, which illustrates the identification stated in Equation \eqref{eq: alpha affine chart}.  This identification, along with the notations given in Definition \ref{def: toric coordinates}, $t_j=e^{u_j}=e^{x_j+i\theta_j}$, we can write Equation \eqref{eq:Pn_mom_map} as
 \[
\mu_a(y_1,\ldots, y_n) = (\xi_1,\ldots, \xi_n)=\frac{a(|y_1|^2,\ldots, |y_n|^2)}{|y_1|^2+\cdots+|y_n|^2+1}=\frac{a(e^{2x_1},\ldots, e^{2x_n})}{ e^{2x_1} + \cdots + e^{2x_n} +1},
\]
so
\[
(x_1,\ldots, x_n)= \frac{1}{2} \Big(\log\xi_1 - \log (a-\xi_1-\cdots-\xi_n), \ldots, \log \xi_n-\log (a-\xi_1-\cdots-\xi_n)\Big).
\]
\qed
\end{example}

\begin{example}[$\bP^2$] \label{ex: P2}  We specialize the above example to $\bP^{2}$ to further demonstrate the change of coordinates between charts.  The last paragraph of this example illustrates the change of basis of $\bT^2_\bC$.  Using the matrix $B$ we chose in Equation \eqref{eq: B Pn}, below we list all possible matrices $A$ such that $BA=\bI_{n}$ and its corresponding map $\alpha_\bC$ 
$$
B = \left[ \begin{array}{rrr} 1 & 0 & -1 \\ 0 & 1 & -1 \end{array}\right],\quad
A = \left[ \begin{array}{cc} b & c\\  b& c\\  b & c \end{array}\right]+\left[ \begin{array}{cc} 1 & 0\\  0& 1\\  0 & 0 \end{array}\right], \quad \alpha_\bC(t_1,t_2)=(t_1^{b+1}t_2^c,t_1^bt_2^{c+1},t_1^bt_2^c),
$$
where $b,c \in \bZ$.

Like we discussed in Example \ref{ex: s_quot_mom_map}, the moment polytope is the same as pictured in Figure (\ref{subfig: standard P2}).  At each vertex $v$, there is a choice of $b,c\in \bZ$ such that the above formula gives $\alpha_\bC^{J_v}$, which gives rise to the  standard $\bT^2_\bC$-action on the $\bC^2$-chart at that vertex.  Below we compare the $\bC^2$-charts at the vertices indexed by $\{1,2\}$ and $\{1,3\}$, respectively.

The vertex $\{1,2\}$ is at the lower left corner of the of the triangle, corresponding to the intersection of the two facets normal to columns 1 and 2 of $B$.  The chart at this vertex is the same as the one presented in Example \ref{ex: Pn chart}.  We have
\[
\begin{array}{l}
\varphi^{\{1,2\}}: V_{\{1,2\}}= \{[z_1:z_2: z_3]\in \bP^2\mid z_3\neq 0\} \longrightarrow  U_{\{1,2\}}=\{(z_1, z_2, 1)\in \bC^3\},\\
 \varphi^{\{1,2\}}([z_1: z_2: z_3])=\left(y_1^{\{1,2\}},y_2^{\{1,2\}}\right)=\left(\frac{z_1}{z_3}, \frac{z_2}{z_3}\right),\\
\alpha_\bC^{\{1,2\}}: \bT^2_\bC \xrightarrow{\ \ \cong\ \ } U_{\{1,2\}}\backslash \{z_j=0\}_{j=1,2}=\bC\left[y_1^{\{1,2\}},y_2^{\{1,2\}}\right]\backslash \left\{y_k^{\{1,2\}}=0\right\}_{k=1,2} \\
\alpha_\bC^{\{1,2\}}(t_1, t_2)=(t_1, t_2,1) \Longrightarrow \left(t_1=y_1^{\{1,2\}},  t_2= y_2^{\{1,2\}}\right).
\end{array}
\]
For vertex $\{1,3\}$, which is at the top of the triangle, we have
\[
\begin{array}{l}
\varphi^{\{1,3\}}: V_{\{1,3\}}= \{[z_1:z_2: z_3]\in \bP^2\mid z_2\neq 0\} \longrightarrow  U_{\{1,3\}}=\{(z_1, 1, z_3)\in \bC^3\},\\
 \varphi^{\{1,3\}}([z_1: z_2: z_3])=\left(y_1^{\{1,3\}},y_2^{\{1,3\}}\right)=\left(\frac{z_1}{z_2}, \frac{z_3}{z_2}\right),\\
\alpha_\bC^{\{1,3\}}: \bT^2_\bC \xrightarrow{\ \  \cong\ \  }U_{\{1,3\}}\backslash \{z_j=0\}_{j=1,3}=\bC\left[y_1^{\{1,3\}},y_2^{\{1,3\}}\right]\backslash \left\{y_k^{\{1,3\}}=0\right\}_{k=1,2},\\
\alpha_\bC^{\{1,3\}}(t_1, t_2)=(t_1t_2^{-1}, 1, t_2^{-1}) \Longrightarrow \left(t_1t_2^{-1}=y_1^{\{1,3\}},  t_2^{-1}= y_2^{\{1,3\}}\right).
\end{array}
\]
Using the identification between the complex toric coordinates $t$ with the inhomogeneous coordinates $y$ on the two charts above, we can see that the coordinate change between these two charts is given by 
\[
y_1^{\{1,2\}}= y_1^{\{1,3\}}\left(y_2^{\{1,3\}}\right)^{-1},\quad y_2^{\{1,2\}}=\left(y_2^{\{1,3\}}\right)^{-1}
\]
on the overlap 
\[
\varphi^{\{1,3\}}\left(V_{\{1,2\}}\cap V_{\{1,3\}}\right)=\{(z_1, 1, z_3) \mid z_3\neq 0\}=\bC\left[y_1^{\{1,3\}}, y_2^{\{1,3\}}\right]\backslash\{y_2^{\{1,3\}}= 0\}.
\]
The above transition map is precisely $\alpha_\bC^{\{1,2\}}\circ \left(\alpha_\bC^{\{1,3\}}\right)^{-1}$, except that $\left(\alpha_\bC^{\{1,3\}}\right)^{-1}$ is only defined on the $(\bC^*)^2$, and the above overlap set is $\bC\times \bC^*$, which is slightly bigger.  So the transition map is a slight extension of the map defined using the $\alpha_\bC$'s.

In this paragraph, we illustrate the change of $\bT^2_\bC$-basis.  Changing the basis of the range of $B$, which is $\mathfrak t^2=\Lie(\bT^2_\bC)$, is equivalent to left multiplying $B$ by an element in $\mathrm{GL}(2,\bZ)$ with determinant $\pm 1$.  For example, taking  $ \left[ \begin{array}{cc} 1 & 0 \\ 1 & 1   \end{array}\right]\in \mathrm{GL}(2,\bZ)$, we get a new   $B' = \left[ \begin{array}{rrr} 1 & 0 & -1 \\ 1 & 1 & -2 \end{array}\right]$.  Then the corresponding moment polytope is a triangle that is slanted, with inward normals given by columns of $B'$, as in Figure (\ref{subfig: slanted P2}).
\qed
\end{example}

\subsection{Justification of choices for $K_M$}\label{sec: KM justification} To determine $B^+$ and $Q^+$, we start by considering a vertex $v$ of the moment polytope ${\mathrm{\Delta}}$ for $M$. As in Section \ref{sec: holom chart}, this corresponds to an index $J_v$ of size $n$. After possible renumbering, we may assume $J_v=(1,\ldots,n)$. In particular, we have a $\bZ$-basis 
$$\mathcal{B}:=\{v^1,\ldots,v^n\} \subset \bZ^n$$
for $\bZ^n$ consisting of inward normals to the first $n$ facets. Therefore each of the remaining inward normals $v^{n+1},\ldots,v^d$ can be expressed as a linear combination of vectors in the basis $\mathcal{B}$ via an $n\times (d-n)$ matrix $C=(c_{jk})_{1 \leq j \leq n, 1 \leq k \leq d-n}$, where the $k^{\textrm{th}}$ column expresses $v^{n+k}$ as a vector with respect to $\mathcal{B}$, namely
\begin{equation}\label{eq:Cmatrix}
v^{n+k} = \sum_{j=1}^n c_{jk}v^j.
\end{equation}

Now we would like to answer the question, \emph{what are the inward normals to the polytope ${\mathrm{\Delta}}^+$ for $K_M$?} These make up the columns of $B^+$. Recall the correspondence on $M$ between vertices of $\mathrm{\Delta}$, multi-indices $J$, and coordinate charts $U_{J}$, from Subsection \ref{sec: holom chart}. Let $y_1, \ldots, y_n$ denote coordinates on one coordinate chart corresponding to $J=J_v$ centered around toric fixed point $q\in M \hookrightarrow K_M$ and $\tilde{y}_1,\ldots, \tilde{y}_n$ coordinates on another coordinate chart corresponding to choices with a tilde. In $K_M$, each chart has an additional coordinate $y_q$ and $y_{\tilde{q}}$. First let $D$ denote the $n \times n$ change of basis matrix between inward normals corresponding to each chart of $M$, namely between $\mathcal{B}$ and $\{\tilde{v}^{1},\ldots,\tilde{v}^{n}\}$ (the inward normals to the facets at the intersection of the second vertex): 
\begin{equation}\label{eq: v tilde}
\tilde{v}^{j}=\sum_{k=1}^n d_{jk}v^{k} .
\end{equation}
where now the $j^{\mathrm{th}}$ row expresses $\tilde{v}^j$ as a linear combination of the vectors $v^k$. Note that $D=(d_{jk}) \in \mbox{GL}(n, \bZ)$ has determinant 1 or $-1$ for the same reason as in Equation \eqref{eq: vertex transition}.  If it has determinant $-1$, interchange $\tilde{v}^1$ and $\tilde{v}^2$; this way we may assume $\det D=+1$. Under the identifications of Equation \eqref{eq: moment map identification}, the transition functions turn the coefficients into exponents so that in coordinates
\begin{equation}\label{eq:coord_change}
y_k = \prod_{j=1}^n \tilde{y}_j^{d_{jk}}, \qquad k=1, \ldots, n.
\end{equation}
Note that if $D$ is the change of basis on vectors, then $(D^{-1})^T$ is the change of coordinates. 

An element locally trivializing $K_M$ may be written $y_q dy_1 \wedge \cdots \wedge dy_n$. On the overlap of two charts, we have two different sets of coordinates with which to express a point: 
\begin{equation}\label{eq:KM_trans_fn}
    y_q dy_1 \wedge \cdots \wedge dy_n = y_{\tilde{q}} d\tilde{y}_1 \wedge \cdots \wedge d\tilde{y}_n.
\end{equation}
Taking the log of Equation \eqref{eq:coord_change} and differentiating, we obtain
\begin{equation}
\begin{aligned}
    \log y_k = \sum_{j=1}^n d_{jk} \log \tilde{y}_j & \implies \frac{dy_k}{y_k} = \sum_{j=1}^n d_{jk} \frac{d\tilde{y}_j}{\tilde{y}_j}\\
    & \implies \frac{dy_1}{y_1} \wedge \cdots \wedge \frac{dy_n}{y_n} = \frac{d\tilde{y}_1}{\tilde{y}_1} \wedge \cdots \wedge \frac{d\tilde{y}_n}{\tilde{y}_n},
    \end{aligned}
\end{equation}
where the last equality holds because $\det D = +1$. We can now apply this to Equation \eqref{eq:KM_trans_fn} to determine $y_{\tilde{q}}$.
\begin{equation}\label{eq: KM change of basis}
\begin{split}
y_q dy_1 \wedge \cdots \wedge dy_n &= y_{\tilde{q}} d\tilde{y}_1 \wedge \cdots \wedge d\tilde{y}_n=y_{\tilde{q}}dy_1 \wedge \cdots \wedge dy_n  \frac{\tilde{y}_1\cdots \tilde{y}_n}{y_1\cdots y_n} \\
 \implies y_q & = y_{\tilde{q}} \frac{\tilde{y}_1\cdots \tilde{y}_n}{y_1\cdots y_n} =  y_{\tilde{q}} \prod_{j=1}^n \tilde{y}_j^{1-\sum_{k=1}^n d_{jk}}
\end{split}
\end{equation}
by Equation \eqref{eq:coord_change}. We now have the information needed to write down $B^+$, which we prove below. 

\begin{lemma}\label{lem: B+} Given a matrix $B$ for $M$ of the following form:
\begin{equation}
B= \left[\begin{array}{c|c} \mathbb{I}_{n} &  C \end{array}\right],
\end{equation}
the following gives a choice of $B^+$ for $K_M$:
\begin{equation} 
B^+_{(n+1) \times (d+1)}  = \left[\begin{array}{c|c|c} \mathbb{I}_{n} &  C &  \mathbf{0}\\
\hline 0 &  1 - \sum_{j=1}^n c_{j,1} \ldots 1 - \sum_{j=1}^n c_{j,d-n} &  1 \end{array}\right].
\end{equation}
\end{lemma}
\begin{corollary} One choice of $Q^+$ satisfying $B^+Q^+=0$, for the $B^+$ in Lemma \ref{lem: B+} for $K_M$, is
\begin{equation}
    Q^+_{(d+1) \times (d-n)} := \left[\begin{array}{ccc} & -C  & \\ \hline & \mathbb{I}_{d-n} & \\ \hline \left( \sum_{j=1}^n c_{j,1}\right) -1  & \ldots & \left( \sum_{j=1}^n c_{j,d-n}\right) -1 \end{array} \right]
\end{equation}
and one choice of $A^+$ satisfying $B^+A^+ = \mathbb{I}_{n+1}$ is
\begin{equation}
    A^+_{(d+1) \times (n+1)} := \left[\begin{array}{c|c}
        \mathbb{I}_n & \mathbf{0} \\
        \hline \mathbf{0} & \mathbf{0}\\
       \hline \mathbf{0}  & 1
    \end{array} \right].
\end{equation}
\end{corollary}
\begin{remark} Other choices of $Q^+$ differ from this one via right multiplication by an element in $\mathrm{GL}(d-n, \bZ)$, and they all satisfy the property that the sum of the entries in each column must be $0$, justifying the choice of $Q^+$ we presented in Equation \eqref{eq: Q+ and B+}.
\end{remark}

\begin{remark}
Note that the sum of the entries in a column of $B^+$ equals 1. Namely, the total weight of this $N_\bC$ action is 1. (We will see this play a role later on in the mirror symmetry section in that the $N_\bC$ action preserves $z_1\ldots z_dp$, which is the \emph{superpotential}. It also preserves the Calabi-Yau condition of $K_M$, in which case the above-mentioned $n$-form is global.) 
\end{remark}

\begin{proof} [Proof of Lemma \ref{lem: B+}]  In coordinates $(y_1,\ldots,y_n,y_q)$ on $K_M$, the Hamiltonian $\bT^{n+1}$-action restricts to the $\bT^n$-action on $M$ when $y_q=0$. In particular, this tells us that for $i_0:M \to K_M$, we have $i_0^* \mu_a^+ = \mu_a$ and $\mathrm{\Delta}^+ \cap \{\xi_{n+1}=\kappa_{d+1} \} = \mathrm{\Delta}$ (for $\kappa$ as defined in Equation \eqref{eq: facet}).  Thus to answer our question regarding the inward normals of $\mathrm{\Delta}^+$, we have that $\mathrm{\Delta}^+$ has $d+1$ facets whose normals are denoted by $v^{j, +}$, for $j=1,\ldots, d+1$.  For $j=1,\ldots, d$, the first $n$ coordinates of $v^{j,+}$ are the same as that of $v^j$, i.e. 
\begin{equation}\label{eq: Delta+ facets}
    v^{j,+} = (v^j, v_{n+1}^{j,+}),\qquad 1 \leq j \leq d,
\end{equation}
and below we will discuss how to determine the last coordinate, $v^{j, +}_{n+1}$ based on our changed of coordinates calculation above.  The last facet, which corresponds to $j=d+1$, is in the $\xi_{n+1}=\kappa_{d+1}$ plane, so its primitive inward normal vector is $v^{d+1,+}:=[0\ldots 0 \; 1]^T$.  This last facet intersects every vertex of $\mathrm{\Delta}^+$.  (See also Remark \ref{subsec:connxn_M_KM} and Equation \eqref{eq: mu a for KM} for discussions on relating $M$ and $K_M$.) This can be visualized in the case of $K_{\bP^1}$ in Figure \ref{fig: KP1 total}.

At a vertex $v^+$ of $\mathrm{\Delta}^+$ corresponding to $v$ for $\mathrm{\Delta}$, we hence have a basis of inward normal vectors 
\begin{equation}
\begin{aligned}
\mathcal{B}^+&:=\{v^{1,+},\ldots,v^{n,+}, v^{d+1,+}\}.
\end{aligned}
\end{equation}
With respect to this basis $\mathcal{B}^+$, by Equation \eqref{eq:Cmatrix} and Equation \eqref{eq: Delta+ facets} we can write $B^+$ as
$$\qquad B^+ = \left[\begin{array}{c|c|c} \mathbb{I}_{n} &  C &  \mathbf{0}\\
\hline \mathbf{0} &  v^{n+1,+}_{n+1} \ldots v^{d,+}_{n+1} &  1 \end{array}\right],
$$
since there is no change of basis in the last vector $v^{d+1,+}$ according to Equation \eqref{eq: KM change of basis}. To fill in the last row, it remains to write each $v^{n+k,+}$, $k=1,\ldots, d-n$, in terms of the basis vectors in $\mathcal{B^+}$. So for each $k$, we take $v^{n+k,+}$, let $\tilde v^{1,+}=v^{n+k, +}$, and extend it to a $\bZ$-basis $\{\tilde v^{1,+},\ldots, \tilde v^{n,+}, v^{d+1, +}\}$ of $\bZ^{n+1}$ corresponding to a vertex $\tilde{v}^+$.  Let $\tilde v=\{\tilde v^1,\ldots, \tilde v^n\}$ be the first $n$ coordinates of $\{\tilde v^{1,+},\ldots, \tilde v^{n,+}\}$, and we look at what Equation  \eqref{eq: KM change of basis} says for this choice of $\tilde v$.  So, $\tilde v$ is related to $v$ by Equation \eqref{eq: v tilde}.   Furthermore, because of the choice that $\tilde v^{1,+}=v^{n+k,+}$, by comparing Equation \eqref{eq: v tilde} and Equation \eqref{eq:Cmatrix},  we see that the first row $[d_{11},\ldots, d_{1n}]$ of $D$ is the same as the transpose of the $k$-th column $[c_{1k},\ldots, c_{nk}]$ of $C$, i.e. $d_{1j}=c_{jk}$ for $j=1,\ldots, n$. Hence we obtain
 \begin{equation}
\begin{split}
   v^{n+k, +}= \tilde{v}^{1,+}&= \sum_{j=1}^n d_{1j}v^{j,+} + \bigg(1-\sum_{j=1}^n d_{1j}\bigg)v^{d+1,+} \\
     &= \sum_{j=1}^n c_{jk}v^{j,+} + \bigg(1-\sum_{j=1}^n c_{jk}\bigg)v^{d+1,+}, \qquad k=1,\ldots, d-n,\\
    \end{split}
\end{equation}
where in the first line of the above calculation, the first term follows from Equation \eqref{eq: v tilde} and the second term follows from the coordinate change of $y_q$ in Equation \eqref{eq: KM change of basis}. This proves the lemma. 
\end{proof}

 \begin{example}[$K_{\bP^n}$]\label{ex: KPn} In Example \ref{ex:KPn_choice1}, we presented one choice of $B^+$. Following the prescription given above in this section, we obtain another choice for $B^+$ below, which differs from the previous one by an invertible linear transformation over $\bZ$. (These two different choices of $B^+$ correspond to two different sets of $\bT^{n+1}$-basis, so their moment polytopes look different;  see Figure (\ref{fig: KP1 straight}) for the case of $n=1$). One choice of $A^+$ such that $B^+A^+=\bI_{n+1}$ is given below 
\begin{equation}\label{eq:B+_straight}
   B^+ = \left[\begin{array}{ccc|c|c}
         && &-1&0\\
         & \mathbb{I}_{n}& & \vdots & \vdots \\
          &&& -1&0\\
        \hline 
        0 & \ldots & 0 & n+1& 1
    \end{array} \right], \qquad    A^+ := \left[\begin{array}{ccc|c}
         && &0\\
         & \mathbb{I}_{n}& & \vdots  \\
          &&& 0\\
        \hline
        0 & \ldots & 0 & 0\\
        \hline 
        0 & \ldots & 0 & 1
    \end{array} \right].
\end{equation}
\qed

\subsection{Moment map for $K_{\bP^n}$ in homogeneous coordinates} \label{subsec: KP1} 

Recall in  Example \ref{ex: KPn_setup}, we obtained $K_{\bP^n}$ as the quotient \[
K_{\bP^n}=U_a^+/\bC^*=(\mu^+_N)^{-1}(a)/U(1), 
\]
where $U_a^+=(\bC^{n+1}-\{0\})\times \bC$ and 
$$
{(\mu_N^+)}^{-1}({a}) =  \{ (z_1,\ldots,z_{n+1},p)\in \bC^{n+2} \mid |z_1|^2+\cdots + |z_{n+1}|^2 -(n+1) |p|^2 = 2a\},
$$
and note that we use $a>0$ as this is the chamber that gives $K_{\bP^n}$. For any $(z,p)\in U_a^+$, there is a unique $\lambda_a(z,p)\in \bR_{>0}$ such that $\lambda_a(z,p)\cdot (z,p)\in (\mu^+_N)^{-1}(a)$ under the action $\lambda_a\cdot(z,p)=(\lambda_a z, \lambda_a^{-n-1}p)$ defined by Equation \eqref{eq: Nc action KPn}.  This defines the following deformation retraction (same as Equation \eqref{eq: retraction KPn})
$$
R^+_a: (\bC^{n+1}-\{0\})\times \bC \rightarrow {(\mu_N^+)}^{-1}(a),\quad R^+_a(z,p) =   ( \lambda_a(z,p) z, \lambda_a(z,p)^{-n-1}p). 
$$

\subsubsection{The $\bT^{n+1}$-action on $K_{\bP^n}$ and its moment map.} Corresponding to the $A^+$ in Example \ref{ex: KPn}, we have the map $\alpha_\bC^+: \bT^{n+1}_\bC\to \bT^{n+2}_\bC$ given by 
  \[\alpha_\bC^+(t_1,\ldots,t_{n+1})=(t_1,\ldots t_n, 1, t_{n+1}),
\] 
and  we get a $\bT^{n+1}_\bC$-action on $\bC^{n+2}$ via $\alpha_\bC^+(\bT_\bC^{n+1})$ by
\[(t_1,\ldots, t_{n+1})\cdot (z_1,\ldots, z_{n+1},p) = (t_1 z_1,\ldots, t_n z_n, z_{n+1}, t_{n+1}p).
\]
As usual, this restricts to  a Hamiltonian $\bT^{n+1}$-action on $\bC^{n+2}$ in coordinates $z,p$ with the standard symplectic form $\frac{i}{2}\left(\sum_{k=1}^{n+1}dz_k\wedge d\bar{z}_k +dp\wedge d\bar{p}\right)$ and moment map (up to a constant)
$$\mu=(\alpha^+)^*\circ \mu_{\bT^{n+1}}: \bC^{n+2}\rightarrow \bR^{n+1}, \quad \mu(z_1,\ldots,z_{n+1},p) = \frac{1}{2}(|z_1|^2,\ldots, |z_n|^2, |p|^2).$$ 

Because this $\bT^{n+1}_\bC$-action on $U_a^+$ commutes with the 
$\bC^*$-action on $U_a^+$, we get that   $\bT^{n+1}_\bC$ acts on the quotient $K_{\bP^n}=U_a^+/\bC^*$ by
\begin{equation}\label{eq: Tn+1 action KPn}
(t_1,\ldots,t_n, t_{n+1})\cdot [z_1:\ldots: z_{n+1}:p]= [t_1 z_1:\ldots : t_n z_n: z_{n+1}:t_{n+1}p].
\end{equation}
This restricts to a Hamiltonian $\bT^{n+1}$-action on $(K_{\bP^n}, \omega_a)$.  Using Equation \eqref{eq: mu a pt2}, we find the moment map $\mu_a:K_{\bP^n} \to \bR^{n+1} $ is given by 
\begin{equation}\label{eq: moment map KPn}
\begin{array}{ll}
\mu_a([z_1,\ldots,z_{n+1},p])
&=\; \mu \circ R^+_a(z_1,\ldots,z_{n+1},p) \\
&=\; \frac{1}{2} (\lambda_a(z,p)^2|z_1|^2,\ldots, \lambda_a(z,p)^2 |z_n|^2, \lambda_a(z,p)^{-2(n+1)}|p|^2).
\end{array}
\end{equation}

\subsubsection{Computing $\lambda_a(z,p)^2$.} Now we try to compute $\lambda_a^2(z,p)$ in order to obtain a more explicit formula for $\mu_a$ in terms of the homogeneous coordinates $(z,p)$. Let 
$|z|^2=\sum_{k=1}^{n+1}|z_k|^2$. By setting 
$\mu_N^+(z_1,\ldots,z_{n+1},p) $ equal to $a$, 
we get that $x:= \lambda_a(z,p)^2|z|^2$ satisfies  
$$
x^{n+2}  - 2a x^{n+1} -(n+1)|z|^{2(n+1)}|p|^2 =0.
$$
So 
$$
|z|^{2(n+1)}|p|^2 = \frac{1}{n+1} (x^{n+2}-2a x^{n+1}) =:f(x).
$$
$$
f'(x) = \frac{n+2}{n+1} x^{n+1} - 2a x^n 
$$
$f'(x)>0$ for $x> \frac{2(n+1)a}{n+2}$. So there is a smooth function 
$$
g: \left(-\frac{2a}{(n+1)(n+2)} \left(\frac{2a(n+1)}{n+2}\right)^{n+1}, \infty\right) \longrightarrow \left(\frac{2(n+1)a}{n+2}, \infty\right) 
$$
which is the inverse of 
$$
f: \left(\frac{2(n+1)a}{n+2}, \infty\right) \longrightarrow \left(-\frac{2a}{(n+1)(n+2)} \left(\frac{2a(n+1)}{n+2}\right)^{n+1}, \infty\right).
$$
Furthermore, $g'(t)>0$, and 
$$
g(0)=2a,   \quad \lim_{t\to +\infty} g(t)= +\infty.
$$
Because $g$ is the inverse of $f$, we get that $x=g(|z|^{2(n+1)}|p|^2)$, and by definition 
\[
\lambda_a(z,p)^2=\frac{x}{|z|^2}=\frac{g(|z|^{2(n+1)}|p|^2)}{|z|^2}.
\]

Using the above formula for $\lambda_a(z,p)^2$ and continuing Equation \eqref{eq: moment map KPn} for the moment map $\mu_a$, we get 
\begin{equation}\label{eq: moment map KPn in g}
 \begin{array}{ll}
&\;\mu_a([z_1,\ldots,z_{n+1},p])
= \mu \circ R^+_a(z_1,\ldots,z_{n+1},p) \\\\
=&\; \left(\frac{g(|z|^{2(n+1)}|p|^2)}{2} \frac{|z_1|^2}{|z|^2},\ldots, 
\frac{g(|z|^{2(n+1)}|p|^2)}{2} \frac{|z_n|^2}{|z|^2}, \frac{1}{2(n+1)}(g(|z|^{2(n+1)}|p|^2) -2a) \right).
\end{array}
\end{equation}
Recall that $g(0)=2a$, so  
$$
\mu_a([z_1,\ldots, z_{n+1},0]) =  \frac{a}{|z_1|^2+\cdots + |z_{n+1}|^2} (|z_1|^2,\ldots, |z_n|^2,0).
$$
So when $p=0$ we obtain the same moment map as above in Example \ref{ex: s_quot_mom_map} for $\bP^n$. 

\subsubsection{Explicit formula for when $n=1$.} We next derive an explicit expression of the inverse function $g$ of $f$  when $n=1$. In this case, $x=\lambda(z,p)^2|z|^2>0$ satisfies
\begin{equation}\label{eq: cubic}
x^3 -2a x^2 -2|z|^4|p|^2 =0.
\end{equation}
If $p\neq 0$, let $y=x^{-1}$. Then $y$ is the unique real root of the following depressed cubic equation. 
$$
y^3 + \frac{a}{|z|^4|p|^2} y -\frac{1}{2|z|^4|p|^2} =0.
$$
By Cardano's formula,
\begin{eqnarray*}
y &=& \sqrt[3]{\frac{1}{4|z|^4|p|^2} + \sqrt{\frac{1}{16|z|^8 |p|^4}+\frac{a^3}{27|z|^{12}|p|^6}} } + 
 \sqrt[3]{\frac{1}{4|z|^4|p|^2} - \sqrt{\frac{1}{16|z|^8 |p|^4} + \frac{a^3}{27|z|^{12}|p|^6}} } \\
&=&  \frac{\sqrt{a}}{\sqrt{3}|z|^2|p|}\left( \sqrt[3]{ \sqrt{1 + \frac{27 |z|^4|p|^2}{16 a^3} } + \frac{ 3\sqrt{3} |z|^2|p|}{4 a\sqrt{a}} }
-\sqrt[3]{ \sqrt{1 + \frac{27 |z|^4|p|^2}{16 a^3} } - \frac{ 3\sqrt{3} |z|^2|p|}{4 a\sqrt{a}} } \right) \\
&=& \frac{1}{2a} \times  \frac{3}{    \left(\sqrt{1 + \rho_a(z,p)^2 } +\rho_a(z,p) \right)^\frac{2}{3}
 + \left(\sqrt{1 + \rho_a(z,p)^2 } - \rho_a(z,p) \right)^\frac{2}{3} +1  } 
\end{eqnarray*}
where 
$$ 
\rho_a(z,p)=\displaystyle{ \frac{3\sqrt{3} |z|^2|p|}{4a\sqrt{a}} }
$$ 
is invariant under the $\bC^*$-action on $(\bC^2-\{0\})\times \bC$, and $\rho_a(z,0)=0$.

Using the relations that $x=g(|z|^{2(n+1)}|p|^2)$ (here  $n=1$) and $x=y^{-1}$, we get that 
\begin{eqnarray*}
 g(|z|^4 |p|^2) &=&   \frac{2a}{3}\left( \left(\sqrt{1 + \rho_a(z,p)^2 } +\rho_a(z,p) \right)^\frac{2}{3}
 + \left(\sqrt{1 + \rho_a(z,p)^2 } - \rho_a(z,p) \right)^\frac{2}{3} +1 \right)\\
&=& 2a+  \frac{2a}{3} \left( \sqrt[3]{ \sqrt{1 + \frac{27 |z|^4|p|^2}{16 a^3} } + \frac{ 3\sqrt{3} |z|^2|p|}{4 a\sqrt{a}} }-\sqrt[3]{ \sqrt{1 + \frac{27 |z|^4|p|^2}{16 a^3} } - \frac{ 3\sqrt{3} |z|^2|p|}{4 a\sqrt{a}} } \right)^2. 
\end{eqnarray*}

We can then substitute this into Equation \eqref{eq: moment map KPn in g} to get 
\begin{align*}
 &\mu_a([z_1,z_2,p]) =  \left( \frac{g(|z|^4|p|^2)}{2} \frac{|z_1|^2}{|z_1|^2+|z_2|^2}, \frac{g(|z|^4|p|^2)-2a}{4} \right)=\\
 =& \Bigg(\frac{a |z_1|^2}{|z|^2} \times \frac{ \left(\sqrt{1 + \rho_a(z,p)^2 } +\rho_a(z,p) \right)^\frac{2}{3}
 + \left(\sqrt{1 + \rho_a(z,p)^2 } - \rho_a(z,p) \right)^\frac{2}{3} +1 } {3}, \\
 &  \frac{a}{2} \times  \frac{ \left(\sqrt{1 + \rho_a(z,p)^2 } +\rho_a(z,p) \right)^\frac{2}{3}
 + \left(\sqrt{1 + \rho_a(z,p)^2 } - \rho_a(z,p) \right)^\frac{2}{3} -2 } {3}\Bigg).
\end{align*}
The moment polytope, which is the image of this moment map $\mu_a$ is shown in Figure (\ref{fig: KP1 straight}). 

So far in this section, the $\bT^2$-action on $K_{\bP^1}$ is given by Equation \eqref{eq: Tn+1 action KPn}, which arose from the choice of $B^+$ and $A^+$ given in Example \ref{ex: KPn}.  If we change the basis of $\bT^2$ and use the $B^+$ and $A^+$ given in Example \ref{ex:KPn_choice1}, we get the $\bT^2$-action on $K_{\bP^1}$ given by
$$
(t_1,t_2)\cdot [z_1:z_2:p] = [t_1 z_1: z_2: t_1^{-1} t_2 p],
$$
then the moment map becomes
\begin{equation}
\begin{aligned}
\mu'_a([z_1,z_2,p])& = \left(\frac{g(|z|^4|p|^2)}{2} \frac{|z_1|^2}{|z_1|^2+|z_2|^2} - \frac{g(|z|^4|p|^2)-2a}{4}, \frac{g(|z|^4|p|^2)-2a}{4} \right)
\end{aligned}
\end{equation}
and its image is shown in Figure (\ref{fig: KP1}).
In particular, when $p=0$, 
$$
\mu'_a([z_1,z_2,0]) = \left(  \frac{a|z_1|^2}{|z_1|^2+|z_2|^2}, 0 \right).
$$
Note that the first component is the moment map for $\bP^1$, and indeed its image is the line segment connecting $(0,0)$ and $(a,0)$ in Figure (\ref{fig: KP1}) and Figure (\ref{fig: KP1 straight}).

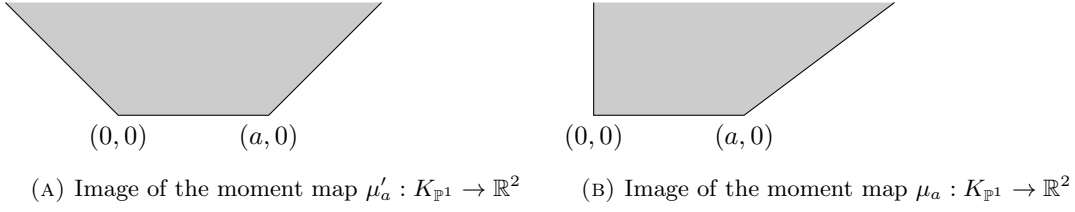
\begin{figure}[h]
\centering
\begin{subfigure}[b]{0.45\textwidth}
\begin{tikzpicture}
\draw[fill=gray!40] (-1.5,1.5) -- (0,0) -- (2,0) -- (3.5,1.5);
\node[below]  (0,0) {$(0,0)$};
\node[below] at (2,0) {$(a,0)$};
\end{tikzpicture}
\caption{Image of the moment map $\mu_a':K_{\bP^1}\to \bR^2$}
\label{fig: KP1}
\end{subfigure}
\begin{subfigure}[b]{0.45\textwidth}
\begin{tikzpicture}
\draw[fill=gray!40] (0,1.5) -- (0,0) -- (2,0) -- (4,1.5);
\node[below]  (0,0) {$(0,0)$};
\node[below] at (2,0) {$(a,0)$};
\end{tikzpicture}
\caption{Image of the moment map $\mu_a:K_{\bP^1}\to \bR^2$}
\label{fig: KP1 straight}
\end{subfigure}
\caption{Two moment polytopes for $K_{\bP^1}$}
\label{fig: KP1 total}
\end{figure}

\end{example}

\section{K\"ahler potential}\label{sec: kahler}
There is a natural K\"ahler potential that toric varieties admit from toric geometry, as follows. With the data of a moment polytope, they come equipped with an ample line bundle defined by a sum of the toric divisors weighted according to the moment polytope; hence they admit a K\"ahler potential induced from sections of that line bundle, see \cite[Proposition 4.3.3]{cls} and \cite[Example 4.1.2(i)]{huy} for details. The K\"ahler potential in this case is a function of the complex affine coordinates. However, many calculations become simpler in the moment map coordinates. We saw above the existence of a symplectic form $\omega_a$ coming from symplectic reduction. We will see in this section that the Legendre transform of the K\"ahler potential one obtains from symplectic reduction is computable. Although the two symplectic forms are different in general, they can have similar properties such as being in the same K\"ahler class.

More specifically, the focus of this chapter is to find a K\"ahler potential $F$ over $\mathring{{\mathrm{\Delta}}}$ for $\omega_a$ in terms of the affine toric coordinates, as well as its Legendre transform $G$ in terms of the action-angle coordinates, specifically the moment map coordinates. First we recall some notation. The diagram below is a summary of the identifications given by Equation \eqref{eq: action angle}, Equation \eqref{eq: cx toric}, and Equation \eqref{eq: moment map identification}, and the notations introduced in Definition \ref{def: toric coordinates},
\begin{equation}
\begin{array}{ccccccc}
\mathfrak t_\bC^n    & \xrightarrow{\exp} & \bT^n_\bC                    &  \cong & \mu^{-1}_a(\mathring{{\mathrm{\Delta}}})&\cong & \mathring{{\mathrm{\Delta}}}\times \bT^n\\
\vin                           &                              & \vin                               &            &        				         &	     &\vin\\
u=(u_j=x_j+i\theta_j)_{j=1}^n& \mapsto                & (e^{u_1},\ldots, e^{u_n})&           &					         &	     & (\xi, e^{i\theta})=(\xi_j, e^{i\theta_j})_{j=1}^n\\
 & & \rotatebox{90}{$=$} & & &&\\
    				&             			 & (t_1,\ldots t_n)&  & \longmapsto & & \mu_a(t)=\xi\in \mathring{{\mathrm{\Delta}}}

\end{array}.
\end{equation}
Note that since $\bT^n$ embeds in $\bT^n_\bC$ as $U(1)^n$, $\theta=(\theta_1,\ldots, \theta_n)$ is a coordinate on the Lie algebra of $\bT^n_\bC$.  The above diagram identifies $\mu_a^{-1}(\mathring{{\mathrm{\Delta}}})$ with $\bT^n_\bC$ and $\mathring{{\mathrm{\Delta}}}\times \bT^n$, and hence endows $\mu^{-1}_a(\mathring{{\mathrm{\Delta}}})$ with the complex toric coordinates $t$ and the action-angle coordinates $(\xi,\theta)$, respectively.  

{\bf K\"ahler potential $F$.} By \cite[Theorem 4.3]{G94}, there exists a $\bT^n$-invariant K\"ahler potential $F$ on $(\bC^*)^n \subset M$ such that ${\omega_a}|_{(\bC^*)^n}=2i \dd \overline{\dd}F$. Contracting the symplectic form with $\dd/\dd \theta$, using that $F$ is not a function of $\theta$ since it is $\bT^n$-invariant (that is, a function $\bR^n \to \bR$ on the norms of the affine toric coordinates of $(\bC^*)^n$),  plugging into the lefthand side of $\iota_{\dd/\dd \theta_j}\omega_a = -d\mu_{a,j}$, and integrating we see that $F:\bR^n \to \bR$ satisfies 
\begin{equation}\label{eq: moment map and F}
\mu_a(t_1, \ldots, t_n)=\mu_a( e^{x_1+i \theta_1},\ldots, e^{x_n+i \theta_n}) = \left( \frac{\partial F}{\partial x_1} (x),\ldots, \frac{\partial F}{\partial x_n}(x) \right). 
\end{equation}
Therefore by \cite[Theorem 3.3]{G94} and the sentence following it, the map $x \mapsto  \left( \frac{\partial F}{\partial x_1} (x),\ldots, \frac{\partial F}{\partial x_n}(x) \right)$
is a diffeomorphism from $\bR^n$ to the interior $\mathring{{\mathrm{\Delta}}}$ of the moment polytope ${\mathrm{\Delta}}$. In particular, let $\xi_j = \frac{\partial F}{\partial x_j}.$ Then
\begin{equation} \label{eq: Hessian F}
\frac{\partial \xi_j}{\partial x_k} = \frac{\partial^2 F}{\partial x_j \partial x_k}.
\end{equation}

{\bf Legendre transform $G$.} In the other direction, there is a function $G: \mathring{{\mathrm{\Delta}}}\to \bR$ such that the inverse map of the diffeomorphism
\begin{equation}\label{eq: cx to moment}
\bR^n\longrightarrow \mathring{{\mathrm{\Delta}}}, \quad (x_1,\ldots, x_n) \mapsto  \left( \frac{\partial F}{\partial x_1} (x),\ldots, \frac{\partial F}{\partial x_n}(x) \right)
\end{equation}
is
\begin{equation}\label{eq: moment to cx}
\mathring{{\mathrm{\Delta}}}\longrightarrow \bR^n,\quad (\xi_1,\ldots, \xi_n) \mapsto  \left( \frac{\partial G}{\partial\xi_1} (\xi),\ldots, \frac{\partial G}{\partial\xi_n}(\xi) \right)=(x_1,\ldots, x_n).
\end{equation}
Specifically, $G$ is the \emph{Legendre transform} of $F$,  namely (up to a constant) 
\begin{equation}\label{eq: Legendre}
F(x)+G(\xi)=\left<x,\xi \right>
\end{equation}
so that differentiating Equation \eqref{eq: Legendre} with respect to $\xi$ we see that $\frac{\partial G}{\partial\xi_i}=x_i$. Furthermore
\begin{equation}\label{eq: Hessian G}
\frac{\partial x_j}{\partial \xi_k} = \frac{\partial^2 G}{\partial \xi_j \partial \xi_k}.
\end{equation}

From the relations in Equations (\ref{eq: Hessian F}) and (\ref{eq: Hessian G}), we see that the Hessian of $G$ is the inverse of the Hessian of $F$. Also from Equations (\ref{eq: Hessian F}) and (\ref{eq: Hessian G}), we see that the complex structure in the $(x_j, \theta_j)$ coordinates is the standard 
$\left(\begin{array}{c | c}
0 & -\bI_n\\
\hline
\bI_n&0 
\end{array}
\right)$, hence the complex structure in the $(\xi_j, \theta_j)$ coordinates is $\left(\begin{array}{c | c}
0 & -[\partial_j \partial_k F]\\
\hline
[\partial_j\partial_k G]&0 
\end{array}
\right)$. We now discuss the K\"ahler form and metric.

Recall in Equation \eqref{eq: action angle}, $\mathring{{\mathrm{\Delta}}}\times \bT^n$ with the symplectic form $\sum_{j=1}^n d\xi_j \wedge d\theta_j$ is symplectomorphic to $\mu_a^{-1}(\mathring{{\mathrm{\Delta}}})$ with the symplectic form $\omega_a$. We explain why the symplectic form is standard in the action-angle coordinates.

\begin{lemma} The symplectic form for $M$ can be written (over the interior of the moment polytope) in the action-angle coordinates as $$\omega_a=\sum\limits_{j=1}^n d\mu_{a,j} \wedge d\theta_j.$$
\end{lemma}

\begin{proof} We impose conditions on the coefficients of terms in $\omega_a$. There are three conditions: $\iota_{\dd/\dd \theta_j}\omega_a = -d\mu_{a,j}$, fibers of the moment map are Lagrangian with tangent space spanned by the $\dd/\dd \theta_j$, and $\omega_a$ is K\"ahler so compatible with the almost complex structure $J$ induced from multiplication by $i$ on the affine coordinates. The first condition is the definition of the moment map, where $\partial/\partial \theta_j$ is the infinitesimal action of $t_j$ on $M$ defined by taking the derivative at the identity as in Equation \eqref{eq: A} and Equation \eqref{eqn:alphaz}, see also \cite[Definition 22.1]{Ca64}. That is, $\theta_{j=1,\ldots n}$ are the imaginary parts of the complex coordinates on the Lie algebra of $\bT^n_\bC$ and $\mu,\theta$ are the action-angle coordinates. In particular, $\dd/ \dd \theta_j$ corresponds to rotating the $j^{\textrm{th}}$ affine toric coordinate $y_j$ by $t_j$, as described in the paragraph above Equation \eqref{eq: alpha affine chart}.

As discussed in the paragraph following Corollary \ref{coro: polytope}, the second condition is true by the Arnold-Liouville theorem which is stated, for example, in \cite[Theorem 18.12]{Ca64}. In more detail, $\bT^n$ is abelian so has trivial Lie bracket and
$$\omega_a \left(\frac{\dd}{\dd \theta_i} , \frac{\dd}{\dd \theta_j} \right) = \left[\frac{\dd}{\dd \theta_i}, \frac{\dd}{\dd \theta_j}\right] =0= \{\mu_{a,i},\mu_{a,j}\}
$$ 
where $\{,\}$ denotes the Poisson bracket. But the tangent space of a $\mu_a$-fiber is precisely the kernel of $d\mu_a$, so by the first condition:
$$T \mu_a^{-1}(c) = \ker(d\mu_a) = \bigcap_{j=1}^n \ker ( \iota_{\dd/\dd \theta_j}\omega_a) \supseteq span\{\dd /\dd \theta_j\}_{j=1}^n.
$$
Since a fiber has dimension $n$, the last containment is equality. The fibers of the moment map are then $\bT^n$-orbits of the $\bT^n$-action $(e^{i \theta_1},\ldots, e^{i \theta_n}) \cdot (y_1,\ldots,y_n) =(e^{i \theta_1}y_1,\ldots, e^{i \theta_n}y_n) $, where $(t_1,\ldots,t_n)=(e^{i \theta_1},\ldots, e^{i \theta_n})$, therefore as $\omega_a$ is toric invariant, $\omega_a$ restricted to a fiber is 0. That is, the fibers are Lagrangian. 

For the third condition, we saw that multiplication by $i$ on $x \in \mathfrak{t}^n_\bC$ induces, under the Jacobian transformation between $\xi$ and $x$, $J=\left(\begin{array}{c | c}
0 & -[\partial_j \partial_k F]\\
\hline
[\partial_j\partial_k G]&0 
\end{array}
\right)$ in the $(\xi_j, \theta_j)$ coordinates, where $[\partial_j \partial_k F] = [\partial_j\partial_k G]^{-1}$ hence $J^2 = -\mathbb{I}_{2n}$. 

Now we prove the statement of the lemma. Let $\omega_a=\sum\limits_{j,k=1}^n a_{jk} d\mu_{a,j} \wedge d\theta_k + \sum\limits_{j \leq k, j,k=1}^n b_{jk} d\theta_j \wedge d\theta_k+c_{jk} d\mu_{a,j} \wedge d\mu_{a,k}.$ Then $\omega_a|_{span\{\dd /\dd \theta_j\}_{j=1}^n}  \equiv 0 \implies b_{jk}=0$ and by the definition of the moment map, $\iota_{\dd/\dd \theta_k}\omega_a = \sum\limits_{j=1}^n -a_{jk} d\mu_{a,j}= -d\mu_{a,k} \implies a_{kk} = 1, a_{jk}=0\;\; \forall j \neq k$. Wrapping up, compatibility with $J$ means that $J^*\omega_a =\omega_a$ thus
\begin{equation*}\label{eq:omega_aa_coords}
    \begin{split}
        & \sum\limits_{j=1}^n  d\mu_{a,j} \wedge d\theta_j + \sum\limits_{\substack{ j,k=1\\ j\leq k}}^n c_{jk}' d\theta_j \wedge d\theta_k =\sum\limits_{j=1}^n  d\mu_{a,j} \wedge d\theta_j + \sum\limits_{\substack{j,k=1\\ j\leq k}}^n c_{jk}d\mu_{a,j} \wedge d\mu_{a,k}\\
   & \therefore c_{jk} =0 \implies \omega_a=\sum\limits_{k=1}^n d\mu_{a,k} \wedge d\theta_k.
    \end{split}
\end{equation*}
where $c_{jk}'$ is a function of $[\partial_j\partial_k G]$. \qed
\end{proof}

Below we will use $\omega_a$ to interchangeably denote the symplectic forms on $\bT^n_\bC$ and $\mathring{{\mathrm{\Delta}}}\times \bT^n$, which we view as $\omega_a$ written in two different set of coordinates. For $K_M$ the same holds true replacing $n$ with $n+1$. 

\begin{corollary}
\begin{equation}\label{eqn:M-symplectic-form}
\omega_a =\sum_{j=1}^n d\xi_j \wedge d\theta_j= \sum_{j,k=1}^n \frac{\partial^2 F}{\partial x_j\partial x_k} dx_j\wedge d\theta_k 
= \frac{i}{2} \sum_{j,k=1}^n \frac{\partial^2 F}{\partial x_j\partial x_k} du_j\wedge d\bar{u}_k .
\end{equation}
\end{corollary}
In the above, we have written $\omega_a$ both in terms of the complex coordinates $u_j=x_j+i\theta_j$ via the identification of $\mu^{-1}(\mathring{{\mathrm{\Delta}}})\cong \bT^n_\bC$, and in terms of the action-angle coordinates $(\xi_j, \theta_j)$ via the identification of $\mu^{-1}(\mathring{{\mathrm{\Delta}}})\cong \mathring{{\mathrm{\Delta}}}\times \bT^n$.

\begin{corollary}\label{def: Riem metric}
The Riemannian metric compatible with  $\omega_a$ and the complex structure $J$ is
\begin{equation}\label{eqn:M-metric} 
g_a = \sum_{j,k=1}^n \frac{\partial^2 F}{\partial x_j\partial x_k} (dx_j dx_k + d\theta_j d\theta_k)= 
 \sum_{j,k=1}^n \frac{\partial^2 G}{\partial \xi_j\partial \xi_k} d\xi_j d\xi_k 
 + \sum_{j,k=1}^n \frac{\partial^2 F}{\partial x_j\partial x_k} d\theta_j d\theta_k,
\end{equation}
which we've also written both in terms of the complex coordinates and the action-angle coordinates.
\end{corollary}

To recap, let  $\omega$ be a $\bT^n$-invariant K\"{a}hler form on $M$. Then
the action of $\bT^n$ on $(M,\omega)$ is Hamiltonian with a moment map $\mu:M\to \bR^n$ which is unique up to addition of 
a constant vector in the target  $\bR^n$. The image ${\mathrm{\Delta}}$ of $\mu$ is a convex polytope known as the moment polytope, and the moment map
determines smooth functions $F:\bR^n \to \bR$ and $G:\mathring{{\mathrm{\Delta}}} \to \bR$ up to adding constants, as described above.  Let $g$ be the Riemannian metric on
$M$ determined by the symplectic structure $\omega$ and the complex structure on $M$. Then $g$ is given by the right hand side of
\eqref{eqn:M-metric}. When $\omega$ is the symplectic structure coming from symplectic reduction, $G(\xi)$ is given explicitly in \cite{G94}.  Below we summarize a way of deriving Guillemin's formula for $G(\xi)$ due to Calderbank-David-Gauduchon \cite{CDG}. 
\begin{lemma}
The Legendre transform $G(\xi)$ of the potential function $F(x)$ is 
$$G(\xi)=\frac{1}{2}\sum_{j=1}^dL_j(\xi) \log L_j(\xi).
$$
\end{lemma}

\begin{proof} The standard flat metric on $\bC^d$ written in polar coordinates, $z_j=r_je^{i\varphi_j}$, is
\begin{equation}\label{eq: metric_g}
g = \sum_{j=1}^d dz_jd\bar z_j=\sum_{j=1}^d dr_j^2+r_j^2d\varphi_j^2 = \sum_{j=1}^d \frac{d\mu_j^2}{2\mu_j}+2\mu_j d\varphi_j^2 = \tilde g +\sum_{j=1}^d 2\mu_j d\varphi_j^2,
\end{equation}
where $\mu_j=\frac{1}{2}r_j^2=\frac{1}{2}|z_j|^2$ is the $j^{\textrm{th}}$ component of the moment map $\mu_{\bT^d}$ given by Equation (\ref{eq: Cd moment map}). This metric is compatible with the standard symplectic form $\omega$ of Equation (\ref{eq: Omega Cd}) and the standard complex structure on $\bC^d$. The first term in the sum above in Equation \eqref{eq: metric_g} in polar coordinates, denoted $\tilde g$, is the metric on the image of the moment map, and the 2nd term is the metric on the torus fiber. Note that $\tilde g$ can be written in terms of the Hessian of a potential function $\widetilde G$ as follows
$$
\tilde g=\sum_{j=1}^d \frac{d\mu_j^2}{2\mu_j}= \sum_{j,k=1}^d \frac{\partial^2 \widetilde G}{\partial\mu_j\partial\mu_k}d\mu_jd\mu_k, \mbox{ where } \widetilde G =\sum_{j=1}^d \mu_j\log \mu_j.
$$
Recall that the metric $g_a$ in  Equation (\ref{eqn:M-metric}) is the canonical metric defined by symplectic reduction from $\bC^d$ with the flat metric $g$.   The first term of $g_a$ on the right-most expression of Equation \eqref{eqn:M-metric}, which is 
$$
\tilde g_a =\sum_{j,k=1}^n \frac{\partial^2 G}{\partial \xi_j\partial \xi_k} d\xi_j d\xi_k,
$$
is a Riemannian metric on $\mathring{{\mathrm{\Delta}}}$ (the unique Riemannian metric such that $\mu_a: (\bT^n_\bC ,g_a)\to (\mathring{{\mathrm{\Delta}}}, \tilde g_a)$ is a Riemannian submersion).  By construction and because $L(\xi)$ defined in Equation (\ref{eq: L}) satisfies $L(\xi)_j=\frac{1}{2}|z_j|^2=\mu_j$, we have 
\begin{equation}\label{eq: G}
\tilde g_a=L^*\tilde g = \sum_{j=1}^d \frac{dL_j^2}{2L_j}, \quad   G= L^*\widetilde G=\frac{1}{2}\sum_{j=1}^dL_j \log L_j.
\end{equation}
\qed
\end{proof}

\begin{example}[$\bP^n$]  Recall $\alpha_\bC$ from Example \ref{ex: s_quot_mom_map} and $\mu_a$ from Equation \eqref{eq:Pn_mom_map}, leading to the conclusion at the end of Example \ref{ex: Pn chart} that
$$(\xi_1,\ldots, \xi_n)=\frac{a(e^{2x_1},\ldots, e^{2x_n})}{ e^{2x_1} + \cdots + e^{2x_n} +1}
$$
and
$$
(x_1,\ldots, x_n)= \frac{1}{2} \Big(\log\xi_1 - \log (a-\xi_1-\cdots-\xi_n), \ldots, \log \xi_n-\log (a-\xi_1-\cdots-\xi_n)\Big).
$$
Thus, integrating, we find by Equation \eqref{eq: moment map and F} and Equation \eqref{eq: moment to cx}:
$$
F(x_1,\ldots,x_n) = \frac{a}{2} \log ( e^{2x_1} +\cdots + e^{2x_n} +1)=\frac{a}{2}\log(|t_1|^2+\cdots +|t_n|^2+1),
$$
$$
G(\xi_1,\ldots,  \xi_n)=\frac{1}{2}\Big(\xi_1\log\xi_1+\cdots+ \xi_n\log \xi_n+(a-\xi_1-\cdots-\xi_n)\log(a-\xi_1-\cdots-\xi_n)\Big).
$$
\qed
\end{example} 

\noindent These calculations apply for $K_M$ as well, replacing $n$ with $n+1$ everywhere.

Note that if $\omega$ is the symplectic structure coming from symplectic reduction, and if we know the moment polytope ${\mathrm{\Delta}}$ to begin with, then according to Equations (\ref{eq: G}) and (\ref{eq: L}), we can readily write down the explicit formulas $G(\xi)$ from the combinatorics of the moment polytope ${\mathrm{\Delta}}$.   Below we describe two perspectives for writing down $F(x)$. Both are less straightforward for the exact same reason, if not impossible, compared to finding $G(\xi)$.
\begin{enumerate}
\item  If we could compute the moment map, $\mu_a$, then we could find $F(x)$ via Equation (\ref{eq: moment map and F}).  However, as we saw in Section \ref{subsec: KP1} (also mentioned in Example \ref{ex: CP2(3)} for $\bP^2(3)$), the moment map can be extremely complicated. In Example \ref{ex: KPn}, our ability to explicitly compute the moment map in the $n=1$ case boils down to being able to solve the cubic polynomial (Equation (\ref{eq: cubic})).  For $n>2$, it's not possible to compute explicitly, at least not using the method of Section \ref{subsec: KP1}.
\item Once we know $G(\xi)$, we can find $F(x)$ via Legendre transform, i.e.~Equations (\ref{eq: moment to cx}) and (\ref{eq: Legendre}).  However, in order to write down $F(x)$ using Equation (\ref{eq: Legendre}), we need to write $\xi$ in terms of $x$, and $\xi(x)$ is exactly the moment map, which is complicated as we just discussed.  One might also think to use Equation (\ref{eq: moment to cx}), which is a system of $n$ nonlinear equations in $(x_1,\ldots, x_n)$ and $(\xi_1,\ldots, \xi_n)$.  Similarly, it's complicated, if not impossible, to use that to find $\xi(x)$.
\end{enumerate}

\begin{remark}[Comparing with the K\"ahler structure defined by an ample line bundle]\label{rmk:omega_comparison} The toric manifolds from Delzant's construction are projective, so there is another natural K\"ahler form obtained from using sections of an ample line bundle. This remark addresses the following question: \textit{is the Guillemin construction of a symplectic form from symplectic reduction the same as the K\"ahler structure from toric geometry given an ample line bundle?} They are different in general, however they are in the same K\"ahler class by \cite[Equation (1.7)]{G94}.  We first consider the 1-dimensional case. The ample line bundles on 
$\bP^1$ are $\cO_{\bP^1}(k)$, where $k$ is a positive integer, and the space of sections of $\cO_{\bP^1}(k)$ can be identified with the space of homogeneous polynomials in two variables $z_1, z_2$ of degree $k$:
$$
H^0(\bP^1,\cO_{\bP^1}(k)) =\bigoplus_{m=0}^k \bC z_1^m z_2^{k-m}. 
$$
Let $\bC^*$ act on $z_1$ and $z_2$ with weights 1 and 0 respectively. The moment map in Section 4.2 of Fulton \cite{fulton} is given by
$$
\hat{\mu}_k([z_1:z_2]) =  \frac{ \sum_{m=0}^k m  |z_1^m z_2^{k-m}|^2 }{ \sum_{m=0}^k |z_1^m z_2^{k-m}|^2}. 
$$
The image of $\hat{\mu}_k$ is the closed interval $[0,k]\subset \bR$.   

When $[z_1:z_2]$ are both non-zero, define $ x =\log|z_1/z_2|$. Then 
letting $a=k$ in $\mu_a$ from Example \ref{ex: s_quot_mom_map} we get,
$$
\mu_k([z_1:z_2]) = \frac{dF_k}{dx}(x), \quad \hat{\mu}_k([z_1:z_2]) = \frac{d\hat{F}_k }{dx}(x), 
$$
where the hat denotes the potential obtained from sections of the ample line bundle $\cO(k)$ and without the hat means we are using symplectic reduction on $S^3(\sqrt{2k})/S^1$:
$$
F_k (x) = \frac{k}{2}\log (1+e^{2x}),\quad   \hat{F}_k (x) = \frac{1}{2} \log\left(\sum_{m=0}^k e^{2mx}\right).
$$
In particular, $F_1(x)= \hat{F}_1(x)$
however they are different for larger $k$. Here the K\"ahler potential $F_k$ for the symplectic form $\omega_k$  defining the moment map ${\mu}_k: \bP^1\to \bR$ is such that 
$$
\omega_k = d\mu_k\wedge d\theta = \frac{d^2 F_k}{dx^2} dx \wedge d\theta 
$$
and similarly for the case without the hat. Therefore we see that $\hat{\omega}_1 = \omega_1$ since the K\"ahler potentials are equal. When $k>1$ is an integer, $\hat{\omega}_k \neq \omega_k$ but they represent the same 
class in $H^2(\bP^1,\bR)\cong \bR$. Also  note that here $k$ is an integer, so the images of $\hat \mu_k$, which is $[0,k]$, does not gives all possible $[0,a]$, where $a\in \bR$, that $\mu_a$ does.

More generally for higher dimensional projective spaces, let $\hat{\omega}_k$ be the symplectic form on $\bP^n$ determined by  the ample line bundle $\cO_{\bP^n}(k)$, where $k>0$. 
Then $\hat{\omega}_k$ is the pullback of $\omega_1$ on  $\bP^{ \binom{k+n}{k}}$ under the degree $k$  embedding $\bP^n\hookrightarrow \bP^{\binom{k+n}{k}}$. 
In particular,  $\hat{\omega}_1=\omega_1$. When $k>1$, $\hat{\omega}_k\neq \omega_k$ but they represent the same class in $H^{1,1}(\bP^n;\bR)\cong\bR$.
\end{remark}

\begin{remark}[Properties of $\bP^n$]
Of interest to symplectic geometers is the choice of symplectic cohomology class, for example when considering global homological mirror symmetry one considers the K\"ahler cone $\mathcal{K}_M$ of all possible K\"ahler classes on the symplectic manifold $M$. Information about $\mathcal{K}_M$ can be obtained using the Calabi-Yau theorem (stated for example in \cite[Theorem 4.B.19]{huy}), that for compact K\"ahler manifolds there is a bijection between $\mathcal{K}_M$ and the set of K\"ahler forms $\omega$ with $\omega^n=\la \times vol$ for some $\lambda \in \bR_{>0}$. As an example, consider the Fano manifold $\bP^n$. The symplectic area of a projective line
$\bP^1\subset \bP^n$ is
$$
\int_{\bP^1}\omega_a =  2\pi a 
$$
and the cone of K\"ahler classes is one-dimensional corresponding to this parameter $a$. When $a=1$, we obtain an integral cohomology class. 

Furthermore, if $c_1(M)=[\alpha]$ is represented by a closed real $(1,1)$-form $\alpha$ and we've fixed a choice of K\"ahler class $\beta$ on $M$, there is a unique K\"ahler structure $g$ on $M$ so that $\alpha=\mathrm{Ric}(g)$ and $[\alpha]=\omega_g$ the K\"ahler form determined by the metric $g$. This is \cite[Proposition 4.B.21]{huy}. Note that given any two of a metric, complex structure, and symplectic form on $M$ which are compatible, the third is uniquely determined, see \cite[p 29]{huy}. For example, the symplectic form $\omega_a$ and the standard complex structure on $\bP^n$ determine a K\"{a}hler metric $g_a$ on $\bP^n$ which satisfies the K\"{a}hler-Einstein equation:
$$
\mathrm{Ric}(g_a) = \frac{n+1}{a} \omega_a.
$$
\end{remark}

\section{Connection to mirror symmetry} \label{sec: HMS}

We give some context for mirror symmetry in which the above calculations play a role. We start with background.

\subsection{Mirror Symmetry for Calabi-Yau manifolds} 

Mirror symmetry
relates the symplectic (resp. complex) geometry of a Calabi-Yau manifold $X$ to the complex (resp. symplectic) geometry of a mirror Calabi-Yau manifold $\check{X}$ of the same dimension. 
Let $(X,\omega,J)$ be a Calabi-Yau manifold, where $\omega$ is the symplectic structure and $J$ is the complex structure, and
let $(\check{X},\check{\omega}, \check{J})$ be the mirror Calabi-Yau manifold.  Kontsevich's \cite{hms} Homological Mirror Symmetry (HMS) conjecture predicts the following 
equivalences of triangulated categories: 
\begin{equation}\label{eqn:HMSAB} 
D^\pi Fuk(X,\omega)\cong D^b Coh(\check{X},\check{J}),
\end{equation}
\begin{equation}\label{eqn:HMSBA}
\quad D^b Coh(X,J)\cong D^\pi Fuk(\check{X},\check{\omega}),
\end{equation} 
where $D^\pi Fuk$ is the split-closed derived Fukaya category (derived by taking the homotopy category so one obtains a triangulated category) and $D^b Coh$ is the bounded derived category of coherent sheaves.

In \cite{Sherid_CY}, N. Sheridan proved the equivalence \eqref{eqn:HMSAB} when $X$ is a smooth Calabi-Yau hypersurface in the projective space $\bP^n$. In this case, the mirror Calabi-Yau manifold $\check{X}$ is (a crepant resolution of) a Calabi-Yau hypersurface in the orbifold $\bP^n/G$, where $G= (\bZ_{n+1})^{n}$.
 (The $n=2$ case was first proved by  Polishchuk-Zaslow \cite{zp}, and the $n=3$ case was first proved by  P. Seidel \cite{seidel_quartic}.) 
 
 More generally, let $X$ be a  smooth Calabi-Yau hypersurface in  a toric Fano manifold $M$. There is a one-to-one correspondence between projective Gorenstein 
Fano toric varieties and isomorphism classes of reflexive lattice 
polytopes ${\mathrm{\Delta}}$ (see \cite[Theorem 8.3.4]{cls}, and in 2-dimensions, the 16 reflexive lattice polygons are listed in \cite[p 382]{cls}).  The Batyrev mirror $\check{X}$ is (a crepant resolution of) a Calabi-Yau hypersurface in the Gorenstein toric Fano variety  $\check{M}$  defined by  the dual reflexive polytope $\check{{\mathrm{\Delta}}}$ \cite{Batyrev}. When $n>3$, $H^{1,1}(X) = H^{1,1}(M) = H^2(M)$, and the K\"{a}hler moduli of the compact Calabi-Yau $(n-1)$-fold can be identified with the K\"{a}hler moduli of the ambient compact toric Fano manifold $M$.

\subsection{Mirror Symmetry for Landau-Ginzburg Models}  \label{sec:LGmodel}

A smooth  Calabi-Yau hypersurface $X$ in a toric Fano manifold $M$ can be identified with the critical locus of a holomorphic function $W$ on the total space $K_M$ of the canonical line bundle of $M$ as follows. When $M$ is Fano, $K_M^*$ is ample and for a generic section $s$ of $K_M^*$, 
the zero locus $X:=s^{-1}(0)\subset M$ is a smooth anti-canonical divisor hence a manifold. In particular $X$ is a compact Calabi-Yau manifold of complex dimension $n-1$.  By the adjunction formula, $K_X \cong (K_M \otimes \mathcal{O}_M(X))|_X \cong \mathcal{O}_X$, which implies $K_M^*|_X \cong \mathcal{O}_M(X)|_X$. For example, when $M=U_a/N_\bC$ where $U_a\subset \bC^d$,  let $s(z_1,\ldots, z_d)\in \bC[z_1,\ldots,z_d]$ be a polynomial in $z_1,\ldots,z_d$ such that the rational function $\frac{ s(z_1,\ldots,z_d)}{z_1\cdots z_d}$ is invariant under the $N_\bC$-action on $\bC^d$.  Then $s(z_1,\ldots,z_d)$ defines a section of the anti-canonical line bundle $K_M^*$.

Define the holomorphic function $W:K_M\to \bC$ by $W(z,p) = \langle p, s(z)\rangle$, where $z\in M$, $p\in (K_M)_z$ (the fiber of $K_M$ over $z$), and $\langle - , - \rangle$ is the pairing between dual vector spaces. (For example, for the polynomial $s$ in the example of the previous paragraph, we see that $p s(z_1,\cdots, z_d)  \in  \bC[z_1,\ldots,z_d,p]$ is invariant under the $N_\bC$-action on $\bC^{d+1}$ so descends to a well-defined holomorphic function $W: K_M \to \bC$.) The critical locus of $W$ is hence given by 
\begin{align*}
    \begin{split}
        \mathrm{Crit}(W) &= \{[z_1,\ldots,z_d,p] \in K_M:  dW(z_1,\ldots,z_d,p)=0\}\\
        &= \{ [z_1,\ldots,z_d,p] \in K_M:  p = s(z_1,\ldots,z_d) =0\}.
    \end{split}
\end{align*}
Namely, the critical locus of $W$ is exactly $X\subset M\subset K_M$ where the second inclusion is by the zero section: $\mathrm{Crit}(W)=s^{-1}(0)=X$. The pair $(K_M, W)$ is an example of a \emph{Landau-Ginzburg (LG) model} and $W$ is known as the \emph{superpotential}. 

\begin{example}[Fermat surface a the critical locus of a superpotential on $K_{\bP^n}$]\label{ex: Fermat}
Let $s(z_1,\ldots,z_{n+1})\in \bC[z_1,\ldots,z_{n+1}]$ be a homogeneous polynomial of degree $n+1$, for example the Fermat polynomial in $(n+1)$-variables $\sum_{j=1}^{n+1} z_j^{n+1}$. Then $p s$ is invariant under the $\bC^*$-action on $(\bC^{n+1}-\{0\})\times\bC$ and descends to a holomorphic function 
$$
W: K_{\bP^n}\to \bC,\quad [z_1,\ldots,z_n,p]\mapsto ps(z_1,\ldots,z_{n+1})
$$
which, in the case of the Fermat polynomial, has the following critical locus \begin{eqnarray*}
\mathrm{Crit}(W) &=& \{ [z_1,\ldots,z_{n+1},p]\in K_{\bP^n}:  dW(z_1,\ldots,z_{n+1},p) = 0\}  \\
&=&  \{ [z_1,\ldots,z_{n+1},p]\in K_{\bP^n}:  \sum_{j=1}^{n+1} z_j^{n+1} = p = 0\} = s^{-1}(0)\cong X_{n+1},
\end{eqnarray*} 
where $X_{n+1} = \{ [z_1,\ldots,z_{n+1}] \in \bP^n: \sum_{j=1}^{n+1} z_j^{n+1}=0\}$ is the Fermat Calabi-Yau hypersurface in $\bP^n$.\qed
\end{example}

Similarly, the mirror $\check{X}$ can be identified with the critical locus of a holomorphic function $\check{W}$ on the total space $K_{\check{M}}$ of the canonical line bundle of $\check{M}$.  The LG model $(K_{\check{M}}, \check{W})$ is mirror to the LG model $(K_M, W)$. 
A natural formulation of the homological mirror symmetry conjecture in this setting is the following equivalences of triangulated categories:
\begin{equation}\label{eqn:LGHMSAB} 
\mathcal{W}(K_M, W)\cong MF(K_{\check{M}},\check{W}),
\end{equation}
\begin{equation}\label{eqn:LGHMSBA}
MF(K_M,W) \cong \mathcal{W}(K_{\check{M}},\check{W}),
\end{equation} 
where wrapping in non-compact fibers leads one to take $\mathcal{W}(-,-)$ the fiber-wise wrapped Fukaya category (being defined in \cite{AAhypersurf}, see also \cite{jeffs20}) of the LG model, and $MF$ is the category of matrix factorizations.

The B-model $MF(K_M,W=\langle p,s(z)\rangle)$ on the LG model is equivalent to the B-model $D^bCoh(X)$ on the Calabi-Yau hypersurface $X=s^{-1}(0)=\mathrm{Crit}(W)\subset M$, as a consequence of Orlov's generalized Kn\"{o}rrer periodicity theorem \cite{Orlov}.  An $A$-model version of this Kn\"orrer periodicity theorem would be an equivalence between $\mathcal W(K_M,W)$ and $D^\pi Fuk(X)$. Recently \cite{jeffs20} has proven a version of Kn\"orrer periodicity for the A-model that uses the notion of a partially wrapped Fukaya category.  Also see \cite{seidel, WehrheimWoodward}.  With these equivalences, Equations \eqref{eqn:LGHMSAB} and \eqref{eqn:LGHMSBA} is equivalent to Equations \eqref{eqn:HMSAB} and \eqref{eqn:HMSBA} when $X=s^{-1}(0)$ is a smooth CY hypersurface in $M$ that is the critical locus of $W=\langle p, s\rangle$ on $K_M$. 


Now let us consider a different LG model of $(K_M, W=z_1\cdots z_dp)$, again $M^n=\bC^d/\!\!/N$ is a toric Fano manifold of dimension $n$ obtained via symplectic reduction from $\bC^d$ and $z\in \bC^d$ are the homogeneous coordinates.  The critical locus of $W$ in this case is a singular CY hypersurface in $M$ defined by $z_1\cdots z_d=0$, and it is the preimage $\mu^{-1}(\partial {\mathrm{\Delta}})$ of the boundary of the moment polytope for $M$.  The LG model $(K_M, W=z_1\cdots z_dp)$ captures the geometry of this singular $\mathrm{Crit}(W)$ via Kn\"orrer periodicity \cite{Orlov, jeffs20}, and $(K_M, W)$ turns out to be the generalized SYZ \cite{syz} mirror (in the sense of \cite{AAK}) of a smooth hypersurface $\mathrm{\Sigma}$ in $(\bC^*)^n$.   For example, if $M$ is a Fano surface defined by a reflexive polygon with $n$ vertices then $(K_M, W)$ is the generalized SYZ mirror of an $n$-punctured torus $\mathrm{\Sigma}$ (such as the thrice punctured torus $\mathrm{\Sigma}=\{1+x+y+t/xy=0\}\subset (\bC^*)^2$, whose mirror is the LG model $(K_{\bP^2}, W=z_1z_2z_3p)$).

More generally, the canonical bundle $K_M$ is a special case of a toric Calabi-Yau manifold $Y =\bC^{d+1}/\!\!/\bT^{d-n}$ of dimension $(n+1)$. The function $z_1\cdots z_d z_{d+1}$  on $\bC^{d+1}$ descends to a well-defined holomorphic function 
$W: Y\to \bC$. It was first proposed by Hori-Vafa \cite{HV00} and then proven in the SYZ framework by Abouzaid-Auroux-Katzarkov \cite{AAK} that LG-models given by $(Y,W)$ for $Y$ a toric CY manifold are generalized SYZ mirrors to hypersurfaces $\Sigma$ in toric varieties. The HMS prediction would be the equivalence of the following 
triangulated categories: 
\begin{equation}\label{eq: W(Y,W)}
\mathcal{W}(Y, W) \cong D^b Coh(\mathrm{\Sigma}) 
\end{equation}
\begin{equation}\label{eqn:MF}
MF(Y,W) \cong \mathcal{W}(\mathrm{\Sigma}),
\end{equation}
where $\mathcal{W}(\mathrm{\Sigma})$ is the wrapped Fukaya category of $\mathrm{\Sigma}$, wrapped due to $\mathrm{\Sigma}$ being a noncompact Liouville manifold. To define the fiber-wise wrapped Fukaya category $\mathcal{W}(Y,W)$ one would need to view $W:Y \to \bC$ as a symplectic fibration, which requires a good understanding of the symplectic structure on $Y$.  Equivalence \eqref{eq: W(Y,W)} is the subject of the work in preparation by Abouzaid-Auroux \cite{AAhypersurf} when $\mathrm{\Sigma}$ is an algebraic hypersurface in $(\bC^*)^n$ and $Y$ is a toric Calabi-Yau $(n+1)$-fold. In the other direction, the third author \cite{heather} proved the equivalence \eqref{eqn:MF} when $n=2$ where  $\mathrm{\Sigma}\subset (\bC^*)^2$ is a punctured Riemann surface via decomposition into pair-of-pants (thrice punctured spheres) and applying the result of \cite{AAEKO}, which establishes \eqref{eqn:MF} when $\mathrm{\Sigma}$ is a punctured sphere.  Lekili-Polishchuk \cite{LPpants} proved the equivalence \eqref{eqn:MF} when $\mathrm{\Sigma}\subset (\bC^*)^n$ is a generalized higher dimensional pair-of-pants. A version of the equivalence in Equation \eqref{eqn:MF} is proven from the microlocal perspective \cite{gam18} using localization results from \cite{GPS20}. 


Further generalizations are given in \cite[Section 10]{AAK}, such as to the SYZ mirrors for hypersurfaces $\Sigma$ of abelian varieties per the speculation of Seidel \cite{seidel_gen2_specul}. In \cite{Ca}, the second author proved a HMS result for genus-2 compact Riemann surfaces $\Sigma_2$ that are hypersurfaces in an abelian variety $V=(\bC^*)^2/\mathrm{\Gamma}_B$ ($\mathrm{\Gamma}_B\cong\bZ^2$) and its generalized SYZ mirror $(Y,v_0)$.  Here, she considers a LG model $(Y,v_0)$ where $Y=\widetilde{Y}/\mathrm{\Gamma}_B$ is the quotient of a toric Calabi-Yau 3-fold $\widetilde{Y}$ of infinite type by the free action of $\mathrm{\Gamma}_B$. The main result in \cite{Ca} is a fully-faithful embedding
\begin{equation}\label{eq:Ca}
D^b Coh(\mathrm{\Sigma}_2)\hookrightarrow H^0 \mathcal{FS}(Y,v_0).
\end{equation} 
where $\mathcal{FS}$ denotes the Fukaya-Seidel category with compact fibers that are abelian varieties degenerating to the following critical locus; the non-compact Calabi-Yau 3-fold $Y$ contains a ``banana'' configuration of three 2-spheres $C_1\cup C_2 \cup C_3$ that intersect at two triple intersection points. The  K\"{a}hler moduli of 
$Y$ is 3-dimensional and the real K\"{a}hler parameters are given by the symplectic areas $A_i$ of $C_i$. The complex moduli $\cM_2$ of
$\mathrm{\Sigma}_2$ is a complex orbifold of dimension 3.  In \cite{Ca}, the second author considers a one-parameter family of symplectic structures on $Y$ with $A_1=A_2=A_3$, which corresponds to a one-parameter family of complex structures on $\mathrm{\Sigma}_2$. 

In \cite{ACLL}, we extend the work \cite{Ca} to any genus 2 curve in $\cM_2$.  To do that, we need to construct more general symplectic structures on $Y$ where the areas of the three 2-spheres may vary independently, which inspired us to do the exercises that are in this paper. This is naturally done in the action-angle coordinates using Guillemin's K\"ahler potential. The canonical bundle $K_M$ is a local model for $\widetilde Y$, and in the example below, we express the superpotential also in the angle-action coordinates.

\begin{example}[The superpotential $W:K_M\to \bC$ in action-angle coordinates]\label{ex: superpotential_xi} Let $s(z_1,\ldots,z_d)=z_1 \cdots z_d$ and $W: K_M\to \bC[z_1,\ldots, z_d,p] \mapsto  p z_1 \cdots z_d$ a morphism between smooth toric varieties. We are interested in the A-model on the Landau-Ginzburg model $(K_M, W)$, so we would like to view $W$ as a symplectic fibration. We write it in terms of the action-angle coordinates $(\xi,\theta)$ instead of the complex homogeneous coordinates $(z_1,\ldots,z_d,p)$. In particular, since the superpotential is defined on the symplectic quotient $K_M$, it is independent of the choice of representative in the homogeneous coordinates. We could write it in terms of the affine coordinates corresponding to the $J^{th}$ chart where the complement of the $J^{th}$ $z_i$'s are scaled to 1, so it would still be the product of all the inhomogeneous coordinates, or we could write it in terms of the coordinates on the $\bT^n_\bC$ which parametrizes all those charts via choices of $\alpha_\bC$. In that case, it would just be projection onto the last coordinate, $t_{n+1}$.

More specifically, let $v: {\mathrm{\Delta}}^+ \times \bT^{n+1}\to \bC$ be the composition 
\begin{equation}
v:{\mathrm{\Delta}}^+ \times \bT^{n+1} \stackrel{\cong}{\longrightarrow}  \bT^{n+1}_\bC \xrightarrow{\alpha_\bC} \bT^{d+1}_\bC \xhookrightarrow{} U_a \times \bC   \xrightarrow{\widetilde{\pi}_a^+} K_M \stackrel{W}{\longrightarrow} \bC.
\end{equation}
where $\widetilde{\pi}^+_a: U_a \times \bC \to K_M$ is the projection map defined in Equation \eqref{eq:projs_KM}. Namely we compose the exponential map 
$$(\xi_1,\ldots, \xi_{n+1}, e^{i\theta_1},\cdots, e^{i \theta_{n+1}}) \mapsto (e^{\xi_1+i\theta_1},\ldots, e^{\xi_n + i \theta_{n+1}})$$ 
with $\alpha_\bC(t_1,\ldots,t_{n+1})$ which, up to permutation is $(y_1^J,\ldots, y_{n+1}^J, 1, \ldots, 1)$, and then multiply everything together to obtain $\prod_{k=1}^{n+1} y_k^J$. Now recall Definition \ref{def:alpha-p}, which defines the $\alpha_\bC$ injection for $K_M$, namely what the $y_k^J$ are as functions of $t_k$. Plugging in for $t_k$, we see that 
\begin{equation}\label{eq: v in action angle coordinate}
\begin{aligned}
v(\xi_1,\ldots, \xi_{n+1}, e^{i\theta_1},\cdots, e^{i \theta_{n+1}}) & = z_1\ldots z_dp = y_1^J \ldots y^J_{n+1}\\
& =t_{n+1}=  \exp\left(\frac{\partial G^+}{\partial \xi_{n+1}}(\xi)+i \theta_{n+1}\right).
\end{aligned}
\end{equation}
Namely the superpotential is projection onto the last coordinate in the parameters $(t_1,\ldots,t_{n+1})$ because that equals the product of the coordinates in $\alpha_\bC(t_1,\ldots,t_{n+1})$.
\qed
\end{example} 


\subsection{Monodromy in mirror symmetry}\label{subsec:monodromy}
Landau-Ginzburg models, of which $(K_M,W)$ is a prototype example, are important in physics and math, in particular in mirror symmetry. In particular, moment maps are an example of the more general notion of a \emph{Lagrangian torus fibration} which is the input needed to obtain a mirror when using the prescription given by SYZ mirror symmetry \cite{syz}. Thus in the case of $K_M$ there are two interesting fibrations, one is the Lagrangian torus fibration coming from the moment map we've described, and the other is $(K_M, W)$ which is a singular symplectic fibration with e.g.~the symplectic form coming from symplectic reduction. For manifolds which are Fano or of general type ($c_1<0$), their mirrors are Landau-Ginzburg models \cite{hori_vafa}, \cite{AAK}. So it is crucial to be able to understand the Fukaya category of a Landau-Ginzburg model, which is the algebraic input for the symplectic side of homological mirror symmetry. To do so, one may start by looking at the Floer theory of $(K_M,W)$ for noncompact Lagrangians which interact well with $W:K_M \to \bC$. Two types of non-compact Lagrangians are often used. 

One type of non-compact Lagrangian considered consists of those in \cite{seidel} which are \emph{thimbles} obtained by parallel transporting a Lagrangian in a fiber of $W$ to the singular locus over 0; this works when $W$ is a Lefschetz fibration. However, for more complicated singular fibers, this process could produce a singular Lagrangian if the fiber Lagrangian degenerates to a submanifold which is of codimension more than 1 from what it was in the generic fiber. In this case, Lagrangians which are \emph{U-shaped} may be used; they go around the critical value(s) of $W$, where often that value is 0. Therefore, it is of interest to know the monodromy around the origin of $W:K_M \to \bC$. This is what brought us to do the work in this paper, as we are interested in the Fukaya category of a certain LG-model in forthcoming work \cite{ACLL}.


Lastly, the theory of the symplectic quotient extends not only to noncompact manifolds, but to toric varieties of \emph{infinite type}. Specifically, one example is the following: let $Y$ be the example mentioned in the paragraph surrounding Equation \eqref{eq:Ca}. This is denoted $X_{(1,1)}$ in the SYZ mirror symmetry paper by \cite{KL}. It can be equipped with the following symplectic form, which is a scaled version of that from symplectic reduction, to account for the infinitely many facets.

\begin{definition}[Definition of $\omega$ from \cite{KL}]\label{def:omega_ACLL} 
There exists an open covering $\{U_v\}$ of the moment polyhedron ${\mathrm{\Delta}}_{\tilde{Y}}$, and non-negative bump functions $\rho_v$ on $\bR^3$ which are supported on $U_v$ and identically 1 in a smaller neighborhood of the boundary stratum, such that the following
\begin{equation}
    \tilde{G}(y) := \frac{1}{2} \sum_{v \in \mathrm{\Sigma}(1)}\rho_v(y)L_v(y) \log L_v(y)
\end{equation}
is a finite sum at each point $y \in \mathring{{\mathrm{\Delta}}}_{\tilde{Y}}$, and whose Legendre transform is $\mathrm{\Gamma}_B$ equivariant and defines a K\"ahler potential for $Y$, (namely its second-order derivatives are positive). Here, as earlier, $L_v\geq 0$ denotes the half spaces whose intersection defines ${\mathrm{\Delta}}_{\tilde{Y}}$.
\end{definition}

Computing monodromy in this case requires a bit more finesse and the computation is done by splitting into three regions: one is around a vertex of the polytope which is modeled on an open neighborhood around $0 \in \bC^3$, one is in a neighborhood of the center of a hexagon modeled on tot$\left(K_{\mathbb{CP}^2(3)}\right)$, and the third is the remaining region between these two which is harder to compute directly but can be estimated and computed up to Hamiltonian isotopy. 

In conclusion, toric varieties provide a rich source of symplectic manifolds, including compact, noncompact, and of infinite type, and this article describes how to compute information about them of interest to symplectic geometers.

\section{Notation}\label{sec:notation}

This is an index for notation appearing throughout the text, with some comparisons to notation used in \cite{G94}.

\subsubsection*{Section \ref{sec: symp_quot}: Toric manifolds as symplectic quotients}
\begin{itemize}[\textbullet]
\item $X$ in Guillemin \cite{G94} is called $M$ or $K_M$ here
\item $\cdot$ denotes a group action as specified in the context
\item $\mathbb I_k$ denotes the $k\times k$ identity matrix
\item $\bT^{k}_\bC$ is the complex algebraic torus $(\bC^*)^k$ of dimension $k$
\item $\bT^{k}$ is the real torus of dimension $k$ obtained as the maximal compact subgroup $U(1)^k$ of $\bT^k_\bC$ by restricting to norm 1 coordinates
\item $n=\dim_\bC M$, $n+1 = \dim_\bC K_M$ 
\item Maps involved in symplectic reduction:
\begin{enumerate}
\item $\rho_\bC: N_\bC \to \bT^d_\bC$ is an embedding and $Q=(Q^l_k)_{k,l}$ is the matrix representing its linearization 
\item $\beta_\bC: \bT^d_\bC \to \bT^d_\bC/N_\bC \cong \bT^n_\bC$ is the quotient map and $B=(v^k_m)_{m,k}$ is the matrix representing its linearization
\item superscripts denote the column, subscripts denote the row
\end{enumerate}
\item $\mu_{\bT^d}$ is the moment map for the standard Hamiltonian $\bT^d$-action on $\bC^d$
\item $\mu_N$ is the moment map for the Hamiltonian $N$-action on $\bC^d$
\item $U_a \subset \bC^d$ is the open set where $N_\bC$ acts freely and which contains $\mu_N^{-1}(a)$
\item $U_a^+ = U_a \times \bC$
\item $R_a: U_a \to \mu_N^{-1}(a)$ retracts $U_a$ onto $\mu_N^{-1}(a)$ via a choice of $\la_a(z) \in N_\bC$ for each $z \in U_a$
\item symbols with a $+$ refer to the $K_M$ case
\item $Z$ in \cite{G94} is a level set, here called $\mu_N^{-1}(a)$ for $M$ or ${\mu^+_{N}}^{-1}(a)$ for $K_M$
\item $z_1,\ldots,z_d$ are the homogeneous coordinates on $\bC^d$ (with additional coordinate $p$ in the case of $K_M$)
\item $r_1,\ldots, r_d, \varphi_1,\ldots, \varphi_d$ are the polar coordinates of the homogeneous coordinates
\item $\mathrm{{\mathrm{\Delta}}}$ denotes the moment polytope (when bounded, for compact $M$) and ${\mathrm{\Delta}}^+$ denotes the moment polyhedron (when unbounded, for noncompact $K_M$)
    \item Projection maps:
    \begin{equation*}
\begin{cases}
\pi_a: (\mu_N)^{-1}(a) \rightarrow M =  (\mu_N)^{-1}(a)/N, \\
\widetilde \pi_a: {U_a}   \rightarrow M = U_a/N_\bC,
\end{cases}
\end{equation*}
defined in Theorems \ref{thm: MW quotient} and \ref{thm: GIT}, and
    \begin{equation*}
\begin{cases}
\pi^+_a: (\mu^+_N)^{-1}(a) \rightarrow K_M =  (\mu^+_N)^{-1}(a)/N, \\
\widetilde \pi^+_a: {U_a} \times \bC \rightarrow K_M = ({U_a}\times \bC)/N_\bC,
\end{cases}
\end{equation*}
defined in Equation \eqref{eq:projs_KM}
\end{itemize}

\subsubsection*{Section \ref{sec: moment map}: Toric actions and moment maps}

\begin{itemize}[\textbullet]
\item $J=(j_1,\ldots,j_r)$ denotes a multi-index where $j_l \in \{1,\ldots,d\}$ are strictly increasing (and will index a choice of $r$ facets)
\item $\bC^d_J=\mathrm{orb}_{\bT^d_\bC}(z)\cong \bC^{d-r}$ denotes the orbit under the standard $\bT^d_\bC$-action of any point $(z_1,\ldots,z_d)$ with $z_j=0 \iff j \in J$
\item $(\bT^d_\bC)_J=\mathrm{stab}_{\bT^d_\bC}(z)\cong (\bC^*)^r$, for any fixed choice $z \in \bC^d_J$ 
\item $\alpha_\bC: \bT^n_\bC \to \bT^d_\bC$ is a right inverse to $\beta_\bC$ and $A=(s^m_k)_{k,m}$ is the matrix representing its linearization
\item $\mu_a$ is the moment map of the Hamiltonian $\bT^n$-action on $\mu^{-1}(a)/N$
\item $(d\rho_\bC)_1^*(\kappa)=-a$
\item $l_i$ in \cite{G94} is called $L_i$ here, which denotes the affine linear defining equation of the $i^{\textrm{th}}$ facet of polyhedron ${\mathrm{\Delta}}$
\item $\cF_j$ denotes the $j^{\textrm{th}}$ facet of the moment polytope
\item $J_f=(j_1,\ldots,j_r)$ indexes the facets which intersect with polytope $\mathrm{\Delta}$ in a face $f$, e.g.~$r=n$ when $f=v$ a vertex
\item $t_1,\ldots,t_n$ are the inhomogeneous coordinates on the dense $(\bC^*)^n=\bT^n_\bC$
\item $\mathring{\mathrm{\Delta}}$ denotes the interior of the polytope
\item $u_j= \log t_j$ are coordinates on the Lie algebra Lie $\bT^n_\bC=\mathfrak{t}^n_\bC$
\item $x_1,\ldots, x_n$ and $\theta_1,\ldots, \theta_n$ are the polar coordinates of the inhomogeneous coordinates $(t_1,\ldots,t_n) \in (\bC^*)^n$
\item $\xi_1,\ldots, \xi_n$ are the moment map coordinates 
\item $\xi_i,\theta_i$ are the action-angle coordinates
\item Notation for Subsection \ref{sec: holom chart}: Holomorphic coordinate charts for $M$
\begin{itemize} 
\item $v$ is a vertex of the polytope $\mathrm{\Delta}$
\item For vertex $v$ and corresponding indexing set $J_v$, 
$$\widetilde V_{J_v}=\{ (z_1, \ldots, z_d) \in \bC^d \mid z_j\neq 0 \text{ if } j\notin J_v\}\cong (\bC^*)^{d-n} \times \bC^{n} \subset U_a$$
\item $V_{J_v}=\widetilde V_{J_v}/N_\bC\cong \bC^n \subset M$ is the open chart of $M$ corresponding to vertex $v$, the union of which cover $M$ 
\item $y_1^{J_v},\ldots,y_n^{J_v}$ (defined in Equation \eqref{eq: V in y coordinates}) denote affine coordinates on the chart $V_{J_v}$
\item $\varphi^{J_v}: V_{J_v} \to \bC^n$ is the chart map defined by $(y_1^{J_v},\ldots,y_n^{J_v})$
\item $p^{J_v}=\mu_a^{-1}(v)$ is a $\bT^n_\bC$ toric fixed point corresponding to vertex $v$, and it's the center of the chart $\varphi^{J_v}$
\item $U_{J_v} = \{(z_1,\ldots, z_d) \in \bC^d \mid  z_j=1 \text { if } j\notin J_v\}\subset \widetilde V_{J_v}\subset \bC^d$ is the slice of the $N_\bC$ action given by the representatives in Equation \eqref{eq: V in y coordinates}, where the letter $U$ is used because it is a subset of $U_a$
\item $\widehat{B}$, $\widehat{A}^{J_v}$ denote choices of bases for matrices $B, A$ so that the corresponding ${\alpha}_\bC^{J_v}$ gives the standard $\bT^n_\bC$ action, as described in the paragraph ``Change of basis of $\bT^n_\bC$ and the embedding of $\bT^n_\bC$ in $\bC^n$-charts" of Subsection \ref{sec: change Tn basis}
\end{itemize}
\end{itemize}

\subsubsection*{Section \ref{sec: kahler}: K\"ahler potential}

\begin{itemize}[\textbullet]
\item $F$ is the K\"ahler potential and is a function of the $x_j$
\item $G$ is the Legendre transform of $F$ and is a function of the $\xi_j$
\item the $j^{\textrm{th}}$ coordinate function of $\mu_a(t_1,\ldots,t_n)$ is $\xi_j$ and equals $\partial F/\partial x_j$

\end{itemize}

\subsubsection*{Section \ref{sec: HMS}: Connection to mirror symmetry}

\begin{itemize}[\textbullet]
\item $W:K_M \to \bC$ denotes the superpotential
\item $s \in \mathrm{\Gamma}(M, K_M^*)$ is a generic section
\item $X= \mathrm{Crit}(W)$ is the critical locus of the superpotential
\item $\mathrm{\Sigma} \subset (\bC^*)^n$ is a hypersurface whose generalized SYZ mirror is a LG model $(K_M, W)$
\item $\mathrm{\Sigma}_2$ is the genus 2 curve 
\item $(Y,v_0)$ is the SYZ mirror of $(\mathrm{Bl}_{\mathrm{\Sigma} \times \{0\}} V \times \bC, y)$, both LG models
\item $\widetilde{Y}$ is the universal cover of $Y$, a toric variety of infinite type
\item $v$ is $W$ written as a function of the moment map coordinates $\xi_j$
\end{itemize}

\bibliographystyle{amsalpha}
\bibliography{glob_gen2_hms}

\end{document}